\documentclass[11pt,reqno]{amsart}

\usepackage[a4paper,margin=1in]{geometry}
\usepackage{amsmath,amssymb,amsfonts,amsthm,mathtools}
\usepackage{mathrsfs}
\usepackage[expansion=false]{microtype}
\usepackage{bm}
\usepackage{enumitem}
\usepackage[nocompress]{cite}
\usepackage[
  colorlinks=true,
  linkcolor=blue,
  citecolor=blue,
  urlcolor=blue,
  hypertexnames=false
]{hyperref}
\usepackage[capitalize,nameinlink]{cleveref}
\usepackage{aliascnt}
\usepackage{placeins}
\usepackage{graphicx}
\usepackage{tikz}
\usepackage{float}
\usepackage[mathlines,right]{lineno}
\usetikzlibrary{arrows.meta,calc,positioning}

\newif\ifshowlinenumbers
\showlinenumbersfalse
\ifshowlinenumbers
  \modulolinenumbers[1]
\fi

\newtheorem{theorem}{Theorem}[section]

\newaliascnt{proposition}{theorem}
\newtheorem{proposition}[proposition]{Proposition}
\aliascntresetthe{proposition}

\newaliascnt{lemma}{theorem}
\newtheorem{lemma}[lemma]{Lemma}
\aliascntresetthe{lemma}

\newaliascnt{corollary}{theorem}
\newtheorem{corollary}[corollary]{Corollary}
\aliascntresetthe{corollary}

\theoremstyle{definition}

\newaliascnt{definition}{theorem}
\newtheorem{definition}[definition]{Definition}
\aliascntresetthe{definition}

\newaliascnt{assumption}{theorem}

\aliascntresetthe{assumption}

\theoremstyle{remark}

\newaliascnt{remark}{theorem}
\newtheorem{remark}[remark]{Remark}
\aliascntresetthe{remark}

\crefname{theorem}{Theorem}{Theorems}
\Crefname{theorem}{Theorem}{Theorems}
\crefname{proposition}{Proposition}{Propositions}
\Crefname{proposition}{Proposition}{Propositions}
\crefname{lemma}{Lemma}{Lemmas}
\Crefname{lemma}{Lemma}{Lemmas}
\crefname{corollary}{Corollary}{Corollaries}
\Crefname{corollary}{Corollary}{Corollaries}
\crefname{definition}{Definition}{Definitions}
\Crefname{definition}{Definition}{Definitions}
\crefname{assumption}{Assumption}{Assumptions}
\Crefname{assumption}{Assumption}{Assumptions}
\crefname{remark}{Remark}{Remarks}
\Crefname{remark}{Remark}{Remarks}

\numberwithin{equation}{section}
\allowdisplaybreaks

\title[Two-phase harmonic map heat flow with a prescribed interface]
{Short-time existence and uniqueness for two-phase harmonic map heat flow with a prescribed moving interface}
\author{Xingyu Wang}
\address{School of Mathematical Sciences, Shanghai Jiao Tong University, Shanghai, China}
\email{xingyuwang@sjtu.edu.cn}
\date{}
\subjclass[2020]{35K51, 35R37, 58E20}
\keywords{harmonic map heat flow; prescribed interface; target free-boundary condition; matrix-valued Allen--Cahn equation}

\hypersetup{
  pdftitle={Short-time existence and uniqueness for two-phase harmonic map heat flow with a prescribed moving interface},
  pdfauthor={Xingyu Wang},
  pdfsubject={Two-phase harmonic map heat flow with a prescribed moving interface},
  pdfkeywords={harmonic map heat flow, prescribed interface, target free-boundary condition, matrix-valued Allen--Cahn equation}
}

\begin{document}

\ifshowlinenumbers
  \linenumbers
\fi

\begin{abstract}
Let $N^+$ and $N^-$ be disjoint compact smooth submanifolds of a Euclidean
space, and let $\Lambda\subset N^+\times N^-$ be a compact smooth embedded
submanifold prescribing admissible pairs of one-sided traces.  Given a smooth family of separating hypersurfaces
$\Gamma_t\subset\mathbb T^d$, we study harmonic map heat flows into $N^\pm$
in the two moving phases, subject to
\[
    (u^+,u^-)\in\Lambda,
    \qquad
    \bigl(-\partial_{\nu_t}u^+,\partial_{\nu_t}u^-\bigr)
    \perp T_{(u^+,u^-)}\Lambda
    \quad\text{on }\Gamma_t.
\]
For $q>d+2$, initial data in $W_q^{2-2/q}$ satisfying these conditions at
$t=0$ generate a unique short-time strong solution in $W_q^{2,1}$, which is
smooth up to the moving interface for every positive time; smooth initial
data satisfying the compatibility conditions of every order yield solutions
that are smooth up to the initial corner.  We double the paired map
across $\Lambda$ by normal reflection, which turns the interface conditions
into a Dirichlet problem near the interface, and match this problem with the two
bulk phases by an overlapping space--time Schwarz map inside a Schauder
fixed-point argument.  When the interface evolves independently by mean
curvature, we also obtain a blow-up alternative in terms of its curvature and
the two phase gradients.  As an application, we prove local smooth
solvability for the periodic matrix-valued sharp-interface system appearing in
the convergence theory of Fei, Lin, Wang, and Zhang
(\emph{Invent. Math.} 233 (2023), 1--80).
\end{abstract}

\maketitle

\section{Introduction}
\label{sec:introduction}

The Keller--Rubinstein--Sternberg problem concerns reaction--diffusion systems
\[
    \partial_tu^\varepsilon=\Delta u^\varepsilon-\varepsilon^{-2}\nabla F(u^\varepsilon)
\]
whose potential $F\ge0$ vanishes on several positive-dimensional wells.  By a
formal asymptotic expansion, Rubinstein, Sternberg, and Keller
\cite{RubinsteinSternbergKeller1989Curve,RubinsteinSternbergKeller1989HarmonicMaps}
found that, as $\varepsilon\to0$, the interface between the wells moves by mean
curvature, while away from the interface $u^\varepsilon$ tends to a harmonic map
heat flow into the wells.  In the two-well case the wells are disjoint compact
smooth submanifolds $N^\pm\subset\mathbb R^L$ with $F^{-1}(0)=N^+\cup N^-$, and the
limit is expected to be a sharp-interface system.  We write it in the abstract
form: a moving hypersurface $\Gamma_t$ separates phase regions $\Omega_t^\pm$, and
\begin{equation}\label{eq:intro-sharp-interface-system}
\begin{cases}
    \partial_tu^\pm=P_{T_{u^\pm}N^\pm}\Delta u^\pm
    &\text{in }\Omega_t^\pm,\\[0.3em]
    (u^+,u^-)\in\Lambda,\quad
    (-\partial_{\nu_t}u^+,\partial_{\nu_t}u^-)\perp T_{(u^+,u^-)}\Lambda
    &\text{on }\Gamma_t.
\end{cases}
\end{equation}
Here $P_{T_zN^\pm}$ is the orthogonal projection onto $T_zN^\pm$, $\nu_t$ is the
unit normal of $\Gamma_t$ pointing from $\Omega_t^-$ into $\Omega_t^+$, and
$\Lambda\subset N^+\times N^-$ is a submanifold of admissible endpoint pairs.  The
two phases are coupled only through the interface conditions in
\eqref{eq:intro-sharp-interface-system}.  In the models of
\cite{BronsardStoth1998,DongWang2025,FeiLinWangZhang2023,FeiWangZhangZhang2018}
discussed below, the interface moves by mean curvature (in
\cite{FeiWangZhangZhang2018}, up to a constant factor),
\begin{equation}\label{eq:intro-mean-curvature-flow}
    \mathcal V_t=\mathcal H_t,
    \qquad
    \mathcal H_t:=-\operatorname{div}_{\Gamma_t}\nu_t,
\end{equation}
where $\mathcal V_t$ is the normal velocity of $\Gamma_t$ in the direction
$\nu_t$.  This law does not involve the phase fields, so the analysis splits
into a classical geometric flow and a bulk problem on a prescribed moving domain.

The interface conditions have a variational origin.  For minimizers of the
static energy, Lin--Pan--Wang \cite{LinPanWang2012} showed that the two
sharp-interface traces are the endpoints of a minimal connecting orbit; this
endpoint relation is the trace constraint $(u^+,u^-)\in\Lambda$ with
$\Lambda=\Lambda_F$, the set of minimal pairs
(\Cref{subsec:minimal-pair-model-geometries}).  The conormal condition is the
natural boundary condition for the sum of the two Dirichlet energies under this
constraint; when the trace relation has the form $u^-=\Phi^+(u^+)$ for a
diffeomorphism $\Phi^+$ between submanifolds $M^\pm\subset N^\pm$, it is the
Euler--Lagrange condition derived by Lin--Wang for minimizing maps
\cite[Section~2.1]{LinWang2019}.  A rigorous derivation of the dynamical limit
for general wells remains open.  Bronsard--Stoth \cite{BronsardStoth1998}
proved, for radially symmetric solutions with two concentric circular wells,
that the limit phases are coupled across the interface by a jump condition in
the gradient, which the formal expansion of
\cite{RubinsteinSternbergKeller1989Curve,RubinsteinSternbergKeller1989HarmonicMaps}
had not revealed; Dong--Wang \cite{DongWang2025} extended this to concentric
spherical wells in every dimension without radial symmetry, using matched
asymptotic expansions.  For the matrix-valued system, Fei--Lin--Wang--Zhang
derived the normal-derivative matching condition by a matched asymptotic
expansion \cite[Appendix~A.1]{FeiLinWangZhang2023}, and the isotropic--nematic
transition with isotropic elasticity was treated in
\cite{FeiWangZhangZhang2018}.  In all these models the
interface conditions obtained are those of
\eqref{eq:intro-sharp-interface-system} for the corresponding $\Lambda$, and,
except in the radially symmetric case of \cite{BronsardStoth1998}, the
convergence proofs assume a smooth solution of the limiting system.  For general wells, Liu \cite{Liu2025} showed, assuming
that the interface moves by a smooth mean curvature flow, that the limit
consists of weak harmonic map heat flows in the two phases satisfying the
minimal pair condition on the interface.

In their convergence analysis of a matrix-valued Allen--Cahn equation,
Fei--Lin--Wang--Zhang construct approximate solutions from a smooth solution of
the limiting sharp-interface system \cite[Theorem~1.1]{FeiLinWangZhang2023}.  The
wells of that system are $\mathrm O^+(n)$ and $\mathrm O^-(n)$, and the admissible
endpoint pairs form the manifold $\Lambda_{\mathrm{mat}}$ of pairs $(P,PR_\xi)$,
where $P\in\mathrm O^+(n)$ and $R_\xi$ is the reflection across $\xi^\perp$ for a
unit vector $\xi$ (\Cref{subsec:matrix-valued-application}).  Since
$\Lambda_{\mathrm{mat}}$ is diffeomorphic to $\mathrm O^+(n)\times\mathbb{RP}^{n-1}$,
both endpoint projections have fibers diffeomorphic to $\mathbb{RP}^{n-1}$, so
for every $n\ge2$ neither trace determines the other.  For $n\ge3$, moreover,
$\Lambda_{\mathrm{mat}}$ is a proper subset of $\mathrm O^+(n)\times\mathrm O^-(n)$,
that is, the matrix potential is only partially minimally paired
\cite[Definition~3.5 and Lemmas~3.6--3.7]{FeiLinWangZhang2023}.

In \cite[Remark~1.2]{FeiLinWangZhang2023}, the existence of smooth solutions of
the limiting system is described as a nontrivial issue in general.  For the
matrix-valued system, weak solutions with a
prescribed interface were constructed in \cite{WangWangZhang2026}, and weak
solutions were obtained through the sharp-interface limit in
\cite{Wang2026SharpInterface}.  For smooth solutions, the local existence
theorem \cite[Theorem~3.2]{LinWang2019} concerns trace
relations $u^-=\Phi^+(u^+)$ given by a diffeomorphism $\Phi^+$ between
submanifolds $M^\pm\subset N^\pm$.  Since $\Lambda_{\mathrm{mat}}$ is not of
this form, the matrix-valued system does not seem to fall directly within that
framework; here
we give a self-contained proof for general admissible trace manifolds
$\Lambda$.

This paper treats the bulk problem, in which $\Gamma_t$ is an arbitrary
prescribed smooth family of separating hypersurfaces, not necessarily a mean
curvature flow; its precise formulation is
\eqref{eq:original-moving-interface-problem} in \Cref{sec:main-results}.  We
study it for two reasons.  First, by the splitting above, it is the analytic
core of the local existence problem for
\eqref{eq:intro-sharp-interface-system}--\eqref{eq:intro-mean-curvature-flow};
in particular, it gives local smooth solutions of the kind assumed in these
convergence results.  Second, the problem has its own analytic
difficulties.  The interface conditions couple a nonlinear constraint on the pair
of one-sided traces with a conormal condition, so neither phase can be solved
with Dirichlet or Neumann data alone.  Normal reflection across $\Lambda$ in the
target encodes both interface conditions as a single reflection symmetry,
without parametrizing one trace by the other, and turns them into a Dirichlet
problem; but the resulting doubled collar is attached to two different bulk
domains, so the single-phase target free-boundary method does not apply directly
(\Cref{rem:single-phase-comparison}).

\subsection*{Summary of results}
For any nonempty compact smooth submanifold $\Lambda\subset N^+\times N^-$ and
$q>d+2$, initial data in $W_q^{2-2/q}$ satisfying the two interface conditions
in \eqref{eq:original-moving-interface-problem} on $\Gamma_0$ generate a strong
solution in $W_q^{2,1}$, unique in the full class of such solutions, which is
smooth up to the moving interface for every positive time
(\Cref{thm:main-short-time-solution}).  Under compatibility conditions of
every order, smooth data yield solutions that are smooth up to the initial
corner $\Gamma_0\times\{0\}$ (\Cref{cor:smooth-short-time-solution}).
Since $W_q^{2-2/q}$ is the natural initial trace space for maximal
$L^q$-regularity, the existence theorem can be restarted from the terminal
trace of a solution.  Hence, when the prescribed interface extends smoothly
past a finite time $T_*$, bounded gradients of the phase fields imply that the
solution continues (\Cref{prop:prescribed-interface-continuation}).  For an
interface moving by mean curvature this gives a blow-up alternative: at a finite
maximal time, either the curvature of $\Gamma_t$ or one of the phase gradients
becomes unbounded (\Cref{cor:mean-curvature-continuation}).  Applied to
$\Lambda_{\mathrm{mat}}$, the main theorem gives local smooth solutions of the
matrix-valued system for any prescribed smooth separating interface family and
smooth initial data satisfying the compatibility conditions
(\Cref{cor:matrix-valued-short-time-solution}); taking for $\Gamma_t$ the mean
curvature flow starting from $\Gamma_0$ on its short existence interval gives
local smooth solutions of the full system
\eqref{eq:intro-sharp-interface-system}--\eqref{eq:intro-mean-curvature-flow}.  The same argument covers the
isotropic--nematic and concentric-sphere geometries
(\Cref{subsec:minimal-pair-model-geometries}).

\subsection*{Ideas of the proof}
Since in general neither phase trace determines the other, we treat the pair
$\mathbf u=(u^+,u^-)$ as a single unknown near the interface.  The key ingredient
is the combination of a normal reflection across $\Lambda$ in the target with an
overlapping space--time Schwarz map.  The argument has three parts.
\begin{enumerate}[label=\emph{(\arabic*)},leftmargin=2.3em]
\item \emph{Doubling across $\Lambda$.}  After an ambient isotopy adapted to
      Fermi coordinates, we fold the two one-sided collars onto a common
      half-collar and extend the paired map by the normal reflection
      $\mathcal R_\Lambda=2\Pi_\Lambda-\mathrm{Id}$ in the target
      ($\Pi_\Lambda$ is the nearest-point projection onto $\Lambda$),
      following the target free-boundary method of Struwe and Chen--Lin
      \cite{Struwe1991,ChenLin1998}.  The interface conditions then say that
      the extension is reflection-symmetric, and we are left with a Dirichlet
      problem on a doubled collar (\Cref{lem:reflection-covariance}).  At the
      regularity $W_q^{2-2/q}$, the reflected initial datum must again lie in
      $W_q^{2-2/q}$ across the interface; this is where the first-order
      compatibility condition at $\Gamma_0$ enters.
\item \emph{An overlapping space--time Schwarz map.}  The doubled collar and
      the two bulk domains overlap.  Given boundary data at the outer edge of
      the collar, we solve in the collar and take its trace on an inner
      hypersurface as Dirichlet data for the bulk problems; the bulk traces are
      then passed back to the collar.  The two transfer hypersurfaces are a
      positive distance apart, so the bulk step gains a factor $T^{1/2}$
      (\Cref{lem:off-diagonal-trace}) and the trace map is a contraction for
      short times, as in parabolic domain decomposition
      \cite{Lions1988Schwarz,GanderStuart1998,GiladiKeller2002,Gander2008Schwarz}.
      The nonlinear coupled system is then solved by Schauder's theorem in the
      frozen coefficients, with the admissible traces centered at a parabolic
      lifting of the initial trace.
\item \emph{Returning to the geometric problem.}  Uniqueness for the nonlinear
      doubled problem implies the reflection symmetry, and evaluating this
      symmetry and its normal derivative at the interface gives both interface
      conditions.  A maximum-principle argument for the distance to the target
      shows that the solution takes values in $N^\pm$ for arbitrary compact
      smooth wells (\Cref{prop:target-preservation}).  Uniqueness in the full
      $W_q^{2,1}$ class comes from a moving-domain energy estimate that uses
      the chord--normal estimate
      $|(\mathbf p-\mathbf q)\cdot\boldsymbol\eta|\le C|\mathbf p-\mathbf q|^2|\boldsymbol\eta|$
      for $\mathbf p,\mathbf q\in\Lambda$ and $\boldsymbol\eta$ normal to
      $\Lambda$ at $\mathbf p$
      (\Cref{prop:physical-solution-uniqueness}); the smoothness statements
      follow from mixed Dirichlet--Neumann Schauder estimates in adapted target
      coordinates.
\end{enumerate}

\subsection*{Organization}
\Cref{sec:main-results} states the results and applications, and
\Cref{sec:geometry,sec:fixed-coordinates} develop the geometry of the trace
constraint and the fixed-domain equations.  The Schwarz construction and the
nonlinear fixed-point argument are in \Cref{sec:schwarz,sec:schauder};
\Cref{sec:continuation} treats continuation.  The appendices give the
compatibility hierarchy and uniform short-time parabolic estimates.

\section{Main results and applications}
\label{sec:main-results}

\subsection{The prescribed-interface problem and main theorem}

Fix integers $d\ge2$ and $L\ge1$.  Let $N^+$ and $N^-$ be disjoint compact smooth
embedded submanifolds of $\mathbb R^L$, without boundary.  We identify the
product Euclidean space $\mathbb R^L\times\mathbb R^L$ isometrically with
$\mathbb R^{2L}$ and set
\begin{equation}\label{eq:ambient-product-spaces}
    \mathbf N:=N^+\times N^-\subset\mathbb R^{2L}.
\end{equation}

Let the admissible trace manifold $\Lambda$ be a nonempty compact smooth embedded
submanifold without boundary of $\mathbf N$.  Neither endpoint projection
$\operatorname{pr}_\pm:\Lambda\to N^\pm$ is assumed to be injective or
surjective.

Boldface letters always denote paired quantities:
\[
    \mathbf z=(z^+,z^-),
    \qquad
    \mathbf v=(v^+,v^-),
    \qquad
    \mathbf h=(h^+,h^-).
\]
For $\mathbf p\in\Lambda$, we write
\begin{equation}\label{eq:ambient-relative-normal-spaces}
    T_{\mathbf p}^{\perp,\mathrm{amb}}\Lambda
    :=
    \bigl(T_{\mathbf p}\Lambda\bigr)^{\perp_{\mathbb R^{2L}}},
    \qquad
    T_{\mathbf p}^{\perp,\mathbf N}\Lambda
    :=
    T_{\mathbf p}\mathbf N
    \cap
    T_{\mathbf p}^{\perp,\mathrm{amb}}\Lambda.
\end{equation}

Let $\mathbb T^d:=\mathbb R^d/\mathbb Z^d$ carry its standard flat metric; this
matches the periodic setting of \cite{FeiLinWangZhang2023} and avoids an outer
boundary.  For $0\le t\le T_0$, let
$\Gamma_t\subset\mathbb T^d$ be a prescribed smooth family of compact closed
embedded hypersurfaces.  By a smooth family we mean that there is a smooth map
\[
    X:\Gamma_0\times[0,T_0]\longrightarrow\mathbb T^d
\]
such that each $X_t:=X(\cdot,t)$ is an embedding with image $\Gamma_t$.
We assume that every $\Gamma_t$ separates the torus
into two smooth nonempty domains:
\begin{equation}\label{eq:torus-phase-decomposition}
    \mathbb T^d\setminus\Gamma_t
    =
    \Omega_t^+\mathbin{\dot\cup}\Omega_t^-,
    \qquad
    \partial\Omega_t^+=\partial\Omega_t^-=\Gamma_t.
\end{equation}
Let $\nu_t$ be the unit normal along $\Gamma_t$ pointing from $\Omega_t^-$
into $\Omega_t^+$.  Thus the outward unit normals of $\Omega_t^+$ and
$\Omega_t^-$ are $-\nu_t$ and $\nu_t$, respectively.

For $\sigma\in\{+,-\}$, the unknown is
\[
    u^\sigma:
    \bigcup_{0\le t\le T_0}\Omega_t^\sigma\times\{t\}
    \longrightarrow N^\sigma.
\]
The prescribed-interface problem is
\begin{equation}\label{eq:original-moving-interface-problem}
\begin{cases}
    \partial_tu^\sigma
    =
    P_{T_{u^\sigma}N^\sigma}\Delta u^\sigma,
    & x\in\Omega_t^\sigma,
      \quad 0<t<T_0,
      \quad \sigma=\pm,\\[0.5em]
    (u^+,u^-)(x,t)\in\Lambda,
    & x\in\Gamma_t,
      \quad 0<t<T_0,\\[0.4em]
    \bigl(-\partial_{\nu_t}u^+,
          \partial_{\nu_t}u^-\bigr)(x,t)
    \perp_{\mathbb R^{2L}}
    T_{(u^+,u^-)(x,t)}\Lambda,
    & x\in\Gamma_t,
      \quad 0<t<T_0,\\[0.4em]
    u^\sigma(x,0)=u_0^\sigma(x),
    & x\in\Omega_0^\sigma,
      \quad \sigma=\pm.
\end{cases}
\end{equation}
For
$\Lambda_{\mathrm{mat}}$, the conormal condition is equivalent to the matching of
normal derivatives used in \cite{FeiLinWangZhang2023}
(\Cref{prop:matrix-normal-derivative-matching-conormality}).

Let $\mathrm{II}^\sigma$ denote the second fundamental form of
$N^\sigma\subset\mathbb R^L$, with the sign convention for which the first
equation in \eqref{eq:original-moving-interface-problem} has the equivalent
ambient form
\begin{equation}\label{eq:ambient-harmonic-map-heat-flow}
    \partial_tu^\sigma-\Delta u^\sigma
    =
    \sum_{i=1}^d
    \mathrm{II}^\sigma_{u^\sigma}
    \bigl(\partial_i u^\sigma,\partial_i u^\sigma\bigr),
    \qquad
    x\in\Omega_t^\sigma,
    \quad 0<t<T_0,
    \quad \sigma=\pm.
\end{equation}
Since $\partial_{\nu_t}u^\sigma\in T_{u^\sigma}N^\sigma$, the conormal condition
is equivalently
\[
    \bigl(-\partial_{\nu_t}u^+,
          \partial_{\nu_t}u^-\bigr)
    \in
    T_{(u^+,u^-)}^{\perp,\mathbf N}\Lambda,
    \qquad \mathbf N=N^+\times N^-.
\]

For $0<T\le T_0$ and $\sigma=\pm$, set
\[
    Q_T^\sigma
    :=
    \bigcup_{0<t<T}\bigl(\Omega_t^\sigma\times\{t\}\bigr).
\]
Function spaces on $Q_T^\sigma$ are defined by pulling back to the fixed
cylinder $\Omega_0^\sigma\times(0,T)$ with the ambient isotopy $\Psi_t$ of
\Cref{lem:fermi-compatible-parametrization}: for a function space $X$ on
$\Omega_0^\sigma\times(0,T)$, such as $W_q^{2,1}$ or
$C^{1+\alpha,(1+\alpha)/2}$, we write $u\in X(Q_T^\sigma)$ if
$(y,t)\mapsto u(\Psi_t(y),t)$ belongs to $X$.  Likewise,
$u\in C([0,T];W_q^{s}(\Omega_t^\sigma))$ means that
$t\mapsto u(\Psi_t(\cdot),t)$ belongs to
$C([0,T];W_q^{s}(\Omega_0^\sigma))$.  Any two smooth isotopies $\Psi_t$ with
$\Psi_t(\Omega_0^\sigma)=\Omega_t^\sigma$ differ by a smooth time-dependent
diffeomorphism of $\Omega_0^\sigma$, which preserves these spaces; hence the
spaces do not depend on the choice of $\Psi_t$.

\begin{definition}[Strong solution]
\label{def:strong-solution}
Let $q>d+2$ and $0<T\le T_0$.  A \emph{strong solution} of
\eqref{eq:original-moving-interface-problem} on $[0,T]$ is a pair
$(u^+,u^-)$ with $u^\sigma\in W_q^{2,1}(Q_T^\sigma;\mathbb R^L)$ and values in
$N^\sigma$, such that the phase equations hold almost everywhere in
$Q_T^\sigma$, the initial condition holds in $\Omega_0^\sigma$, and the two
interface conditions hold on $\Gamma_t$, $0<t\le T$, in the sense of traces.
By the parabolic Sobolev--Morrey embedding,
$u^\sigma\in C^{1+\alpha,(1+\alpha)/2}(\overline{Q_T^\sigma};\mathbb R^L)$ for
$0<\alpha<1-(d+2)/q$, so $u^\sigma$ and $\nabla u^\sigma$ have continuous
traces on the interface and at $t=0$.
\end{definition}

The initial interface conditions in the next theorem are meaningful at the
natural regularity.  Indeed, for $u_0^\sigma\in
W_q^{2-2/q}(\Omega_0^\sigma;N^\sigma)$,
\[
    u_0^\sigma|_{\Gamma_0}
    \in W_q^{2-3/q}(\Gamma_0),
    \qquad
    \partial_{\nu_0}u_0^\sigma|_{\Gamma_0}
    \in W_q^{1-3/q}(\Gamma_0).
\]
Because $q>d+2$, these spaces embed respectively into
$C^{1+\alpha}(\Gamma_0)$ and $C^\alpha(\Gamma_0)$ for every
$0<\alpha<1-(d+2)/q$.  Thus the tangent space and the orthogonality relation
at $t=0$ have pointwise continuous representatives.

\begin{theorem}[Finite-regularity short-time existence and uniqueness]
\label{thm:main-short-time-solution}
Assume the geometric setting above, fix $q>d+2$, and let
\[
    u_0^\sigma
    \in
    W_q^{2-2/q}(\Omega_0^\sigma;N^\sigma),
    \qquad \sigma=\pm.
\]
Suppose that, on $\Gamma_0$ in the Sobolev trace sense,
\begin{equation}\label{eq:zeroth-compatibility}
    (u_0^+,u_0^-)\in\Lambda,
    \qquad
    \bigl(-\partial_{\nu_0}u_0^+,
          \partial_{\nu_0}u_0^-\bigr)
    \perp_{\mathbb R^{2L}}
    T_{(u_0^+,u_0^-)}\Lambda.
\end{equation}
Then there exists $T\in(0,T_0]$ and a unique strong solution $(u^+,u^-)$ of
\eqref{eq:original-moving-interface-problem} on $[0,T]$ in the sense of
\Cref{def:strong-solution}; in particular
\[
    u^\sigma
    \in W_q^{2,1}(Q_T^\sigma;\mathbb R^L),
    \qquad \sigma=\pm.
\]
The initial and interface conditions are attained in their natural trace
spaces.

Moreover, for every $0<\tau<T'<T$, the solution is smooth up to the moving
interface on the closed time strip $\tau\le t\le T'$.
\end{theorem}

By the parabolic trace theorem, in addition
\[
    u^\sigma
    \in
    C([0,T];W_q^{2-2/q}(\Omega_t^\sigma;\mathbb R^L)),
    \qquad \sigma=\pm.
\]

\begin{remark}[On the hypotheses of \Cref{thm:main-short-time-solution}]
\label{rem:hypotheses-main-theorem}
The assumption $q>d+2$ gives bounded gradients, by the embedding in
\Cref{def:strong-solution}.  It is
used for the compactness in the fixed-point argument
(\Cref{prop:state-map-compact-continuous}), for the preservation of the target
manifolds $N^\pm$ (\Cref{prop:target-preservation}), and for the energy
estimate proving uniqueness (\Cref{prop:physical-solution-uniqueness}).  The
theorem gives existence on some interval $[0,T]$ and asserts no lower
bound for $T$ in terms of norms of $u_0^\pm$; \Cref{sec:continuation} does not
need one, since it restarts the solution from its terminal trace.
\end{remark}

\begin{corollary}[Smooth solution up to the initial time]
\label{cor:smooth-short-time-solution}
Under the hypotheses of \Cref{thm:main-short-time-solution}, suppose in
addition that
\[
    u_0^\sigma
    \in
    C^\infty(\overline{\Omega_0^\sigma};N^\sigma),
    \qquad \sigma=\pm,
\]
and that the initial data and prescribed interface satisfy the compatibility
conditions of every order at $t=0$ in the sense of
\Cref{subsec:compatibility-conditions}.  Then the solution given by
\Cref{thm:main-short-time-solution} is smooth on the
closed moving phase cylinders:
\[
    u^\sigma
    \in
    C^\infty(\overline{Q_T^\sigma};N^\sigma),
    \qquad \sigma=\pm.
\]
\end{corollary}

Let $\mathcal V_t$ denote the scalar normal velocity of $\Gamma_t$ in the
direction $\nu_t$ and $\mathcal H_t:=-\operatorname{div}_{\Gamma_t}\nu_t$ its
scalar mean curvature, so that $\mathcal V_t=\mathcal H_t$ is mean curvature
flow, as in \eqref{eq:intro-mean-curvature-flow}.  Write $A_{\Gamma_t}$ for the second fundamental form.

\begin{corollary}[Continuation criterion under mean curvature flow]
\label{cor:mean-curvature-continuation}
Fix $q>d+2$ and $T_*<\infty$.  Let $(\Gamma_t)_{0\le t<T_*}$ be a smooth
compact mean curvature flow, and let $(u^+,u^-)$ be a strong solution of
\eqref{eq:original-moving-interface-problem} such that
\[
    u^\sigma\in W_q^{2,1}(Q_S^\sigma;\mathbb R^L)
    \qquad
    \text{for every }S<T_*,\quad \sigma=\pm.
\]
If
\begin{equation}\label{eq:mean-curvature-continuation-bound}
    \sup_{0\le t<T_*}
    \left(
       \|A_{\Gamma_t}\|_{L^\infty(\Gamma_t)}
       +\|\nabla u^+(t)\|_{L^\infty(\Omega_t^+)}
       +\|\nabla u^-(t)\|_{L^\infty(\Omega_t^-)}
    \right)<\infty,
\end{equation}
then $\Gamma_t$ extends smoothly past $T_*$, and $(u^+,u^-)$ extends uniquely
as a strong solution to $[0,T_*+\varepsilon)$ for some $\varepsilon>0$.
The extended phase fields are smooth on every closed positive-time strip.  Let
$T_{\max}$ denote the maximal time up to which the interface is smooth and the
phase fields form a strong solution.  If $T_{\max}<\infty$, then
\[
    \limsup_{t\uparrow T_{\max}}
    \left(
       \|A_{\Gamma_t}\|_{L^\infty}
       +\|\nabla u^+(t)\|_{L^\infty}
       +\|\nabla u^-(t)\|_{L^\infty}
    \right)=\infty.
\]
\end{corollary}

The same argument, without the curvature term, applies whenever a prescribed
interface family extends smoothly past $T_*$
(\Cref{prop:prescribed-interface-continuation}).

\begin{remark}[Comparison with the single-phase problem]
\label{rem:single-phase-comparison}
For a single map on a compact manifold $M$ with the free-boundary condition
$u(\partial M)\subset\Sigma$, $\partial_\nu u\perp T\Sigma$, the localized
reflection method of \cite{ChenLin1998} adds a reflected collar to $M$, solves a
Dirichlet problem on the resulting manifold with boundary, and obtains the
Dirichlet data on its outer boundary by reflection from the previous iterate.
Here the two phases have different bulk domains, and the doubled collar does not
lie in a single manifold on which one Dirichlet problem can be solved; the collar
problem and the two bulk problems are therefore solved separately and matched on
an overlap by the Schwarz map of \Cref{sec:schwarz}.  We also work at the
natural initial regularity $W_q^{2-2/q}$, whereas \cite{ChenLin1998} takes the
initial map in $W_p^2$.  This requires that the reflected initial datum stay in
$W_q^{2-2/q}$ (\Cref{subsec:reflected-doubled-coefficients}) and uses a
parabolic lifting of the initial trace (\Cref{subsec:state-trace-spaces}).
\end{remark}

\begin{remark}[Scope]
\label{rem:scope-of-setting}
The interface family is prescribed and is not coupled to $(u^+,u^-)$; models in
which the interface velocity depends on the traces of $u^\pm$ are not covered.
\end{remark}

\subsection{Minimal-pair origin and model geometries}
\label{subsec:minimal-pair-model-geometries}

In phase-transition models the trace manifold is selected by minimizing
connecting orbits.  Let $F\in C^\infty(\mathbb R^L;[0,\infty))$ satisfy
$F^{-1}(0)=N^+\cup N^-$.  For a finite-energy connecting curve
$\gamma:\mathbb R\to\mathbb R^L$, set
\[
    E_F(\gamma)
    :=\int_{\mathbb R}
       \left(\frac12|\gamma'|^2+F(\gamma)\right)\,ds,
    \qquad
    c_F:=\inf_{\gamma\in\mathscr H_F}E_F(\gamma),
\]
where $\mathscr H_F$ is the class of $H^1_{\mathrm{loc}}$ curves with
$E_F(\gamma)<\infty$, $\gamma(-\infty)\in N^-$, and
$\gamma(+\infty)\in N^+$.  The associated
minimal-pair set is
\[
    \Lambda_F
    :=\bigl\{(p^+,p^-)\in N^+\times N^-:
       E_F(\gamma)=c_F\text{ for some }\gamma\in\mathscr H_F,
       \ \gamma(\pm\infty)=p^\pm\bigr\}.
\]
Thus $\Lambda_F$ consists of the endpoint pairs of connecting curves that attain
$c_F$.  For positive-dimensional wells, the infimum over curves with prescribed
endpoints may equal $c_F$ without being attained, and such endpoint pairs do not
belong to $\Lambda_F$.
Following \cite[Definition~3.5]{FeiLinWangZhang2023}, we call $F$ fully
minimally paired if every pair in $N^+\times N^-$ is a minimal pair, and
partially minimally paired otherwise.  Whenever
$\Lambda_F$ is a nonempty compact smooth embedded submanifold without boundary, the
main theorem applies with $\Lambda=\Lambda_F$.  The proof uses only the
geometry of $N^\pm$ and $\Lambda_F$; it does not need the minimizing connection
to be unique or nondegenerate.

Besides the matrix geometry of \Cref{subsec:matrix-valued-application}, we
mention three standard examples.  At the isotropic--nematic transition,
\[
    N^-:=\{0\},
    \qquad
    N^+:=\mathcal N
    :=\left\{s_+\left(n\otimes n-\frac13I_3\right):n\in\mathbb S^2\right\}
    \cong\mathbb{RP}^2,
\]
and $\Lambda_{\mathrm{NI}}=\mathcal N\times\{0\}$, where $s_+>0$ is the constant
determined by the potential.  For isotropic elasticity, the conormal condition
is the one-sided Neumann condition $\partial_{\nu_t}n=0$ appearing in
\cite{FeiWangZhangZhang2018}; with anisotropic elasticity the formal limit
instead has the strong anchoring condition $n=\nu_t$ on the interface (see the
discussion in \cite[Section~1.1]{FeiWangZhangZhang2018}), which is not of the form considered here.  For concentric spherical wells
$N^-=a\mathbb S^{n-1}$ and $N^+=b\mathbb S^{n-1}$, $0<a<b$, consider a radial
potential $F(u)=G(|u|)$ that vanishes exactly on these two spheres, such as
$F(u)=\frac14(|u|^2-a^2)^2(|u|^2-b^2)^2$ in \cite{DongWang2025}.  Writing a
connecting curve as $\gamma=\rho\,\omega$ with $\rho=|\gamma|$ and
$\omega\in\mathbb S^{n-1}$, one has $|\gamma'|^2=|\rho'|^2+\rho^2|\omega'|^2$, so
$E_F(\gamma)$ is at least the energy of the radial curve with the same modulus,
with equality only if $\omega$ is constant.  Hence minimizing connecting orbits
are radial, and $\Lambda_F=\Lambda_{\mathrm{sph}}$, where
\[
    \Lambda_{\mathrm{sph}}
    :=\{(b\omega,a\omega):\omega\in\mathbb S^{n-1}\}.
\]
Writing $u^+=b\omega^+$ and $u^-=a\omega^-$ with
$\omega^\pm$ taking values in $\mathbb S^{n-1}$, the interface conditions
become
\[
    \omega^+=\omega^-,
    \qquad
    b^2\partial_{\nu_t}\omega^+
    =a^2\partial_{\nu_t}\omega^-,
\]
as in \cite{BronsardStoth1998,DongWang2025}.  At the fully minimally paired
extreme,
$N^-:=\mathbb S^k\times\{0\}$ and
$N^+:=\{0\}\times\mathbb S^\ell$ admit a potential for which every endpoint
pair is minimal \cite[Section~3.2]{FeiLinWangZhang2023}; hence
$\Lambda=N^+\times N^-$ and the two sides satisfy separate homogeneous
Neumann conditions.  In the first example one well is a point; the other two
give a graph relation and the full product relation, respectively.  For
$n\ge3$, the matrix potential of \Cref{subsec:matrix-valued-application} is
partially minimally paired.

\subsection{Application to the matrix-valued sharp-interface system}
\label{subsec:matrix-valued-application}

We next apply the abstract theorem to the periodic matrix-valued
sharp-interface problem.  Fix $n\ge2$ and let
\[
    \mathbb M_n:=\mathbb R^{n\times n}
\]
be equipped with the Frobenius inner product
\[
    X:Y:=\operatorname{tr}(X^\top Y),
\]
and identify $\mathbb M_n$ isometrically with $\mathbb R^{n^2}$.  Set
\[
    N^+:=\mathrm O^+(n):=\{A\in\mathrm O(n):\det A=1\},
    \qquad
    N^-:=\mathrm O^-(n):=\{A\in\mathrm O(n):\det A=-1\}.
\]
For $\xi\in\mathbb S^{n-1}$, define
\[
    R_\xi:=I_n-2\xi\xi^\top,
\]
and let
\begin{equation}\label{eq:matrix-minimal-pair-manifold}
    \Lambda_{\mathrm{mat}}
    :=
    \left\{
        (P,PR_\xi):
        P\in\mathrm O^+(n),\
        \xi\in\mathbb S^{n-1}
    \right\}
    \subset
    \mathrm O^+(n)\times\mathrm O^-(n).
\end{equation}
Since $R_\xi=R_{-\xi}$, the map
\[
    \iota_{\mathrm{mat}}:
    \mathrm O^+(n)\times\mathbb{RP}^{n-1}
    \longrightarrow
    \mathbb M_n\times \mathbb M_n,
    \qquad
    \iota_{\mathrm{mat}}(P,[\xi])=(P,PR_\xi),
\]
is well-defined and smooth.  It is injective because its first component
determines $P$, while $R_\xi$ determines the line $[\xi]$ as its
$(-1)$-eigenspace.  To check immersion, represent a tangent vector to
$\mathbb{RP}^{n-1}$ by $\eta\in\xi^\perp$.  If the differential annihilates
$(PX,\eta)$, then its first component gives $X=0$, and its second component
gives
\[
    \dot R_\xi=-2(\eta\xi^\top+\xi\eta^\top)=0.
\]
Applying this identity to $\xi$ yields $\eta=0$, so
$\iota_{\mathrm{mat}}$ is an injective immersion; since its domain is compact,
it is a smooth embedding, and $\Lambda_{\mathrm{mat}}$ is a compact smooth
embedded submanifold of
$\mathrm O^+(n)\times\mathrm O^-(n)$.  For the matrix-valued potential
\[
    F_{\mathrm{mat}}(A)
    :=
    \frac14\|A^\top A-I_n\|_{\mathrm F}^2,
\]
where $\|\cdot\|_{\mathrm F}$ is the Frobenius norm, the endpoint pairs of
globally minimizing connecting orbits are precisely the elements of
$\Lambda_{\mathrm{mat}}$ \cite[Lemmas~3.6 and~3.7]{FeiLinWangZhang2023};
that is, $\Lambda_{F_{\mathrm{mat}}}=\Lambda_{\mathrm{mat}}$ in the notation
of \Cref{subsec:minimal-pair-model-geometries}.

\begin{proposition}[Normal-derivative matching as conormality]
\label{prop:matrix-normal-derivative-matching-conormality}
Let
\[
    \mathbf P=(P,PR_\xi)\in\Lambda_{\mathrm{mat}},
\]
and suppose that
\[
    V_+\in T_P\mathrm O^+(n),
    \qquad
    V_-\in T_{PR_\xi}\mathrm O^-(n).
\]
Then
\begin{equation}\label{eq:matrix-normal-derivative-matching}
    V_+=V_-
\end{equation}
if and only if
\begin{equation}\label{eq:matrix-conormal-condition}
    (-V_+,V_-)
    \perp_{\mathbb R^{2n^2}}
    T_{\mathbf P}\Lambda_{\mathrm{mat}}.
\end{equation}
\end{proposition}

\begin{proof}
Write $R:=R_\xi$.  Differentiating
\[
    s\longmapsto
    \bigl(P\exp(sX),P\exp(sX)R_{\xi(s)}\bigr),
    \qquad
    \xi'(0)=\eta\in\xi^\perp,
\]
gives
\begin{equation}\label{eq:matrix-Lambda-tangent-space}
\begin{aligned}
    T_{(P,PR)}\Lambda_{\mathrm{mat}}
    =
    \Bigl\{
        \bigl(PX,PXR+P\dot R\bigr):\quad&
        X\in\mathfrak{so}(n),\\
        &\dot R=-2(\eta\xi^\top+\xi\eta^\top),
        \quad \eta\in\xi^\perp
    \Bigr\}.
\end{aligned}
\end{equation}

Assume first that \eqref{eq:matrix-conormal-condition} holds.  Write
\[
    V_+=PX_+,
    \qquad
    V_-=PRX_-,
    \qquad
    X_\pm\in\mathfrak{so}(n).
\]
Testing against tangent vectors of the form $(PX,PXR)$ gives
\[
    -X_+:X+(RX_-R):X=0
    \qquad
    \text{for every }X\in\mathfrak{so}(n),
\]
and therefore
\begin{equation}\label{eq:matrix-skew-relation}
    X_+=RX_-R.
\end{equation}
Testing against tangent vectors of the form $(0,P\dot R)$ gives
\[
    X_-:R\dot R=0
    \qquad
    \text{for every }\eta\in\xi^\perp.
\]
Since
\[
    R\dot R
    =2(\xi\eta^\top-\eta\xi^\top),
\]
it follows that $X_-\xi$ is orthogonal to $\xi^\perp$.  Skew-symmetry also
gives $\xi^\top X_-\xi=0$, so
\[
    X_-\xi=0.
\]
Consequently,
\[
    RX_-R=X_-,
    \qquad
    RX_-=X_-.
\]
Together with \eqref{eq:matrix-skew-relation}, this yields $X_+=X_-$ and
hence
\[
    V_-=PRX_-=PX_-=PX_+=V_+.
\]

Conversely, assume $V_+=V_-=:V$.  Write
\[
    V=PX,
    \qquad
    X\in\mathfrak{so}(n).
\]
Because $V\in T_{PR}\mathrm O^-(n)$, the matrix $RX$ is skew-symmetric,
and hence
\[
    RX=XR.
\]
Thus $X$ preserves the $(-1)$-eigenspace $\operatorname{span}\{\xi\}$ of
$R$.  Skew-symmetry then gives $X\xi=0$, and therefore
\[
    RX=XR=X.
\]
For an arbitrary tangent vector in
\eqref{eq:matrix-Lambda-tangent-space},
\[
    (-V,V):\bigl(PX_0,PX_0R+P\dot R\bigr)
    =-X:X_0+X:X_0R+X:\dot R
    =0.
\]
Indeed, the first two terms cancel because $XR=X$, while the last term
vanishes because $X$ is skew-symmetric and $\dot R$ is symmetric.  This proves
\eqref{eq:matrix-conormal-condition}.
\end{proof}

The harmonic map heat flow into the orthogonal group has the familiar matrix
form.  For $A\in\mathrm O(n)$, the orthogonal projection
onto
\[
    T_A\mathrm O(n)=\{A\Omega:\Omega\in\mathfrak{so}(n)\}
\]
is
\begin{equation}\label{eq:orthogonal-group-tangent-projection}
    P_{T_A\mathrm O(n)}Z
    =\frac12\bigl(Z-AZ^\top A\bigr).
\end{equation}
Differentiating $A^\top A=I_n$ twice and using
\eqref{eq:orthogonal-group-tangent-projection} gives
\begin{equation}\label{eq:orthogonal-group-harmonic-map-operator}
    P_{T_A\mathrm O(n)}\Delta A
    =
    \Delta A
    -\sum_{j=1}^d
      \partial_{x_j}A\,A^\top\partial_{x_j}A.
\end{equation}

\begin{corollary}[Local smooth solvability for the matrix-valued system]
\label{cor:matrix-valued-short-time-solution}
Let $\{\Gamma_t\}_{0\le t\le T_0}$ be a prescribed smooth family of
separating hypersurfaces in $\mathbb T^d$, and fix $q>d+2$.  Let
\[
    A_0^\sigma
    \in C^\infty(\overline{\Omega_0^\sigma};\mathrm O^\sigma(n)),
    \qquad \sigma=\pm,
\]
and assume on $\Gamma_0$, in the trace sense, that
\[
    (A_0^+,A_0^-)\in\Lambda_{\mathrm{mat}},
    \qquad
    \partial_{\nu_0}A_0^+
    =
    \partial_{\nu_0}A_0^-.
\]
Assume also that the initial data and the prescribed interface satisfy the
compatibility conditions of every order at $t=0$, in the sense of
\Cref{subsec:compatibility-conditions} with $\Lambda=\Lambda_{\mathrm{mat}}$,
for the system
\begin{equation}\label{eq:matrix-valued-flat-torus-system}
\begin{cases}
    \displaystyle
    \partial_tA^\sigma
    =
    \Delta A^\sigma
    -\displaystyle\sum_{j=1}^d
      \partial_{x_j}A^\sigma
      (A^\sigma)^\top
      \partial_{x_j}A^\sigma,
    &x\in\Omega_t^\sigma,\quad \sigma=\pm,\\[0.8em]
    (A^+,A^-)\in\Lambda_{\mathrm{mat}},
    &x\in\Gamma_t,\\[0.4em]
    \partial_{\nu_t}A^+
    =
    \partial_{\nu_t}A^-,
    &x\in\Gamma_t,\\[0.4em]
    A^\sigma(\cdot,0)=A_0^\sigma,
    &x\in\Omega_0^\sigma,\quad \sigma=\pm.
\end{cases}
\end{equation}
Then, for some $T\in(0,T_0]$, the system
\eqref{eq:matrix-valued-flat-torus-system} has a unique strong solution
$(A^+,A^-)$, and this solution satisfies
\[
    A^\sigma
    \in C^\infty(\overline{Q_T^\sigma};\mathrm O^\sigma(n)),
    \qquad \sigma=\pm.
\]
\end{corollary}

\begin{proof}
The wells $\mathrm O^\pm(n)\subset \mathbb M_n$ and the manifold
$\Lambda_{\mathrm{mat}}$ satisfy the geometric hypotheses of
\Cref{thm:main-short-time-solution}.  Identity
\eqref{eq:orthogonal-group-harmonic-map-operator} identifies the abstract bulk
equation with the first equation in
\eqref{eq:matrix-valued-flat-torus-system}.  By
\Cref{prop:matrix-normal-derivative-matching-conormality}, the matrix
normal-derivative matching condition is equivalent to the abstract conormal
interface condition in
\eqref{eq:original-moving-interface-problem}; applied order by order to the
time jets, it also identifies the compatibility conditions of
\Cref{subsec:compatibility-conditions} with those of the matching form.  The
conclusion therefore
follows from \Cref{thm:main-short-time-solution}.  The all-order compatibility
conditions then allow the bootstrap through the initial corner by
\Cref{prop:initial-corner-smoothness}, yielding smoothness on the closed
moving phase cylinders.
\end{proof}

\subsection*{Notation guide}
The following list collects the notation most easily confused. Precise
definitions are given where each symbol first appears.
\begin{itemize}[leftmargin=2.3em,itemsep=0pt,topsep=3pt]
\item $\sigma\in\{+,-\}$: phase index; $\nu_t$: unit normal pointing from
      $\Omega_t^-$ into $\Omega_t^+$.
\item $\Pi_\Lambda$, $\mathcal R_\Lambda$: nearest-point projection onto
      $\Lambda$ and normal reflection across $\Lambda$.
\item $\Phi_t^\sigma$: one-sided Fermi map for phase $\sigma$;
      $\Psi_t$: ambient isotopy adapted to the Fermi maps.
\item $\mathcal C$, $\mathcal C_{\mathrm{dbl}}$: positive and doubled collar
      coordinate domains.
\item $\mathcal D^\sigma$: fixed bulk domain for phase $\sigma$;
      $\Sigma_\rho^\sigma$: offset hypersurface at $r=\rho$.
\item Subscripts $c$, $b$: collar and bulk coefficients and operators;
      subscripts $\mathrm{ref}$, $\mathrm{dbl}$: coefficients conjugated by the
      target reflection and coefficients on the doubled collar.
\item $\mathcal T_\rho$: two-phase trace at $r=\rho$.
\item $\mathscr L_{\mathbf z}$: doubled linear operator frozen at the state
      $\mathbf z$.
\item $\mathcal Y_T$: parabolic Dirichlet trace space.
\item $\mathcal X_T^{\mathrm{match}}$: matched H\"older state space.
\item $\mathcal A_T$, $\mathcal K_T$: admissible frozen-state set and
      admissible Schwarz trace set.
\end{itemize}

\section{Geometry of the trace constraint}
\label{sec:geometry}

\medskip
\noindent\textit{Convention on constants.}
Unless otherwise stated, $C$ denotes a positive constant that may change from
line to line.  It may depend on the fixed torus, the prescribed family
$\{\Gamma_t\}_{0\le t\le T_0}$, the wells $N^\pm$, the submanifold $\Lambda$,
the fixed coordinate maps, the collar parameters, the initial data, and the
fixed exponents.  It is independent of the short-time length
$T\in(0,T_0]$, the frozen state, and the admissible Schwarz trace.  Dependence
on additional quantities is indicated explicitly.

\subsection{Nearest-point projection and normal reflection}

Following the normal-reflection strategy of Struwe and Chen--Lin
\cite{Struwe1991,ChenLin1998}, we use reflection across $\Lambda$ in the
product ambient space $\mathbb R^{2L}$.

\begin{definition}[Tubular neighborhood and normal reflection]
\label{def:reflection-map}
For $\rho>0$, set
\[
    \operatorname{Tub}_\rho(\Lambda)
    :=\{\mathbf z\in\mathbb R^{2L}:
       \operatorname{dist}_{\mathbb R^{2L}}(\mathbf z,\Lambda)<\rho\}.
\]
Choose $\rho_\Lambda>0$ so small that the normal exponential map
\[
    \Theta_\Lambda:
    \{(\mathbf p,\boldsymbol\eta):
      \mathbf p\in\Lambda,
      \boldsymbol\eta\in T_{\mathbf p}^{\perp,\mathrm{amb}}\Lambda,
      |\boldsymbol\eta|<2\rho_\Lambda\}
    \longrightarrow\operatorname{Tub}_{2\rho_\Lambda}(\Lambda),
\]
\[
    \Theta_\Lambda(\mathbf p,\boldsymbol\eta)
    =\mathbf p+\boldsymbol\eta,
\]
is a diffeomorphism.  Let
\[
    \Pi_\Lambda:\operatorname{Tub}_{2\rho_\Lambda}(\Lambda)\longrightarrow\Lambda
\]
be the nearest-point projection and define
\begin{equation}\label{eq:reflection-map}
    \mathcal R_\Lambda(\mathbf z)
    :=2\Pi_\Lambda(\mathbf z)-\mathbf z.
\end{equation}
\end{definition}

\begin{lemma}[Differential properties of the reflection]
\label{lem:reflection-properties}
The map $\mathcal R_\Lambda$ is a smooth involution of
$\operatorname{Tub}_{2\rho_\Lambda}(\Lambda)$ and
\begin{equation}\label{eq:reflection-involution}
    \mathcal R_\Lambda\circ\mathcal R_\Lambda
    =\operatorname{Id},
    \qquad
    [D\mathcal R_\Lambda(\mathbf z)]^{-1}
    =D\mathcal R_\Lambda(\mathcal R_\Lambda(\mathbf z)).
\end{equation}
At every $\mathbf p\in\Lambda$,
\begin{equation}\label{eq:reflection-differential}
    D\mathcal R_\Lambda(\mathbf p)
    =\operatorname{Id}
    \quad\text{on }T_{\mathbf p}\Lambda,
    \qquad
    D\mathcal R_\Lambda(\mathbf p)
    =-\operatorname{Id}
    \quad\text{on }T_{\mathbf p}^{\perp,\mathrm{amb}}\Lambda.
\end{equation}
The fixed-point set of $\mathcal R_\Lambda$ is exactly $\Lambda$.  Moreover,
\begin{equation}\label{eq:reflection-preserves-distance}
    \operatorname{dist}\bigl(\mathcal R_\Lambda(\mathbf z),\Lambda\bigr)
    =\operatorname{dist}(\mathbf z,\Lambda),
\end{equation}
and hence, for every $0<\rho<2\rho_\Lambda$,
\begin{equation}\label{eq:reflection-invariant-subtube}
    \mathcal R_\Lambda\bigl(\operatorname{Tub}_\rho(\Lambda)\bigr)
    =\operatorname{Tub}_\rho(\Lambda),
    \qquad
    \mathcal R_\Lambda\bigl(\overline{\operatorname{Tub}_\rho(\Lambda)}\bigr)
    =\overline{\operatorname{Tub}_\rho(\Lambda)}.
\end{equation}
\end{lemma}

\begin{proof}
In tubular coordinates write
\[
    \mathbf z=\mathbf p+\boldsymbol\eta,
    \qquad
    \mathbf p\in\Lambda,
    \quad
    \boldsymbol\eta\in T_{\mathbf p}^{\perp,\mathrm{amb}}\Lambda.
\]
Then
\[
    \Pi_\Lambda(\mathbf p+\boldsymbol\eta)=\mathbf p,
    \qquad
    \mathcal R_\Lambda(\mathbf p+\boldsymbol\eta)
    =\mathbf p-\boldsymbol\eta.
\]
In particular,
\[
    \operatorname{dist}(\mathbf p-\boldsymbol\eta,\Lambda)
    =|\boldsymbol\eta|
    =\operatorname{dist}(\mathbf p+\boldsymbol\eta,\Lambda),
\]
which proves \eqref{eq:reflection-preserves-distance} and
\eqref{eq:reflection-invariant-subtube}.  The same formulas immediately give
smoothness, the involution identity, and the fixed-point statement.  At $\boldsymbol\eta=0$, a tangent variation changes $\mathbf p$
and is unchanged by reflection, while a normal variation changes
$\boldsymbol\eta$ and changes sign.  This proves
\eqref{eq:reflection-differential}.  Differentiating the involution identity
gives the inverse formula in \eqref{eq:reflection-involution}.
\end{proof}

We reflect in the full product space $\mathbb R^{2L}$, where the doubled problem
is solved.  After the target constraint is recovered, the ambient and relative
conormal formulations agree.

\section{Reflection and the doubled collar problem}
\label{sec:fixed-coordinates}

\subsection{An ambient isotopy adapted to Fermi coordinates}
\label{subsec:global-parametrization}

Let
\[
    X_t:\Gamma_0\longrightarrow\Gamma_t,
    \qquad 0\le t\le T_0,
\]
be a smooth parametrization and let $\nu_t$ point into $\Omega_t^+$.  Using the
canonical flat structure of $\mathbb T^d$, for $r_*>0$ sufficiently small, define
the signed Fermi map by
\begin{equation}\label{eq:signed-Fermi-map}
    \Phi_t(p,s)
    :=
    X_t(p)+s\nu_t(X_t(p))
    \quad\text{in }\mathbb T^d,
    \qquad
    p\in\Gamma_0,
    \quad |s|<r_*.
\end{equation}
After decreasing $r_*$, the maps $\Phi_t$ are smooth diffeomorphisms onto a
uniform tubular neighborhood of $\Gamma_t$.

For $\sigma\in\{+,-\}$ and $r\ge0$, write
\begin{equation}\label{eq:one-sided-Fermi-map}
    \Phi_t^\sigma(p,r)
    :=
    \Phi_t(p,\sigma r).
\end{equation}
Then $\Phi_t^\sigma$ parametrizes the one-sided collar
\[
    \mathcal C_t^\sigma(\rho)
    :=
    \Phi_t^\sigma
    \bigl(\Gamma_0\times(0,\rho)\bigr)
    \subset\Omega_t^\sigma.
\]

\begin{lemma}[Ambient isotopy adapted to Fermi coordinates]
\label{lem:fermi-compatible-parametrization}
After decreasing $r_*>0$, there exists a smooth family of diffeomorphisms
\[
    \Psi_t:\mathbb T^d\longrightarrow\mathbb T^d,
    \qquad 0\le t\le T_0,
\]
such that
\[
    \Psi_0=\operatorname{Id}_{\mathbb T^d},
    \qquad
    \Psi_t(\Gamma_0)=\Gamma_t,
    \qquad
    \Psi_t(\Omega_0^\sigma)=\Omega_t^\sigma,
\]
and
\begin{equation}\label{eq:parametrization-compatibility}
    \Psi_t(\Phi_0^\sigma(p,r))
    =
    \Phi_t^\sigma(p,r),
    \qquad
    0\le r\le r_*,
    \quad \sigma=\pm.
\end{equation}
\end{lemma}

\begin{proof}
Choose $r_1>0$ such that
\[
    \Phi_t:\Gamma_0\times(-2r_1,2r_1)\longrightarrow\mathbb T^d
\]
is a smooth diffeomorphism onto its image, uniformly in $t$.  On this tubular
neighborhood define
\[
    V_0(\Phi_t(p,s),t):=\partial_t\Phi_t(p,s).
\]
Choose $\chi\in C_c^\infty((-2r_1,2r_1))$ with
$\chi=1$ on $[-r_1,r_1]$ and set
\[
    V(\Phi_t(p,s),t):=\chi(s)V_0(\Phi_t(p,s),t).
\]
Since $V$ vanishes near the boundary of the tubular neighborhood, it extends
smoothly by zero to a time-dependent vector field on the compact torus.  Let
$\Psi_t$ be its nonautonomous flow.  For $|s|\le r_1$, the curves
$t\mapsto\Psi_t(\Phi_0(p,s))$ and $t\mapsto\Phi_t(p,s)$ solve the same ODE with
the same initial value.  Hence
\[
    \Psi_t(\Phi_0(p,s))=\Phi_t(p,s).
\]
The remaining assertions follow from preservation of the signed normal
coordinate and the fact that $\Psi_t$ is a diffeomorphism.  Relabel $r_1$ as
$r_*$.
\end{proof}

After decreasing $r_*>0$ if necessary, the map
\[
    (p,s,t)
    \longmapsto
    \bigl(\Phi_t(p,s),t\bigr)
\]
is a smooth diffeomorphism onto its image.  On every closed smaller collar,
$\Phi_t$, $\Phi_t^{-1}$, and all their derivatives are bounded uniformly for
$t\in[0,T_0]$.

After the initial data have been fixed, choose $\delta>0$ so small that
\begin{equation}\label{eq:delta-choice}
    0<2\delta<r_*,
\end{equation}
the two bulk domains defined below are nonempty, and
\begin{equation}\label{eq:initial-paired-collar-tube}
\begin{aligned}
    \Bigl\{
       \bigl(
          u_0^+(\Phi_0^+(p,r)),
          u_0^-(\Phi_0^-(p,r))
       \bigr):
       p\in\Gamma_0,\ 0\le r\le2\delta
    \Bigr\}
    \Subset
    \operatorname{Tub}_{\rho_\Lambda}(\Lambda).
\end{aligned}
\end{equation}
The last condition can be met because the paired trace at $r=0$ lies in
$\Lambda$ and the initial maps are continuous.  Set
\begin{equation}\label{eq:collars}
    \mathcal C:=\Gamma_0\times(0,2\delta),
    \qquad
    \mathcal C_{\mathrm{dbl}}
    :=\Gamma_0\times(-2\delta,2\delta).
\end{equation}
The fixed bulk domains are
\begin{equation}\label{eq:bulk-domains}
    \mathcal D^\sigma
    :=
    \Omega_0^\sigma
    \setminus
    \overline{\mathcal C_0^\sigma(\delta)},
    \qquad \sigma=\pm.
\end{equation}
Each $\mathcal D^\sigma$ is a smooth relatively compact domain and
$\overline{\mathcal D^\sigma}$ is a compact manifold with boundary.  Since the ambient
space has no boundary,
\begin{equation}\label{eq:bulk-boundaries}
    \partial \mathcal D^\sigma=\Sigma_\delta^\sigma,
\end{equation}
where, for $0<\rho\le2\delta$,
\[
    \Sigma_\rho:=\Gamma_0\times\{\rho\},
    \qquad
    \Sigma_\rho^\sigma:=\Phi_0^\sigma(\Sigma_\rho).
\]
The hypersurface $\Sigma_\delta^\sigma$ is the boundary of $\mathcal D^\sigma$, while
$\Sigma_\rho^\sigma$ is an interior hypersurface of $\mathcal D^\sigma$ for
$\delta<\rho\le2\delta$.  The collars and the bulk domains are shown in
\Cref{fig:torus-domain-decomposition}.

\begin{definition}[Two-phase trace]
\label{def:two-phase-trace}
For fields $w^\sigma$ defined near $\Sigma_\rho^\sigma$, define
\begin{equation}\label{eq:two-phase-trace}
    \mathcal T_\rho(w^+,w^-)(p,t)
    :=
    \Bigl(
        w^+(\Phi_0^+(p,\rho),t),
        w^-(\Phi_0^-(p,\rho),t)
    \Bigr)
    \in\mathbb R^{2L}.
\end{equation}
For $\rho=\delta$ this is a boundary trace, while for
$\delta<\rho\le2\delta$ it is an interior hypersurface trace.
\end{definition}

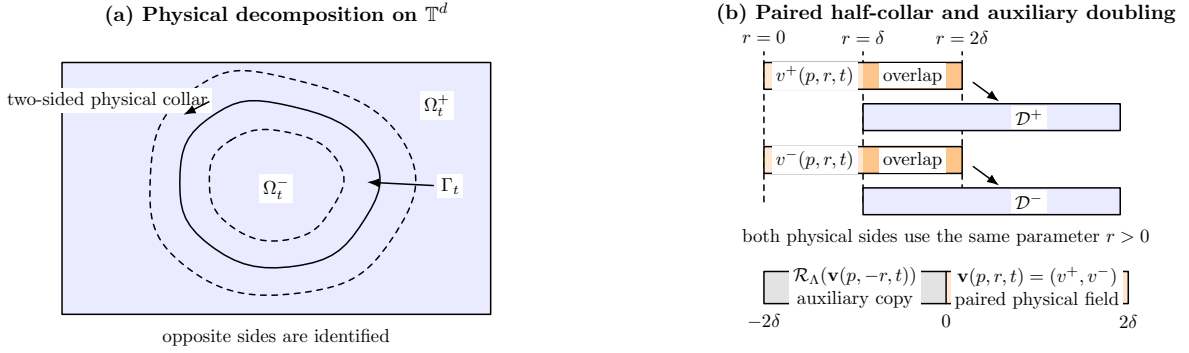
\begin{figure}[htbp]
\centering
\resizebox{0.98\textwidth}{!}{%
\begin{tikzpicture}[
    x=1cm,y=1cm,>=Latex,
    line join=round,line cap=round,
    every node/.style={font=\small}
]
\tikzset{
    boundary/.style={thick},
    aux/.style={dashed,thick},
    label/.style={fill=white,inner sep=2pt,align=center},
    arrow/.style={->,thick,shorten >=2pt,shorten <=2pt}
}
\colorlet{bulkcol}{blue!8}
\colorlet{collarcol}{orange!18}
\colorlet{overlapcol}{orange!45}
\colorlet{reflcol}{gray!22}

\begin{scope}[xshift=-6.7cm]
  \node[font=\bfseries] at (0,3.25) {(a) Physical decomposition on $\mathbb T^d$};
  \fill[bulkcol] (-4,-2.35) rectangle (4,2.35);
  \draw[boundary] (-4,-2.35) rectangle (4,2.35);
  \node at (0,-2.78) {opposite sides are identified};
  \draw[boundary] plot[smooth cycle,tension=0.85] coordinates
    {(-1.8,0.2) (-1.2,1.35) (0.2,1.55) (1.75,0.7)
     (1.6,-0.65) (0.35,-1.45) (-1.25,-1.05)};
  \draw[aux] plot[smooth cycle,tension=0.85] coordinates
    {(-2.35,0.25) (-1.55,1.85) (0.25,2.05) (2.35,0.9)
     (2.15,-0.95) (0.45,-1.95) (-1.7,-1.45)};
  \draw[aux] plot[smooth cycle,tension=0.85] coordinates
    {(-1.25,0.15) (-0.85,0.9) (0.2,1.05) (1.15,0.5)
     (1.05,-0.35) (0.25,-0.95) (-0.85,-0.7)};
  \node[label] at (0,0) {$\Omega_t^-$};
  \node[label] at (3.0,1.55) {$\Omega_t^+$};
  \node[label] (g) at (3.25,0.05) {$\Gamma_t$};
  \draw[arrow] (g.west) -- (1.62,0.12);
  \node[label] (c) at (-3.15,1.65) {two-sided physical collar};
  \draw[arrow] (c.east) -- (-1.8,1.34);
\end{scope}

\begin{scope}[xshift=5.8cm]
  \node[font=\bfseries] at (0,3.25) {(b) Paired half-collar and auxiliary doubling};

  \draw[aux] (-3.4,-0.25)--(-3.4,2.55);
  \draw[aux] (-1.55,-0.25)--(-1.55,2.55);
  \draw[aux] (0.3,-0.25)--(0.3,2.55);
  \node at (-3.4,2.82) {$r=0$};
  \node at (-1.55,2.82) {$r=\delta$};
  \node at (0.3,2.82) {$r=2\delta$};

  \fill[collarcol] (-3.4,1.85) rectangle (0.3,2.35);
  \fill[overlapcol] (-1.55,1.85) rectangle (0.3,2.35);
  \draw[boundary] (-3.4,1.85) rectangle (0.3,2.35);
  \fill[bulkcol] (-1.55,1.08) rectangle (3.25,1.58);
  \draw[boundary] (-1.55,1.08) rectangle (3.25,1.58);
  \node[label] at (-2.45,2.10) {$v^+(p,r,t)$};
  \node[label] at (-0.62,2.10) {overlap};
  \node[label] at (1.55,1.33) {$\mathcal D^+$};
  \draw[arrow] (0.45,2.02) -- (1.05,1.58);

  \fill[collarcol] (-3.4,0.28) rectangle (0.3,0.78);
  \fill[overlapcol] (-1.55,0.28) rectangle (0.3,0.78);
  \draw[boundary] (-3.4,0.28) rectangle (0.3,0.78);
  \fill[bulkcol] (-1.55,-0.49) rectangle (3.25,0.01);
  \draw[boundary] (-1.55,-0.49) rectangle (3.25,0.01);
  \node[label] at (-2.45,0.53) {$v^-(p,r,t)$};
  \node[label] at (-0.62,0.53) {overlap};
  \node[label] at (1.55,-0.24) {$\mathcal D^-$};
  \draw[arrow] (0.45,0.45) -- (1.05,0.01);

  \node[align=center] at (0,-0.92)
    {both physical sides use the same parameter $r>0$};

  \fill[reflcol] (-3.4,-2.15) rectangle (0,-1.55);
  \fill[collarcol] (0,-2.15) rectangle (3.4,-1.55);
  \draw[boundary] (-3.4,-2.15) rectangle (3.4,-1.55);
  \draw[boundary] (0,-2.15)--(0,-1.55);
  \node[label] at (-1.7,-1.85)
    {$\mathcal R_\Lambda(\mathbf v(p,-r,t))$\\auxiliary copy};
  \node[label] at (1.7,-1.85)
    {$\mathbf v(p,r,t)=(v^+,v^-)$\\paired physical field};
  \node at (-3.4,-2.48) {$-2\delta$};
  \node at (0,-2.48) {$0$};
  \node at (3.4,-2.48) {$2\delta$};
\end{scope}
\end{tikzpicture}%
}
\caption{The first panel shows the physical torus.  In the analytic
construction, the two one-sided physical collars are folded onto the same
half-collar $r>0$ as the two components of $\mathbf v=(v^+,v^-)$.  Each
component overlaps its own bulk domain on $\delta<r<2\delta$.  The lower strip
is the auxiliary doubling of this paired field; in particular, the region
$r<0$ is not the physical phase $\Omega_t^-$.}
\label{fig:torus-domain-decomposition}
\end{figure}
\FloatBarrier

\subsection{Ambient extensions of the target nonlinearities}
\label{subsec:ambient-extensions}

For each $\sigma\in\{+,-\}$, choose a tubular neighborhood
$\operatorname{Tub}_{\rho_N}(N^\sigma)\subset\mathbb R^L$ and let
\[
    \Pi_{N^\sigma}:\operatorname{Tub}_{\rho_N}(N^\sigma)\longrightarrow N^\sigma
\]
be the nearest-point projection.  On a smaller tubular neighborhood define
\begin{equation}\label{eq:target-nonlinearity-extension}
    \widetilde{\mathrm{II}}^\sigma(z)(X,Y)
    :=-D^2\Pi_{N^\sigma}(z)(X,Y),
    \qquad
    z\in \operatorname{Tub}_{\rho_N/2}(N^\sigma),
    \quad X,Y\in\mathbb R^L.
\end{equation}
Then $\widetilde{\mathrm{II}}^\sigma=\mathrm{II}^\sigma$ on $N^\sigma$.  Extend
$\widetilde{\mathrm{II}}^\sigma$ smoothly with compact support to all of
$\mathbb R^L$.  Only its restriction to the fixed tubular neighborhood is used
below.

For a smooth time-dependent map $u$ taking values in $N^\sigma$,
differentiating $\Pi_{N^\sigma}(u)=u$ twice in the spatial variables and using
$D\Pi_{N^\sigma}(u)=P_{T_uN^\sigma}$ shows that
\[
    \partial_tu=P_{T_uN^\sigma}\Delta u
    \quad\Longleftrightarrow\quad
    \partial_tu
    =\Delta u
     +\widetilde{\mathrm{II}}^\sigma(u)(\nabla u,\nabla u).
\]

\subsection{The one-sided collar equations}

Define
\begin{equation}\label{eq:collar-pullback}
    v^\sigma(p,r,t)
    :=
    u^\sigma(\Phi_t^\sigma(p,r),t),
    \qquad 0<r<r_*.
\end{equation}
The initial collar data are
\[
    v_0^\sigma(p,r)
    :=
    u_0^\sigma(\Phi_0^\sigma(p,r)).
\]
With
\[
    \mathbf v:=(v^+,v^-),
    \qquad
    \mathbf v_0:=(v_0^+,v_0^-),
\]
we have
\[
    \partial_rv^+(p,0,t)=\partial_{\nu_t}u^+(X_t(p),t),
    \qquad
    \partial_rv^-(p,0,t)=-\partial_{\nu_t}u^-(X_t(p),t).
\]
Thus
\[
    \partial_r\mathbf v
    =
    \bigl(\partial_{\nu_t}u^+,-\partial_{\nu_t}u^-\bigr),
\]
which differs from the variational conormal vector by an overall minus sign.
Orthogonality is unchanged by this sign.  Hence the interface conditions become
\begin{equation}\label{eq:collar-interface-conditions}
    \mathbf v(p,0,t)\in\Lambda,
    \qquad
    \partial_r\mathbf v(p,0,t)
    \perp_{\mathbb R^{2L}}T_{\mathbf v(p,0,t)}\Lambda.
\end{equation}

Use local coordinates $x^d=r$ on $\Gamma_0\times(0,r_*)$.  Indices
$i,j,k,\ell$ run from $1$ to $d$, while $a,b$ are tangential.  Repeated
spatial indices are summed unless explicitly stated otherwise.  Let
\[
    a_{c,ij}^\sigma
    :=
    \partial_i\Phi_t^\sigma\cdot\partial_j\Phi_t^\sigma,
    \qquad
    \lvert a_c^\sigma\rvert:=\det(a_{c,ij}^\sigma),
\]
and let $(a_c^{\sigma,ij})$ be the inverse metric, which is the collar principal coefficient.  In Fermi coordinates,
\[
    a_{c,ar}^\sigma=0,
    \qquad
    a_{c,rr}^\sigma=1.
\]
For a scalar or Euclidean-vector-valued field $U$, the Euclidean Laplacian in
these coordinates is
\begin{equation}\label{eq:Fermi-Laplacian-formula}
    (\Delta_xU)\circ\Phi_t^\sigma
    =
    \lvert a_c^\sigma\rvert^{-1/2}
    \partial_i
    \bigl(\lvert a_c^\sigma\rvert^{1/2}a_c^{\sigma,ij}
    \partial_j(U\circ\Phi_t^\sigma)\bigr).
\end{equation}
If $v^\sigma= u^\sigma\circ\Phi_t^\sigma$, then
\begin{equation}\label{eq:Fermi-time-derivative-formula}
    \partial_tv^\sigma
    =
    (\partial_tu^\sigma)\circ\Phi_t^\sigma
    +
    \left\langle
       \partial_t\Phi_t^\sigma,
       \nabla_xu^\sigma\circ\Phi_t^\sigma
    \right\rangle.
\end{equation}
In collar coordinates, the transport term in
\eqref{eq:Fermi-time-derivative-formula} becomes
\[
    a_c^{\sigma,k\ell}
    \langle\partial_t\Phi_t^\sigma,
           \partial_\ell\Phi_t^\sigma\rangle
    \partial_kv^\sigma.
\]
Expanding the divergence in \eqref{eq:Fermi-Laplacian-formula} and using the
ambient form of harmonic map heat flow gives the coefficient formulas below.

The pulled-back equation is
\begin{equation}\label{eq:collar-equation-component}
    \partial_tv^\sigma
    -a_c^{\sigma,ij}\partial_{ij}v^\sigma
    -b_c^{\sigma,k}\partial_kv^\sigma
    =
    Q_c^\sigma(p,r,t,v^\sigma,\nabla v^\sigma),
\end{equation}
where
\begin{equation}\label{eq:collar-target-nonlinearity}
    Q_c^\sigma(p,r,t,z,\Xi)
    :=
    a_c^{\sigma,ij}(p,r,t)
    \widetilde{\mathrm{II}}^\sigma(z)(\Xi_i,\Xi_j),
\end{equation}
and
\begin{equation}\label{eq:collar-first-order-coefficient}
    b_c^{\sigma,k}
    =
    \lvert a_c^\sigma\rvert^{-1/2}
    \partial_i\bigl(\lvert a_c^\sigma\rvert^{1/2}a_c^{\sigma,ik}\bigr)
    +
    a_c^{\sigma,k\ell}
    \left\langle
        \partial_t\Phi_t^\sigma,
        \partial_\ell\Phi_t^\sigma
    \right\rangle.
\end{equation}

For $\mathbf z=(z^+,z^-)\in\mathbb R^{2L}$, define
\[
    \mathsf A_c^{ij}(p,r,t),\ \mathsf B_c^k(p,r,t)
    \in\mathcal L(\mathbb R^{2L})
\]
by
\begin{equation}\label{eq:pair-linear-coefficients}
    \bigl(\mathsf A_c^{ij}\mathbf z\bigr)^\sigma
    :=
    a_c^{\sigma,ij}z^\sigma,
    \qquad
    \bigl(\mathsf B_c^k\mathbf z\bigr)^\sigma
    :=b_c^{\sigma,k}z^\sigma.
\end{equation}
For
\[
    \boldsymbol\Xi=(\boldsymbol\Xi_1,\ldots,\boldsymbol\Xi_d)
    \in(\mathbb R^{2L})^d,
    \qquad
    \boldsymbol\Xi_i=(\Xi_i^+,\Xi_i^-),
\]
define
\begin{equation}\label{eq:pair-target-nonlinearity}
    \bigl(\mathbf Q_c(p,r,t,\mathbf z,\boldsymbol\Xi)\bigr)^\sigma
    :=
    Q_c^\sigma(p,r,t,z^\sigma,\Xi^\sigma).
\end{equation}
Then the physical pair equation on $\mathcal C$ is
\begin{equation}\label{eq:physical-collar-problem}
\begin{cases}
    \partial_t\mathbf v
    -
    \mathsf A_c^{ij}\partial_{ij}\mathbf v
    -\mathsf B_c^k\partial_k\mathbf v
    =
    \mathbf Q_c(p,r,t,\mathbf v,\nabla\mathbf v),
    & 0<r<2\delta,\\[0.3em]
    \mathbf v(p,0,t)\in\Lambda,
    & r=0,\\[0.3em]
    \partial_r\mathbf v(p,0,t)
    \perp_{\mathbb R^{2L}}T_{\mathbf v(p,0,t)}\Lambda,
    & r=0,\\[0.3em]
    \mathbf v(p,2\delta,t)=\mathbf h(p,t),
    & r=2\delta,\\[0.3em]
    \mathbf v(p,r,0)=\mathbf v_0(p,r).
\end{cases}
\end{equation}
The trace $\mathbf h$ is fixed in \Cref{sec:schwarz} by the matching
condition of the Schwarz map.

\subsection{Reflected and doubled coefficients}
\label{subsec:reflected-doubled-coefficients}

Set
\[
    \varepsilon_i
    :=
    \begin{cases}
        1,&1\le i\le d-1,\\
        -1,&i=d.
    \end{cases}
\]
For $r<0$, define the spatially reflected coefficients
\begin{equation}\label{eq:spatially-reflected-coefficients}
    \widetilde a_c^{\sigma,ij}(p,r,t)
    :=
    \varepsilon_i\varepsilon_ja_c^{\sigma,ij}(p,-r,t),
    \qquad
    \widetilde b_c^{\sigma,i}(p,r,t)
    :=
    \varepsilon_ib_c^{\sigma,i}(p,-r,t).
\end{equation}
Define
\[
    \widetilde{\mathsf A}_c^{ij}(p,r,t),\
    \widetilde{\mathsf B}_c^{\,k}(p,r,t)
    \in\mathcal L(\mathbb R^{2L})
\]
componentwise by
\begin{equation}\label{eq:spatially-reflected-linear-coefficients}
    \bigl(\widetilde{\mathsf A}_c^{ij}\mathbf z\bigr)^\sigma
    :=
    \widetilde a_c^{\sigma,ij}z^\sigma,
    \qquad
    \bigl(\widetilde{\mathsf B}_c^{\,k}\mathbf z\bigr)^\sigma
    :=\widetilde b_c^{\sigma,k}z^\sigma,
\end{equation}
and define
\begin{equation}\label{eq:spatially-reflected-nonlinearity}
    \bigl(
        \widetilde{\mathbf Q}_c
        (p,r,t,\mathbf z,\boldsymbol\Xi)
    \bigr)^\sigma
    :=
    \widetilde a_c^{\sigma,ij}
    \widetilde{\mathrm{II}}^\sigma(z^\sigma)
    (\Xi_i^\sigma,\Xi_j^\sigma).
\end{equation}

For $r<0$, conjugate these coefficients by the target reflection.  Set
$S_\Lambda(\mathbf z):=D\mathcal R_\Lambda(\mathbf z)$.  Define
\begin{equation}\label{eq:target-reflected-principal}
    \mathsf A_{\mathrm{ref}}^{ij}(p,r,t,\mathbf z)
    :=S_\Lambda(\mathbf z)^{-1}
      \circ\widetilde{\mathsf A}_c^{ij}(p,r,t)
      \circ S_\Lambda(\mathbf z).
\end{equation}
The first-order coefficient is defined analogously:
\begin{equation}\label{eq:target-reflected-first-order}
    \mathsf B_{\mathrm{ref}}^k(p,r,t,\mathbf z)
    :=S_\Lambda(\mathbf z)^{-1}
      \circ\widetilde{\mathsf B}_c^{\,k}(p,r,t)
      \circ S_\Lambda(\mathbf z).
\end{equation}
The reflected nonlinearity is
\begin{equation}\label{eq:target-reflected-nonlinearity}
\begin{aligned}
    &\mathbf Q_{\mathrm{ref}}
      (p,r,t,\mathbf z,\boldsymbol\Xi)\\
    &:=S_\Lambda(\mathbf z)^{-1}
      \Bigl[
        \widetilde{\mathbf Q}_c
        \bigl(p,r,t,\mathcal R_\Lambda(\mathbf z),
              S_\Lambda(\mathbf z)\boldsymbol\Xi\bigr)\\
    &\hspace{4em}
        +\widetilde{\mathsf A}_c^{ij}(p,r,t)
          D^2\mathcal R_\Lambda(\mathbf z)
          (\boldsymbol\Xi_i,\boldsymbol\Xi_j)
      \Bigr].
\end{aligned}
\end{equation}
Here $S_\Lambda(\mathbf z)$ acts componentwise on the $d$-tuple
$\boldsymbol\Xi$.

The doubled coefficients on $\mathcal C_{\mathrm{dbl}}$ are
\begin{equation}\label{eq:doubled-principal-coefficient}
    \mathsf A_{\mathrm{dbl}}^{ij}(p,r,t,\mathbf z)
    :=
    \begin{cases}
        \mathsf A_c^{ij}(p,r,t),&0\le r<2\delta,\\
        \mathsf A_{\mathrm{ref}}^{ij}(p,r,t,\mathbf z),
        &-2\delta<r<0,
    \end{cases}
\end{equation}
\begin{equation}\label{eq:doubled-first-order-coefficient}
    \mathsf B_{\mathrm{dbl}}^k(p,r,t,\mathbf z)
    :=
    \begin{cases}
        \mathsf B_c^k(p,r,t),&0\le r<2\delta,\\
        \mathsf B_{\mathrm{ref}}^k(p,r,t,\mathbf z),&-2\delta<r<0,
    \end{cases}
\end{equation}
and
\begin{equation}\label{eq:doubled-target-nonlinearity}
    \mathbf Q_{\mathrm{dbl}}(p,r,t,\mathbf z,\boldsymbol\Xi)
    :=
    \begin{cases}
        \mathbf Q_c(p,r,t,\mathbf z,\boldsymbol\Xi),
        &0\le r<2\delta,\\
        \mathbf Q_{\mathrm{ref}}(p,r,t,\mathbf z,\boldsymbol\Xi),
        &-2\delta<r<0.
    \end{cases}
\end{equation}
The term $\mathbf Q_{\mathrm{dbl}}$ is quadratic in $\boldsymbol\Xi$,
including the $D^2\mathcal R_\Lambda$ contribution on the reflected side.
For a state field $\mathbf z$ with values in the tubular neighborhood, write
\begin{equation}\label{eq:doubled-geometric-operator}
    \mathscr L_{\mathbf z}
    :=\partial_t
      -\mathsf A_{\mathrm{dbl}}^{ij}(\mathbf z)\partial_{ij}
      -\mathsf B_{\mathrm{dbl}}^k(\mathbf z)\partial_k,
\end{equation}
where the spatial and temporal coefficient arguments are suppressed.

Define the reflected initial datum by
\begin{equation}\label{eq:doubled-initial-data}
    \mathbf v_0^{\mathrm{dbl}}(p,r)
    :=
    \begin{cases}
        \mathbf v_0(p,r),&0\le r<2\delta,\\
        \mathcal R_\Lambda(\mathbf v_0(p,-r)),
        &-2\delta<r<0.
    \end{cases}
\end{equation}
At $r=0$, \eqref{eq:zeroth-compatibility} implies that the trace belongs to
$\Lambda$ and is therefore fixed by $\mathcal R_\Lambda$.  If $\tau$ is tangent
to $\Gamma_0$, then
$D_\tau\mathbf v_0(p,0)\in T_{\mathbf v_0(p,0)}\Lambda$, and hence
\[
    D\mathcal R_\Lambda(\mathbf v_0)D_\tau\mathbf v_0
    =D_\tau\mathbf v_0.
\]
For the normal derivative, the spatial reflection contributes a minus sign,
while the conormal compatibility gives
$\partial_r\mathbf v_0\in T^{\perp,\mathrm{amb}}\Lambda$.  Thus
\[
    -D\mathcal R_\Lambda(\mathbf v_0)\partial_r\mathbf v_0
    =\partial_r\mathbf v_0.
\]
The function and all first derivatives therefore have matching one-sided
traces.  Set $s_q=2-2/q$.  Since $q>d+2\ge3$, one has
\[
    s_q>1+\frac1q.
\]
Smooth composition by $\mathcal R_\Lambda$ preserves $W_q^{s_q}$ on each
half-collar.  The fractional Sobolev gluing theorem across a smooth
hypersurface then applies: piecewise $W_q^{s_q}$ fields glue globally when their
value and first-derivative traces agree.  Locally, this follows by flattening
the hypersurface and using
\[
    W_q^{s_q}
    =
    \{f\in W_q^1:\partial_jf\in W_q^{s_q-1}\}.
\]
The value match removes the hypersurface measure in the first distributional
derivatives, and the derivative match glues those derivatives in
$W_q^{s_q-1}$ because $s_q-1>1/q$.  Therefore
\begin{equation}\label{eq:doubled-initial-Wq-trace}
    \mathbf v_0^{\mathrm{dbl}}
    \in
    W_q^{2-2/q}(\mathcal C_{\mathrm{dbl}};\mathbb R^{2L}).
\end{equation}

The nonlinear doubled problem is
\begin{equation}\label{eq:nonlinear-doubled-problem}
\begin{cases}
    \mathscr L_{\mathbf v}\mathbf v
    =
    \mathbf Q_{\mathrm{dbl}}(p,r,t,\mathbf v,\nabla\mathbf v),
    & -2\delta<r<2\delta,\\[0.3em]
    \mathbf v(p,2\delta,t)=\mathbf h(p,t),
    & r=2\delta,\\[0.3em]
    \mathbf v(p,-2\delta,t)
    =
    \mathcal R_\Lambda(\mathbf h(p,t)),
    & r=-2\delta,\\[0.3em]
    \mathbf v(p,r,0)=\mathbf v_0^{\mathrm{dbl}}(p,r).
\end{cases}
\end{equation}

\begin{lemma}[Reflection covariance of the doubled equation]
\label{lem:reflection-covariance}
Let $\mathbf v$ be a sufficiently regular field with values in
$\operatorname{Tub}_{\rho_\Lambda}(\Lambda)$ and define
\[
    \mathbf v^R(p,r,t)
    :=\mathcal R_\Lambda(\mathbf v(p,-r,t)).
\]
Then $\mathbf v$ satisfies the doubled equation at $(p,-r,t)$ if and only if
$\mathbf v^R$ satisfies it at $(p,r,t)$.  Moreover, the boundary values at
$r=\pm2\delta$ and the reflected initial datum are interchanged by this
operation.
\end{lemma}

\begin{proof}
Write $\rho(p,r,t)=(p,-r,t)$.  The spatial derivatives satisfy
\[
    \partial_i(\mathbf v\circ\rho)
    =\varepsilon_i(\partial_i\mathbf v)\circ\rho,
\]
\[
    \partial_{ij}(\mathbf v\circ\rho)
    =\varepsilon_i\varepsilon_j(\partial_{ij}\mathbf v)\circ\rho.
\]
For $\mathbf v^R=\mathcal R_\Lambda(\mathbf v\circ\rho)$, the chain rule gives
\begin{align*}
    \partial_t\mathbf v^R
    &=D\mathcal R_\Lambda(\mathbf v\circ\rho)
      (\partial_t\mathbf v)\circ\rho,\\
    \partial_k\mathbf v^R
    &=D\mathcal R_\Lambda(\mathbf v\circ\rho)
      \varepsilon_k(\partial_k\mathbf v)\circ\rho,\\
    \partial_{ij}\mathbf v^R
    &=D\mathcal R_\Lambda(\mathbf v\circ\rho)
      \varepsilon_i\varepsilon_j(\partial_{ij}\mathbf v)\circ\rho\\
    &\quad+
      D^2\mathcal R_\Lambda(\mathbf v\circ\rho)
      \bigl(\varepsilon_i(\partial_i\mathbf v)\circ\rho,
            \varepsilon_j(\partial_j\mathbf v)\circ\rho\bigr).
\end{align*}
Substitution into the equation on the physical half-collar $r>0$ produces
the principal, first-order, and nonlinear coefficients in
\eqref{eq:target-reflected-principal},
\eqref{eq:target-reflected-first-order}, and
\eqref{eq:target-reflected-nonlinearity}, respectively.  The involution identity
\eqref{eq:reflection-involution} gives the converse.  The statements about the
initial and boundary values follow directly from
\eqref{eq:doubled-initial-data} and the two Dirichlet conditions.
\end{proof}

\subsection{Bulk fixed-domain problems}
\label{subsec:bulk-problems}

Away from the interface, use the global pullback
\begin{equation}\label{eq:bulk-pullback}
    w^\sigma(y,t)
    :=u^\sigma(\Psi_t(y),t),
    \qquad y\in \mathcal D^\sigma.
\end{equation}
The initial data are
\[
    w_0^\sigma(y):=u_0^\sigma(y).
\]
Let
\[
    a_{b,ij}^\sigma(y,t)
    :=\partial_i\Psi_t(y)\cdot\partial_j\Psi_t(y),
    \qquad
    \lvert a_b^\sigma\rvert:=\det(a_{b,ij}^\sigma),
\]
and let $(a_b^{\sigma,ij})$ be its inverse matrix.  The pullback of the
Laplacian has divergence form
\[
    (\Delta_xu^\sigma)\circ\Psi_t
    =\lvert a_b^\sigma\rvert^{-1/2}
     \partial_i
     \bigl(\lvert a_b^\sigma\rvert^{1/2}
       a_b^{\sigma,ij}\partial_jw^\sigma\bigr).
\]
The time derivative contributes the transport term
\[
    a_b^{\sigma,k\ell}
    \left\langle
        \partial_t\Psi_t,
        \partial_\ell\Psi_t
    \right\rangle
    \partial_kw^\sigma.
\]
After expanding the divergence, the fixed bulk equation becomes
\begin{equation}\label{eq:bulk-equation}
    \partial_tw^\sigma
    -a_b^{\sigma,ij}(y,t)\partial_{ij}w^\sigma
    -b_b^{\sigma,k}(y,t)\partial_kw^\sigma
    =Q_b^\sigma(y,t,w^\sigma,\nabla w^\sigma),
\end{equation}
where
\begin{equation}\label{eq:bulk-coefficients}
\begin{aligned}
    b_b^{\sigma,k}(y,t)
    &:={}\lvert a_b^\sigma\rvert^{-1/2}
       \partial_i\bigl(\lvert a_b^\sigma\rvert^{1/2}
       a_b^{\sigma,ik}\bigr)
       +a_b^{\sigma,k\ell}
        \left\langle
            \partial_t\Psi_t,
            \partial_\ell\Psi_t
        \right\rangle,\\
    Q_b^\sigma(y,t,z,\Xi)
    &:={}
       a_b^{\sigma,ij}(y,t)
        \widetilde{\mathrm{II}}^\sigma(z)(\Xi_i,\Xi_j).
\end{aligned}
\end{equation}
Because $\Psi_t$ is defined on the whole torus, the formulas above also
define $a_b^{\sigma,ij}$ and $b_b^{\sigma,k}$ on all of $\Omega_0^\sigma$,
although the bulk Dirichlet problem uses them only on $\mathcal D^\sigma$.
The geometric coefficients are smooth, and $(a_b^{\sigma,ij})$ is
uniformly elliptic on the fixed compact cylinders.
The complete geometric operators are
\[
\begin{aligned}
    \mathscr L_c^\sigma
    &:=\partial_t-a_c^{\sigma,ij}\partial_{ij}
       -b_c^{\sigma,k}\partial_k,\\
    \mathscr L_b^\sigma
    &:=\partial_t-a_b^{\sigma,ij}\partial_{ij}
       -b_b^{\sigma,k}\partial_k.
\end{aligned}
\]

\begin{lemma}[Agreement of the geometric operators on the collar]
\label{lem:coefficient-agreement-collar}
Write $\phi^\sigma:=\Phi_0^\sigma$ on $0<r<2\delta$.  For every smooth
Euclidean-valued field $f(y,t)$,
\begin{equation}\label{eq:geometric-operator-collar}
    \mathscr L_c^\sigma(f\circ\phi^\sigma)
    =
    (\mathscr L_b^\sigma f)\circ\phi^\sigma.
\end{equation}
For every $C^1$ field $W$ with values in $\mathbb R^L$, set $V:=W\circ\phi^\sigma$.
Suppressing the spatial and temporal arguments,
\begin{equation}\label{eq:target-nonlinearity-collar}
    Q_b^\sigma(W,\nabla_yW)\circ\phi^\sigma
    =
    Q_c^\sigma(V,\nabla_{p,r}V).
\end{equation}
\end{lemma}

\begin{proof}
Let $U(x,t):=f(\Psi_t^{-1}(x),t)$.  Since
$\Psi_t\circ\phi^\sigma=\Phi_t^\sigma$, the pullback formulas
\eqref{eq:Fermi-Laplacian-formula} and
\eqref{eq:Fermi-time-derivative-formula} give
\[
    \mathscr L_c^\sigma(f\circ\phi^\sigma)
    =((\partial_t-\Delta_x)U)\circ\Phi_t^\sigma
    =(\mathscr L_b^\sigma f)\circ\phi^\sigma.
\]
The identity includes the first-order terms coming from derivatives of the
metric and from the transport velocity.
The chain rule for first derivatives and the transformation law for the
inverse metrics give \eqref{eq:target-nonlinearity-collar}.
The operator identity extends to $W_q^{2,1}$ fields by density, and the
nonlinearity identity holds for $C^1$ fields.
\end{proof}

Given an intermediate trace
$\mathbf h_\delta=(h_\delta^+,h_\delta^-)$ on
$\Gamma_0\times[0,T]$, the two bulk Dirichlet problems are
\begin{equation}\label{eq:bulk-Dirichlet-problem}
\begin{cases}
    \mathscr L_b^\sigma w^\sigma
    =Q_b^\sigma(y,t,w^\sigma,\nabla w^\sigma),
    &y\in \mathcal D^\sigma,\\[0.3em]
    w^\sigma(\Phi_0^\sigma(p,\delta),t)
    =h_\delta^\sigma(p,t),
    &p\in\Gamma_0,\\[0.3em]
    w^\sigma(y,0)=w_0^\sigma(y),
    &y\in \mathcal D^\sigma.
\end{cases}
\end{equation}
\section{Schwarz trace construction}
\label{sec:schwarz}

\subsection{State and trace spaces}
\label{subsec:state-trace-spaces}

Fix
\begin{equation}\label{eq:q-alpha-choice}
    q>d+2,
    \qquad
    0<\alpha<\beta<1-\frac{d+2}{q}.
\end{equation}
For $0<T\le T_0$, write
\[
    \mathcal C_{\mathrm{dbl},T}:=\mathcal C_{\mathrm{dbl}}\times(0,T),
    \qquad
    \mathcal D_T^\sigma:=\mathcal D^\sigma\times(0,T).
\]
Define
\[
\begin{aligned}
    \mathcal X_T
    :=
    &\ C^{1+\alpha,\frac{1+\alpha}{2}}
       (\overline{\mathcal C_{\mathrm{dbl}}}\times[0,T];\mathbb R^{2L})\\
    &\times
       C^{1+\alpha,\frac{1+\alpha}{2}}
       (\overline{\mathcal D^+}\times[0,T];\mathbb R^L)\\
    &\times
       C^{1+\alpha,\frac{1+\alpha}{2}}
       (\overline{\mathcal D^-}\times[0,T];\mathbb R^L).
\end{aligned}
\]
An element is denoted by
\[
    \widehat{\mathcal U}
    =
    (\widehat{\mathbf v},\widehat w^+,\widehat w^-).
\]
Let $\mathcal X_T^{\mathrm{match}}$ be the closed subspace satisfying the overlap
condition
\begin{equation}\label{eq:frozen-overlap-condition}
    \widehat{\mathbf v}(p,r,t)
    =
    \mathcal T_r(\widehat w^+,\widehat w^-)(p,t),
    \qquad
    \delta<r<2\delta.
\end{equation}
Use the natural product norm $\|\cdot\|_{\mathcal X_T}$.

Set
\begin{equation}\label{eq:initial-state}
    \mathcal U_0
    :=
    (\mathbf v_0^{\mathrm{dbl}},w_0^+,w_0^-)
    \in\mathcal X_T^{\mathrm{match}}
\end{equation}
and define the admissible state set
\begin{equation}\label{eq:admissible-state-set}
\begin{aligned}
    \mathcal A_T
    :=
    \Bigl\{
        \widehat{\mathcal U}\in\mathcal X_T^{\mathrm{match}}:\;&
        \widehat{\mathcal U}(\cdot,0)=\mathcal U_0,\\
        &\|\widehat{\mathcal U}-\mathcal U_0\|_{\mathcal X_T}\le1
    \Bigr\}.
\end{aligned}
\end{equation}
The initial trace at $r=2\delta$ is
\begin{equation}\label{eq:initial-Schwarz-trace}
    \mathbf h_0(p):=\mathbf v_0(p,2\delta).
\end{equation}
It belongs to
$W_q^{2-3/q}(\Gamma_0;\mathbb R^{2L})$.  The initial-trace lifting in
\Cref{lem:uniform-short-time-estimates} provides a fixed field
\begin{equation}\label{eq:initial-Schwarz-lifting}
    \boldsymbol\ell
    \in\mathcal Y_{T_0}(\Gamma_0;\mathbb R^{2L}),
    \qquad
    \boldsymbol\ell(\cdot,0)=\mathbf h_0,
    \qquad
    \|\boldsymbol\ell\|_{\mathcal Y_{T_0}}
    \le C\|\mathbf h_0\|_{W_q^{2-3/q}}.
\end{equation}
Unlike the time-independent function $\mathbf h_0$, this lifting has the
mixed time--space regularity required of a parabolic Dirichlet trace.  Choose
a radius $R_{\mathrm{tr}}>0$, independent of $T$, as in
\Cref{prop:Schwarz-self-map}, and set
\begin{equation}\label{eq:admissible-trace-set}
\begin{aligned}
    \mathcal K_T
    :=
    \Bigl\{
        \mathbf h\in\mathcal Y_T(\Gamma_0;\mathbb R^{2L}):\;&
        \mathbf h(\cdot,0)=\mathbf h_0,\\
        &\|\mathbf h-\boldsymbol\ell\|_{\mathcal Y_T(\Gamma_0;\mathbb R^{2L})}
          \le R_{\mathrm{tr}}
    \Bigr\}.
\end{aligned}
\end{equation}

On the overlap cylinder, define the continuous linear map
\[
    (\mathbf v,w^+,w^-)
    \longmapsto
    \mathbf v-\mathcal T_r(w^+,w^-).
\]
The matching space $\mathcal X_T^{\mathrm{match}}$ is the kernel of this map,
so it is a closed linear subspace. The initial state $\mathcal U_0$ belongs to it because
the collar and bulk initial data are restrictions of the same physical maps, so
$\mathcal A_T$ is nonempty.  It is the intersection of a closed
affine subspace with a closed ball, hence closed, bounded, and convex. The same
argument applies to $\mathcal K_T$, centered at the lifting $\boldsymbol\ell$
restricted to $(0,T)$.

\begin{lemma}[Uniform tubular control]
\label{lem:tubular-control}
The parameters $\rho_\Lambda$ and $\delta$ may be chosen successively so that
\eqref{eq:initial-paired-collar-tube} holds and
\begin{equation}\label{eq:Lambda-tube-inside-target-tubes}
    \operatorname{pr}_\sigma
    \bigl(\overline{\operatorname{Tub}_{\rho_\Lambda}(\Lambda)}\bigr)
    \Subset \operatorname{Tub}_{\rho_N/2}(N^\sigma),
    \qquad \sigma=\pm.
\end{equation}
Moreover, after decreasing $T>0$ if necessary, every
$\widehat{\mathcal U}\in\mathcal A_T$ and $\mathbf h\in\mathcal K_T$ satisfy
\begin{align}
    \widehat{\mathbf v}
    (\overline{\mathcal C_{\mathrm{dbl}}}\times[0,T])
    &\subset\operatorname{Tub}_{\rho_\Lambda}(\Lambda),
    \label{eq:frozen-Lambda-tube}\\
    \mathbf h(\Gamma_0\times[0,T])
    &\subset\operatorname{Tub}_{\rho_\Lambda}(\Lambda),
    \label{eq:trace-Lambda-tube}
\end{align}
and, for $\sigma=\pm$,
\begin{equation}\label{eq:frozen-target-tubes}
\begin{aligned}
    \widehat v^\sigma
    (\Gamma_0\times[0,2\delta]\times[0,T])
    &\Subset\operatorname{Tub}_{\rho_N/2}(N^\sigma),\\
    \widehat w^\sigma(\overline{\mathcal D^\sigma}\times[0,T])
    &\Subset\operatorname{Tub}_{\rho_N/2}(N^\sigma).
\end{aligned}
\end{equation}
All derivatives of $\mathcal R_\Lambda$ and the inverse differentials required
below are uniformly bounded on
$\overline{\operatorname{Tub}_{\rho_\Lambda}(\Lambda)}$, and hence on the ranges
of the admissible fields and traces.
\end{lemma}

\begin{proof}
Compactness of $\Lambda\subset N^+\times N^-$ permits the choice of
$\rho_\Lambda$ in \eqref{eq:Lambda-tube-inside-target-tubes}.  After this
choice, decreasing $\delta$ gives \eqref{eq:initial-paired-collar-tube}, which
ensures that the reflected initial datum is well-defined and that
\[
    \mathbf v_0^{\mathrm{dbl}}(\overline{\mathcal C_{\mathrm{dbl}}})
    \Subset\operatorname{Tub}_{\rho_\Lambda}(\Lambda).
\]
Every $\widehat{\mathcal U}\in\mathcal A_T$ satisfies
$\widehat{\mathcal U}(\cdot,0)=\mathcal U_0$ and
$\|\widehat{\mathcal U}-\mathcal U_0\|_{\mathcal X_T}\le1$ by
\eqref{eq:admissible-state-set}; since $\widehat{\mathcal U}-\mathcal U_0$
vanishes at $t=0$, its temporal H\"older seminorm gives
$\|\widehat{\mathcal U}-\mathcal U_0\|_{C^0}\le T^{(1+\alpha)/2}$.  Since the
initial fields lie in a compact subset of the open tubes, and the admissible
states stay within $T^{(1+\alpha)/2}$ of them in $C^0$, decreasing $T$ gives the
assertions for the frozen state.  For the trace, write
\[
    \mathbf h-\mathbf h_0
    =(\mathbf h-\boldsymbol\ell)
      +(\boldsymbol\ell-\mathbf h_0).
\]
The first term has zero initial trace and is controlled by
\eqref{eq:trace-C0-smallness}; the second tends to zero in $C^1$ as
$T\downarrow0$ because the fixed lifting belongs to
$C^{1+\beta,(1+\beta)/2}$, by the embedding of $\mathcal Y_{T_0}$ recorded in
\Cref{subsec:function-spaces}.  After decreasing $T$, we obtain
\eqref{eq:frozen-Lambda-tube}, \eqref{eq:trace-Lambda-tube}, and
\eqref{eq:frozen-target-tubes}.  By
\eqref{eq:reflection-invariant-subtube}, normal reflection preserves this
fixed tubular neighborhood, and
\[
    \overline{\operatorname{Tub}_{\rho_\Lambda}(\Lambda)}
    \Subset
    \operatorname{Tub}_{2\rho_\Lambda}(\Lambda),
\]
where $\mathcal R_\Lambda$ is smooth and
$[D\mathcal R_\Lambda]^{-1}=D\mathcal R_\Lambda\circ\mathcal R_\Lambda$ by
\eqref{eq:reflection-involution}; this gives the last assertion.
\end{proof}

\subsection{Frozen doubled coefficients}

For $0<\phi<\pi$, write
\[
    \mathfrak S_\phi:=\{\lambda\in\mathbb C\setminus\{0\}:|\arg\lambda|<\phi\}.
\]

\begin{proposition}[Uniform coefficient regularity and parameter ellipticity]
\label{prop:frozen-coefficients}
After decreasing $T>0$ if necessary, the coefficients
\[
    (p,r,t)\longmapsto
    \mathsf A_{\mathrm{dbl}}^{ij}
    (p,r,t,\widehat{\mathbf v}(p,r,t))
\]
are uniformly Lipschitz in the spatial reference variables and uniformly
$1/2$-H\"older in time, with constants independent of
$\widehat{\mathcal U}\in\mathcal A_T$.  Define the principal symbol
\[
    \mathscr P_{\widehat{\mathbf v}}(p,r,t,\xi)
    :=\mathsf A_{\mathrm{dbl}}^{ij}
      (p,r,t,\widehat{\mathbf v})\xi_i\xi_j.
\]
There exist $\theta\in(0,\pi/2)$ and $C_{\mathrm{pe}}>0$, independent of the
frozen state, such that
\begin{equation}\label{eq:doubled-parameter-ellipticity}
    \left\|
       \bigl(\lambda I+\mathscr P_{\widehat{\mathbf v}}
       (p,r,t,\xi)\bigr)^{-1}
    \right\|_{\mathcal L(\mathbb C^{2L})}
    \le
    \frac{C_{\mathrm{pe}}}{|\lambda|+|\xi|^2}
\end{equation}
for every $\lambda\in\overline{\mathfrak S}_{\pi-\theta}$ and
$(\lambda,\xi)\ne(0,0)$.  At the two artificial boundary components
$r=\pm2\delta$, the frozen principal operator together with the Dirichlet
boundary operator satisfies the Lopatinskii--Shapiro condition uniformly in
$\widehat{\mathcal U}$.
\end{proposition}

\begin{proof}
All geometric coefficients are fixed smooth functions of the reference
variables.  By \Cref{lem:tubular-control}, every frozen collar field takes
values in $\operatorname{Tub}_{\rho_\Lambda}(\Lambda)$, so the maps
$D\mathcal R_\Lambda$, $[D\mathcal R_\Lambda]^{-1}$, and all their derivatives
required below are uniformly bounded on the admissible ranges.
For $r>0$, the principal coefficient is independent of the frozen state.  For
$r<0$ it is
\begin{equation}\label{eq:negative-frozen-principal}
\begin{aligned}
    \mathsf A_{\mathrm{dbl}}^{ij}
    (p,r,t,\widehat{\mathbf v})
    ={}&
    [D\mathcal R_\Lambda(\widehat{\mathbf v})]^{-1}
    \circ\widetilde{\mathsf A}_c^{ij}(p,r,t)
    \circ D\mathcal R_\Lambda(\widehat{\mathbf v}).
\end{aligned}
\end{equation}
The product rule and the uniform
$C^{1+\alpha,(1+\alpha)/2}$ bound give the spatial Lipschitz and temporal $1/2$-H\"older bounds on each side of
$r=0$.

At $r=0$, both Fermi systems have the same tangential metric and satisfy
$a_c^{ar}=0$, $a_c^{rr}=1$.  Write
$a_{c,0}^{ij}(p,t):=a_c^{+,ij}(p,0,t)=a_c^{-,ij}(p,0,t)$; the subscript
$0$ denotes evaluation at $r=0$.  Hence their principal coefficients have the
common trace
\[
    a_{c,0}^{ij}(p,t)\operatorname{Id}_{\mathbb R^{2L}}.
\]
This scalar operator commutes with $D\mathcal R_\Lambda$, so the two one-sided
coefficients agree at $r=0$.  Piecewise Lipschitz functions with a common
trace on a smooth hypersurface are globally Lipschitz; this proves the first
assertion.

On the physical side, the symbol is the direct sum of the scalar Fermi
principal coefficients $a_c^{\sigma,ij}$ for the two phases.  Uniform
ellipticity gives constants
$0<c_0\le c_1$ such that every scalar eigenvalue $\mu$ lies in
$[c_0|\xi|^2,c_1|\xi|^2]$.  Therefore, after fixing
$\theta\in(0,\pi/2)$,
\[
    |\lambda+\mu|\ge c_\theta(|\lambda|+|\xi|^2),
    \qquad
    \lambda\in\overline{\mathfrak S}_{\pi-\theta}.
\]
On the reflected side, with
$S_\Lambda=D\mathcal R_\Lambda(\widehat{\mathbf v})$,
\[
    \mathscr P_{\widehat{\mathbf v}}(p,r,t,\xi)
    =S_\Lambda^{-1}\widetilde{\mathscr P}(p,r,t,\xi)S_\Lambda,
\]
where $\widetilde{\mathscr P}$ is the direct sum of the same scalar symbols
with $\xi_d$ replaced by $-\xi_d$.  The condition numbers of $S_\Lambda$ are uniformly bounded on
$\overline{\operatorname{Tub}_{\rho_\Lambda}(\Lambda)}$, and
\[
    (\lambda I+S_\Lambda^{-1}\widetilde{\mathscr P}S_\Lambda)^{-1}
    =S_\Lambda^{-1}(\lambda I+\widetilde{\mathscr P})^{-1}S_\Lambda.
\]
This proves \eqref{eq:doubled-parameter-ellipticity}.

It remains to verify the boundary symbol.  At $r=2\delta$, freezing the
coefficients gives a direct sum of scalar uniformly elliptic second-order
operators.  For a spectral parameter in
$\overline{\mathfrak S}_{\pi-\theta}$ and a tangential covector, the decaying
half-line solution of each scalar homogeneous symbol equation has one scalar
amplitude; homogeneous Dirichlet data force that amplitude to vanish.  At
$r=-2\delta$, a constant target change by $S_\Lambda$ reduces the frozen system to
the same direct sum, and homogeneous Dirichlet data are invariant under this
change.  This is the Dirichlet Lopatinskii--Shapiro condition, and uniformity
follows from compactness of the coefficient set and the uniform bounds on
$S_\Lambda$ and $S_\Lambda^{-1}$.
\end{proof}

The compact range and uniform first-derivative bounds of the frozen states give
uniform $L^\infty$ bounds for
$\mathsf B_{\mathrm{dbl}}^k(\widehat{\mathbf v})$,
$\mathbf Q_{\mathrm{dbl}}(\widehat{\mathbf v},
\nabla\widehat{\mathbf v})$, and
$Q_b^\sigma(\widehat w^\sigma,\nabla\widehat w^\sigma)$.
The first-order coefficients may jump at $r=0$; their boundedness is sufficient
for \Cref{lem:uniform-Wq-dirichlet}.  Under convergence in
$\mathcal X_T^{\mathrm{match}}$, the principal and first-order coefficients
converge in $L^\infty$, and the right-hand sides converge in $L^q$.

\subsection{Frozen problems and the Schwarz map}

Fix
\[
    \widehat{\mathcal U}
    =
    (\widehat{\mathbf v},\widehat w^+,\widehat w^-)
    \in\mathcal A_T.
\]
We use $\mathscr L_{\widehat{\mathbf v}}$ from
\eqref{eq:doubled-geometric-operator}: its coefficients are evaluated at
$\widehat{\mathbf v}$, while its derivatives act on the unknown.
For $\mathbf h\in\mathcal K_T$, solve the frozen doubled problem
\begin{equation}\label{eq:frozen-doubled-problem}
\begin{cases}
    \mathscr L_{\widehat{\mathbf v}}\mathbf v
    =
    \mathbf Q_{\mathrm{dbl}}
    (p,r,t,\widehat{\mathbf v},\nabla\widehat{\mathbf v}),
    & \mathcal C_{\mathrm{dbl},T},\\[0.3em]
    \mathbf v(p,2\delta,t)=\mathbf h(p,t),\\[0.3em]
    \mathbf v(p,-2\delta,t)
    =\mathcal R_\Lambda(\mathbf h(p,t)),\\[0.3em]
    \mathbf v(p,r,0)=\mathbf v_0^{\mathrm{dbl}}(p,r).
\end{cases}
\end{equation}
By \Cref{prop:frozen-coefficients,lem:uniform-Wq-dirichlet}, this problem has a
unique solution
\[
    \mathbf v^{\widehat{\mathcal U}}[\mathbf h]
    \in
    W_q^{2,1}(\mathcal C_{\mathrm{dbl},T};\mathbb R^{2L}).
\]
Define its intermediate trace
\begin{equation}\label{eq:intermediate-trace}
    \mathbf h_\delta^{\widehat{\mathcal U}}[\mathbf h]
    :=
    \mathbf v^{\widehat{\mathcal U}}[\mathbf h]
    \big|_{r=\delta}.
\end{equation}

Use this as boundary data for the frozen bulk problems
\begin{equation}\label{eq:frozen-bulk-problems}
\begin{cases}
    \mathscr L_b^\sigma w^\sigma
    =
    Q_b^\sigma
    (y,t,\widehat w^\sigma,\nabla\widehat w^\sigma),
    &\mathcal D_T^\sigma,\\[0.3em]
    w^\sigma(\Phi_0^\sigma(p,\delta),t)
    =
    \bigl(
        \mathbf h_\delta^{\widehat{\mathcal U}}[\mathbf h]
    \bigr)^\sigma(p,t),\\[0.3em]
    w^\sigma(y,0)=w_0^\sigma(y).
\end{cases}
\end{equation}
The solutions are denoted by
\[
    w^{\sigma,\widehat{\mathcal U}}[\mathbf h]
    \in
    W_q^{2,1}(\mathcal D_T^\sigma;\mathbb R^L).
\]

\begin{definition}[Schwarz trace map]
\label{def:Schwarz-map}
For fixed $\widehat{\mathcal U}\in\mathcal A_T$, define
\begin{equation}\label{eq:Schwarz-map}
\begin{aligned}
    \mathcal S_T^{\widehat{\mathcal U}}(\mathbf h)
    :=
    \mathcal T_{2\delta}
    \Bigl(
        w^{+,\widehat{\mathcal U}}[\mathbf h],
        w^{-,\widehat{\mathcal U}}[\mathbf h]
    \Bigr).
\end{aligned}
\end{equation}
Thus
\[
    \mathbf h
    \longmapsto
    \mathbf v^{\widehat{\mathcal U}}[\mathbf h]
    \longmapsto
    \mathbf h_\delta^{\widehat{\mathcal U}}[\mathbf h]
    \longmapsto
    (w^{+,\widehat{\mathcal U}}[\mathbf h],
     w^{-,\widehat{\mathcal U}}[\mathbf h])
    \longmapsto
    \mathcal S_T^{\widehat{\mathcal U}}(\mathbf h).
\]
\end{definition}

\subsection{An off-diagonal boundary-to-interior estimate}

\begin{lemma}[Boundary-to-interior short-time trace estimate]
\label{lem:off-diagonal-trace}
Let $D$ be a fixed compact smooth manifold with boundary, and let
$\Sigma\Subset D^\circ$ be a smooth compact hypersurface.  Let $E$ be a
finite-dimensional space and fix $q>d+2$.  Set
\[
    d_\Sigma:=\operatorname{dist}(\Sigma,\partial D)>0.
\]
Fix $T_0>0$, and suppose that the coefficients below are defined on
$\overline D\times[0,T_0]$.  Consider the componentwise scalar operator
\[
    \mathscr L u
    :=
    \partial_tu
    -a^{ij}(x,t)\partial_{ij}u
    -b^i(x,t)\partial_iu
    -c(x,t)u,
\]
where $a^{ij}=a^{ji}$, the matrix $(a^{ij})$ is uniformly elliptic, and the
coefficients satisfy the uniform regularity and boundedness assumptions of
\Cref{lem:uniform-Wq-dirichlet} on $\overline D\times[0,T_0]$.

For $0<T\le T_0$, let $u\in W_q^{2,1}(D_T;E)$ solve
\begin{equation}\label{eq:off-diagonal-problem}
\begin{cases}
    \mathscr L u=0,&D\times(0,T),\\
    u=h,&\partial D\times(0,T),\\
    u(\cdot,0)=0,&D,
\end{cases}
\end{equation}
where
\[
    h\in\mathcal Y_T(\partial D;E),
    \qquad
    h(\cdot,0)=0.
\]
Then
\begin{equation}\label{eq:off-diagonal-trace-estimate}
    \|u|_{\Sigma\times(0,T)}\|_{\mathcal Y_T(\Sigma;E)}
    \le
    CT^{1/2}
    \|h\|_{\mathcal Y_T(\partial D;E)}.
\end{equation}
Here $C$ is independent of $T\in(0,T_0]$.  It may depend on $D$, $\Sigma$,
$T_0$, $q$, the uniform coefficient bounds, the ellipticity constants, and,
through the cutoff below, on $d_\Sigma^{-1}$.
\end{lemma}

\begin{proof}
By the uniform inhomogeneous Dirichlet estimate from
\Cref{lem:uniform-Wq-dirichlet},
\begin{equation}\label{eq:off-diagonal-global-bound}
    \|u\|_{W_q^{2,1}(D_T;E)}
    \le
    C\|h\|_{\mathcal Y_T(\partial D;E)},
\end{equation}
where the constant is independent of $T\in(0,T_0]$.

Choose open sets
\[
    \Sigma\Subset V_1\Subset V_2\Subset D^\circ
\]
and a cutoff $\chi\in C_c^\infty(V_2)$ satisfying
$\chi\equiv1$ on $V_1$.  The sets and the cutoff are fixed independently of
$T$.  They may be chosen so that, with respect to a fixed finite atlas on $D$,
\[
    \|\nabla\chi\|_{L^\infty(D)}
    \le Cd_\Sigma^{-1},
    \qquad
    \|\nabla^2\chi\|_{L^\infty(D)}
    \le Cd_\Sigma^{-2}.
\]
Set $v:=\chi u$.  Since $\chi$ is compactly supported in $D^\circ$ and
$u(\cdot,0)=0$,
\[
    v|_{\partial D\times(0,T)}=0,
    \qquad
    v(\cdot,0)=0.
\]
Using $\mathscr Lu=0$, we obtain
\[
    \mathscr Lv=[\mathscr L,\chi]u,
    \qquad
    [\mathscr L,\chi]u
    :=\mathscr L(\chi u)-\chi\mathscr Lu.
\]
In any local coordinate chart, symmetry of $a^{ij}$ gives
\begin{equation}\label{eq:off-diagonal-commutator}
    [\mathscr L,\chi]u
    =-2a^{ij}(\partial_i\chi)\partial_ju
     -\bigl(a^{ij}\partial_{ij}\chi+b^i\partial_i\chi\bigr)u.
\end{equation}
The commutator therefore contains only first- and zeroth-order derivatives of
$u$, and a finite-atlas argument gives
\begin{equation}\label{eq:off-diagonal-commutator-first-bound}
    \|[\mathscr L,\chi]u\|_{L^q(D_T;E)}
    \le
    C_\chi\bigl(
        \|\nabla u\|_{L^q(D_T;E)}
        +\|u\|_{L^q(D_T;E)}
    \bigr),
\end{equation}
where $C_\chi$ is independent of $T$.

Because $u(\cdot,0)=0$, \eqref{eq:zero-initial-Lq} and
\eqref{eq:zero-initial-grad-Lq} imply
\[
    \|u\|_{L^q(D_T;E)}
    \le T\|u\|_{W_q^{2,1}(D_T;E)}
\]
and
\[
    \|\nabla u\|_{L^q(D_T;E)}
    \le CT^{1/2}\|u\|_{W_q^{2,1}(D_T;E)}.
\]
Substitution into \eqref{eq:off-diagonal-commutator-first-bound} yields
\begin{align}
    \|[\mathscr L,\chi]u\|_{L^q(D_T;E)}
    &\le
    C_\chi(T^{1/2}+T)\|u\|_{W_q^{2,1}(D_T;E)}
    \notag\\
    &\le
    C_\chi(1+T_0^{1/2})T^{1/2}
    \|u\|_{W_q^{2,1}(D_T;E)}.
    \label{eq:off-diagonal-commutator-smallness}
\end{align}
Together with \eqref{eq:off-diagonal-global-bound}, this gives
\begin{equation}\label{eq:off-diagonal-forcing-smallness}
    \|[\mathscr L,\chi]u\|_{L^q(D_T;E)}
    \le
    CT^{1/2}\|h\|_{\mathcal Y_T(\partial D;E)}.
\end{equation}

The function $v$ solves
\[
\begin{cases}
    \mathscr Lv=[\mathscr L,\chi]u,&D\times(0,T),\\
    v=0,&\partial D\times(0,T),\\
    v(\cdot,0)=0,&D.
\end{cases}
\]
A second application of \Cref{lem:uniform-Wq-dirichlet} therefore gives
\begin{equation}\label{eq:off-diagonal-cutoff-Wq}
    \|v\|_{W_q^{2,1}(D_T;E)}
    \le C\|[\mathscr L,\chi]u\|_{L^q(D_T;E)}
    \le CT^{1/2}\|h\|_{\mathcal Y_T(\partial D;E)}.
\end{equation}
Since $\chi\equiv1$ near $\Sigma$, one has $v=u$ in a neighborhood of
$\Sigma\times(0,T)$.  The uniform interior hypersurface trace estimate
\eqref{eq:uniform-hypersurface-trace} consequently yields
\[
\begin{aligned}
    \|u|_{\Sigma\times(0,T)}\|_{\mathcal Y_T(\Sigma;E)}
    &=\|v|_{\Sigma\times(0,T)}\|_{\mathcal Y_T(\Sigma;E)}\\
    &\le C\|v\|_{W_q^{2,1}(D_T;E)}\\
    &\le CT^{1/2}\|h\|_{\mathcal Y_T(\partial D;E)}.
\end{aligned}
\]
This proves \eqref{eq:off-diagonal-trace-estimate}.
\end{proof}

\subsection{Schwarz contraction and self-map property}

\begin{lemma}[Uniform Schwarz trace contraction]
\label{lem:Schwarz-contraction}
There exists a function $\kappa(T)\to0$ as $T\downarrow0$, independent of
$\widehat{\mathcal U}\in\mathcal A_T$, such that
\begin{equation}\label{eq:Schwarz-contraction}
    \|\mathcal S_T^{\widehat{\mathcal U}}(\mathbf h)
      -\mathcal S_T^{\widehat{\mathcal U}}(\widetilde{\mathbf h})
    \|_{\mathcal Y_T(\Gamma_0;\mathbb R^{2L})}
    \le
    \kappa(T)
    \|\mathbf h-\widetilde{\mathbf h}
    \|_{\mathcal Y_T(\Gamma_0;\mathbb R^{2L})}
\end{equation}
for all $\mathbf h,\widetilde{\mathbf h}\in\mathcal K_T$.
\end{lemma}

\begin{proof}
Let
\[
    \mathbf z_{\mathrm{dbl}}
    :=
    \mathbf v^{\widehat{\mathcal U}}[\mathbf h]
    -
    \mathbf v^{\widehat{\mathcal U}}[\widetilde{\mathbf h}].
\]
The frozen coefficients and frozen forcing are identical for the two inputs,
so $\mathbf z_{\mathrm{dbl}}$ solves
\begin{equation}\label{eq:inner-difference-Schwarz}
\begin{cases}
    \mathscr L_{\widehat{\mathbf v}}\mathbf z_{\mathrm{dbl}}=0,
    &\mathcal C_{\mathrm{dbl},T},\\[0.3em]
    \mathbf z_{\mathrm{dbl}}|_{r=2\delta}
    =\mathbf h-\widetilde{\mathbf h},\\[0.3em]
    \mathbf z_{\mathrm{dbl}}|_{r=-2\delta}
    =\mathcal R_\Lambda(\mathbf h)
     -\mathcal R_\Lambda(\widetilde{\mathbf h}),\\[0.3em]
    \mathbf z_{\mathrm{dbl}}(\cdot,0)=0.
\end{cases}
\end{equation}
The ranges of both traces lie in one compact subset of the tubular
neighborhood.  By \Cref{lem:nemytskii-trace},
\begin{equation}\label{eq:reflection-trace-Lipschitz}
    \|\mathcal R_\Lambda(\mathbf h)
      -\mathcal R_\Lambda(\widetilde{\mathbf h})
    \|_{\mathcal Y_T}
    \le C
    \|\mathbf h-\widetilde{\mathbf h}\|_{\mathcal Y_T}.
\end{equation}
Uniform Dirichlet maximal regularity applied to
\eqref{eq:inner-difference-Schwarz} gives
\[
    \|\mathbf z_{\mathrm{dbl}}\|_{W_q^{2,1}(\mathcal C_{\mathrm{dbl},T})}
    \le C
    \|\mathbf h-\widetilde{\mathbf h}\|_{\mathcal Y_T}.
\]
Taking the trace at $r=\delta$ yields
\begin{equation}\label{eq:inner-trace-Lipschitz}
    \|\mathbf z_{\mathrm{dbl}}|_{r=\delta}\|_{\mathcal Y_T(\Gamma_0;\mathbb R^{2L})}
    \le C
    \|\mathbf h-\widetilde{\mathbf h}\|_{\mathcal Y_T(\Gamma_0;\mathbb R^{2L})}.
\end{equation}

For each phase set
\[
    z^\sigma
    :=w^{\sigma,\widehat{\mathcal U}}[\mathbf h]
      -w^{\sigma,\widehat{\mathcal U}}[\widetilde{\mathbf h}].
\]
The frozen bulk forcing cancels.  Thus
\begin{equation}\label{eq:bulk-difference-Schwarz}
\begin{cases}
    \mathscr L_b^\sigma z^\sigma=0,
    &\mathcal D_T^\sigma,\\[0.3em]
    z^\sigma|_{\Sigma_\delta^\sigma}
    =(\mathbf z_{\mathrm{dbl}}|_{r=\delta})^\sigma,\\[0.3em]
    z^\sigma(\cdot,0)=0.
\end{cases}
\end{equation}
The output surface $\Sigma_{2\delta}^\sigma$ is compactly contained in
$\mathcal D^\sigma$.  Applying \Cref{lem:off-diagonal-trace} gives
\begin{equation}\label{eq:bulk-off-diagonal-smallness}
    \|z^\sigma|_{\Sigma_{2\delta}^\sigma}\|_{\mathcal Y_T}
    \le CT^{1/2}
    \|(\mathbf z_{\mathrm{dbl}}|_{r=\delta})^\sigma\|_{\mathcal Y_T}.
\end{equation}
Combining the two phases and using \eqref{eq:inner-trace-Lipschitz},
\[
    \|\mathcal S_T^{\widehat{\mathcal U}}(\mathbf h)
      -\mathcal S_T^{\widehat{\mathcal U}}(\widetilde{\mathbf h})
    \|_{\mathcal Y_T}
    \le CT^{1/2}
    \|\mathbf h-\widetilde{\mathbf h}\|_{\mathcal Y_T}.
\]
Hence one may take $\kappa(T)=CT^{1/2}$.
\end{proof}

\begin{proposition}[Self-map property of the Schwarz trace operator]
\label{prop:Schwarz-self-map}
For sufficiently small $T>0$,
\[
    \mathcal S_T^{\widehat{\mathcal U}}:
    \mathcal K_T\longrightarrow\mathcal K_T
\]
for every $\widehat{\mathcal U}\in\mathcal A_T$.  By
\Cref{lem:Schwarz-contraction}, after decreasing $T$ further this map is also a
contraction, with contraction constant $\kappa(T)<1$ uniform in the frozen
state.
\end{proposition}

\begin{proof}
Use the parabolic lifting $\boldsymbol\ell$ from
\eqref{eq:initial-Schwarz-lifting} as the comparison trace.  Its restrictions
to $(0,T)$ have uniformly bounded $\mathcal Y_T$ norms.  The data in the
frozen doubled problem with boundary trace $\boldsymbol\ell$ therefore have
uniformly bounded maximal-regularity norms: the initial datum is measured in
$W_q^{2-2/q}$, not in $W_q^2$.  The same estimate for the two bulk problems,
followed by the interior trace theorem, yields a constant $C_*$ independent
of $T\in(0,T_0]$ and of the frozen state such that
\begin{equation}\label{eq:reference-Schwarz-bound}
    \sup_{\widehat{\mathcal U}\in\mathcal A_T}
    \|\mathcal S_T^{\widehat{\mathcal U}}(\boldsymbol\ell)
      -\boldsymbol\ell\|_{\mathcal Y_T}
    \le C_*.
\end{equation}
This estimate depends only on the fixed lifting and the admissible state set,
so $C_*$ can be chosen before the radius of $\mathcal K_T$.

Fix, once and for all,
\[
    R_{\mathrm{tr}}:=2C_*+1.
\]
For an arbitrary $\mathbf h\in\mathcal K_T$, the contraction estimate gives
\begin{align*}
    \|\mathcal S_T^{\widehat{\mathcal U}}(\mathbf h)
      -\boldsymbol\ell\|_{\mathcal Y_T}
    &\le
    \|\mathcal S_T^{\widehat{\mathcal U}}(\mathbf h)
      -\mathcal S_T^{\widehat{\mathcal U}}(\boldsymbol\ell)\|_{\mathcal Y_T}\\
    &\quad+
    \|\mathcal S_T^{\widehat{\mathcal U}}(\boldsymbol\ell)
      -\boldsymbol\ell\|_{\mathcal Y_T}\\
    &\le
    \kappa(T)R_{\mathrm{tr}}+C_*.
\end{align*}
Choose $T$ so that
$\kappa(T)R_{\mathrm{tr}}+C_*\le R_{\mathrm{tr}}$.  At $t=0$, the
output equals
$\mathcal T_{2\delta}(w_0^+,w_0^-)=\mathbf v_0(\cdot,2\delta)=\mathbf h_0$,
so the output belongs to $\mathcal K_T$.  Taking $T$ smaller if necessary gives
$\kappa(T)<1$ and proves the contraction assertion.
\end{proof}

\subsection{Trace fixed point and overlap gluing}

\begin{proposition}[Trace fixed point and gluing]
\label{prop:trace-fixed-gluing}
For every fixed $\widehat{\mathcal U}\in\mathcal A_T$, there exists a unique
trace
\[
    \mathbf h^{\widehat{\mathcal U}}\in\mathcal K_T
\]
such that
\begin{equation}\label{eq:trace-fixed-point}
    \mathbf h^{\widehat{\mathcal U}}
    =
    \mathcal S_T^{\widehat{\mathcal U}}
    (\mathbf h^{\widehat{\mathcal U}}).
\end{equation}
The collar and bulk solutions generated by this trace agree throughout the
positive overlap $\delta<r<2\delta$.
\end{proposition}

\begin{proof}
The existence and uniqueness of the trace follow from Banach's contraction
theorem.  Set
\[
    \mathbf v^{\widehat{\mathcal U}}
    :=\mathbf v^{\widehat{\mathcal U}}
      [\mathbf h^{\widehat{\mathcal U}}],
\]
\[
    \mathbf h_\delta^{\widehat{\mathcal U}}
    :=\mathbf v^{\widehat{\mathcal U}}|_{r=\delta},
\]
and
\[
    w^{\sigma,\widehat{\mathcal U}}
    :=w^{\sigma,\widehat{\mathcal U}}
      [\mathbf h^{\widehat{\mathcal U}}].
\]
The definition of the bulk boundary condition gives matching at the inner
edge of the overlap:
\begin{equation}\label{eq:matching-at-delta}
    \mathbf v^{\widehat{\mathcal U}}|_{r=\delta}
    =\mathcal T_\delta
      (w^{+,\widehat{\mathcal U}},
       w^{-,\widehat{\mathcal U}}).
\end{equation}
The fixed-point identity gives matching at the outer edge:
\begin{equation}\label{eq:matching-at-2delta}
    \mathbf v^{\widehat{\mathcal U}}|_{r=2\delta}
    =\mathcal T_{2\delta}
      (w^{+,\widehat{\mathcal U}},
       w^{-,\widehat{\mathcal U}}).
\end{equation}

On the overlap define the two-phase bulk field
\[
    \mathbf v_{\mathrm b}^{\widehat{\mathcal U}}(p,r,t)
    :=\mathcal T_r
      (w^{+,\widehat{\mathcal U}},
       w^{-,\widehat{\mathcal U}})(p,t).
\]
By the overlap condition defining $\mathcal X_T^{\mathrm{match}}$,
\[
    \widehat v^\sigma
    =\widehat w^\sigma\circ\Phi_0^\sigma
    \qquad(\delta<r<2\delta).
\]
Restricting \eqref{eq:frozen-doubled-problem} to this overlap, its
$\sigma$-component satisfies
$\mathscr L_c^\sigma v^{\sigma,\widehat{\mathcal U}}
=Q_c^\sigma(\widehat v^\sigma,\nabla\widehat v^\sigma)$.
Applying \eqref{eq:geometric-operator-collar} and
\eqref{eq:target-nonlinearity-collar} to the frozen bulk equation gives
\[
\begin{aligned}
    \mathscr L_c^\sigma
       (\mathbf v_{\mathrm b}^{\widehat{\mathcal U}})^\sigma
    &=
    Q_b^\sigma(\widehat w^\sigma,
                      \nabla\widehat w^\sigma)\circ\Phi_0^\sigma\\
    &=Q_c^\sigma(\widehat v^\sigma,\nabla\widehat v^\sigma).
\end{aligned}
\]
Thus the difference
$e^\sigma:=v^{\sigma,\widehat{\mathcal U}}
-(\mathbf v_{\mathrm b}^{\widehat{\mathcal U}})^\sigma$
satisfies
\[
    \partial_te^\sigma-a_c^{\sigma,ij}\partial_{ij}e^\sigma
       -b_c^{\sigma,k}\partial_ke^\sigma=0
    \quad\text{on }\Gamma_0\times(\delta,2\delta)\times(0,T).
\]
Its initial data vanish because both fields are obtained from the same
initial physical maps.  Its Dirichlet traces vanish at $r=\delta$ and
$r=2\delta$ by
\eqref{eq:matching-at-delta} and \eqref{eq:matching-at-2delta}.

Uniqueness from \Cref{lem:uniform-Wq-dirichlet} gives
\begin{equation}\label{eq:overlap-gluing}
    \mathbf v^{\widehat{\mathcal U}}
    =\mathbf v_{\mathrm b}^{\widehat{\mathcal U}}
    \qquad
    \text{on }\Gamma_0\times(\delta,2\delta)\times(0,T).
    \qedhere
\end{equation}
\end{proof}

Define the state map by
\begin{equation}\label{eq:state-map}
    \mathcal F_T(\widehat{\mathcal U})
    :=
    \bigl(
       \mathbf v^{\widehat{\mathcal U}},
       w^{+,\widehat{\mathcal U}},
       w^{-,\widehat{\mathcal U}}
    \bigr).
\end{equation}
By \eqref{eq:overlap-gluing}, its range belongs to
$\mathcal X_T^{\mathrm{match}}$.

\section{Schauder fixed point and recovery}
\label{sec:schauder}

\subsection{Self-map property of \texorpdfstring{$\mathcal F_T$}{FT}}

\begin{proposition}[State self-map]
\label{prop:state-self-map}
After decreasing $T>0$,
\[
    \mathcal F_T(\mathcal A_T)\subset\mathcal A_T.
\]
\end{proposition}

\begin{proof}
Fix a frozen state $\widehat{\mathcal U}\in\mathcal A_T$, and let
$\mathbf h^{\widehat{\mathcal U}}$, $\mathbf v^{\widehat{\mathcal U}}$, and
$w^{\sigma,\widehat{\mathcal U}}$ be generated by the trace fixed point.  The
output has the prescribed initial value and satisfies the overlap condition by
\Cref{prop:trace-fixed-gluing}.  It remains to estimate its distance from the
initial state in $\mathcal X_T$.

The fixed trace satisfies
\[
    \|\mathbf h^{\widehat{\mathcal U}}\|_{\mathcal Y_T}
    \le
    \|\boldsymbol\ell\|_{\mathcal Y_{T_0}}+R_{\mathrm{tr}}.
\]
The frozen nonlinearities are uniformly bounded on the admissible set, and
\Cref{lem:nemytskii-trace} gives a uniform $\mathcal Y_T$ bound for the
reflected boundary value
$\mathcal R_\Lambda(\mathbf h^{\widehat{\mathcal U}})$.
Uniform Dirichlet maximal regularity, first on the doubled collar and then on
the two bulk domains, consequently gives
\begin{equation}\label{eq:uniform-output-maximal-regularity}
\begin{aligned}
    \|\mathbf v^{\widehat{\mathcal U}}\|_{W_q^{2,1}}
    +\sum_{\sigma=\pm}
      \|w^{\sigma,\widehat{\mathcal U}}\|_{W_q^{2,1}}
    \le M_*.
\end{aligned}
\end{equation}
Here $M_*$ is independent of $T$ and of
$\widehat{\mathcal U}\in\mathcal A_T$.

Every output component has the corresponding component of $\mathcal U_0$ as
its initial trace.  Apply the initial-trace short-time estimate
\eqref{eq:initial-trace-holder-smallness}, with the exponents $\alpha<\beta$
chosen in \eqref{eq:q-alpha-choice}, to
\eqref{eq:uniform-output-maximal-regularity}.  We obtain
\begin{equation}\label{eq:state-output-holder-smallness}
    \|\mathcal F_T(\widehat{\mathcal U})-\mathcal U_0\|_{\mathcal X_T}
    \le C T^\eta
       \left(M_*+\|\mathcal U_0\|_{W_q^{2-2/q}}\right)
    =:\omega_*(T),
\end{equation}
for some $\eta>0$, uniformly in the frozen state.  Since
$\omega_*(T)\to0$, choose $T$ so that $\omega_*(T)\le1$.  The initial
condition and overlap matching have already been verified; hence
$\mathcal F_T(\widehat{\mathcal U})\in\mathcal A_T$.
\end{proof}

\subsection{Compactness, continuity, and the Schauder fixed point}

\begin{proposition}[Compactness and continuity of the state map]
\label{prop:state-map-compact-continuous}
For sufficiently small $T$, the map
\[
    \mathcal F_T:\mathcal A_T\longrightarrow\mathcal A_T
\]
is compact and continuous.
\end{proposition}

\begin{proof}
The state self-map estimates provide a uniform $W_q^{2,1}$ bound for the
doubled collar component and for both bulk components.  Since
$q>d+2$ and
\[
    0<\alpha<1-\frac{d+2}{q},
\]
the compact parabolic embedding on each fixed compact cylinder gives
\[
    W_q^{2,1}
    \Subset
    C^{1+\alpha,(1+\alpha)/2}.
\]
Hence $\mathcal F_T(\mathcal A_T)$ is relatively compact in
$\mathcal X_T$.  The overlap condition defining $\mathcal X_T^{\mathrm{match}}$ is
closed, so the range is relatively compact in $\mathcal X_T^{\mathrm{match}}$.

It remains to prove continuity.  Let
\[
    \widehat{\mathcal U}_m
    =
    (\widehat{\mathbf v}_m,
     \widehat w_m^+,
     \widehat w_m^-)
    \longrightarrow
    \widehat{\mathcal U}
    =
    (\widehat{\mathbf v},
     \widehat w^+,
     \widehat w^-)
\]
in $\mathcal X_T^{\mathrm{match}}$. We first establish continuous dependence of
the frozen solutions for a fixed Schwarz trace.
In the coefficient expressions below, the spatial and temporal arguments
are suppressed as in \eqref{eq:doubled-geometric-operator}.
Write $\mathscr L_m:=\mathscr L_{\widehat{\mathbf v}_m}$
and
\[
    \mathbf f_m
    :=
    \mathbf Q_{\mathrm{dbl}}
    (p,r,t,\widehat{\mathbf v}_m,
     \nabla\widehat{\mathbf v}_m),
    \qquad
    \mathbf f
    :=
    \mathbf Q_{\mathrm{dbl}}
    (p,r,t,\widehat{\mathbf v},
     \nabla\widehat{\mathbf v}).
\]
All frozen collar fields take values in a common compact subset of the tubular
neighborhood of $\Lambda$.  The smooth dependence of the coefficients on the
state on each half-collar, together with convergence in
$C^{1+\alpha,(1+\alpha)/2}$, implies
\begin{equation}\label{eq:frozen-coefficient-convergence}
    \|\mathsf A_{\mathrm{dbl}}^{ij}(\widehat{\mathbf v}_m)
      -\mathsf A_{\mathrm{dbl}}^{ij}(\widehat{\mathbf v})\|_{L^\infty}
    +
    \|\mathsf B_{\mathrm{dbl}}^k(\widehat{\mathbf v}_m)
      -\mathsf B_{\mathrm{dbl}}^k(\widehat{\mathbf v})\|_{L^\infty}
    +
    \|\mathbf f_m-\mathbf f\|_{L^q}
    \longrightarrow0.
\end{equation}

Fix $\mathbf h\in\mathcal K_T$.  Let
\[
    \mathbf v_m^{\mathbf h}
    :=\mathbf v^{\widehat{\mathcal U}_m}[\mathbf h],
    \qquad
    \mathbf v^{\mathbf h}
    :=\mathbf v^{\widehat{\mathcal U}}[\mathbf h].
\]
Their initial and boundary data coincide.  The difference
$\mathbf z_m^{\mathbf h}:=\mathbf v_m^{\mathbf h}-\mathbf v^{\mathbf h}$
has zero initial and boundary data and satisfies
\[
\begin{aligned}
    \mathscr L_m\mathbf z_m^{\mathbf h}
    =
    {}&\bigl(\mathsf A_{\mathrm{dbl}}^{ij}(\widehat{\mathbf v}_m)
       -\mathsf A_{\mathrm{dbl}}^{ij}(\widehat{\mathbf v})\bigr)
       \partial_{ij}\mathbf v^{\mathbf h}\\
    &+\bigl(\mathsf B_{\mathrm{dbl}}^k(\widehat{\mathbf v}_m)
       -\mathsf B_{\mathrm{dbl}}^k(\widehat{\mathbf v})\bigr)
       \partial_k\mathbf v^{\mathbf h}
     +(\mathbf f_m-\mathbf f).
\end{aligned}
\]
Uniform Dirichlet maximal regularity gives
\begin{align*}
    \|\mathbf z_m^{\mathbf h}\|_{W_q^{2,1}}
    \le C\Bigl(&
      \|\mathsf A_{\mathrm{dbl}}^{ij}(\widehat{\mathbf v}_m)
        -\mathsf A_{\mathrm{dbl}}^{ij}(\widehat{\mathbf v})\|_{L^\infty}
      \|\partial_{ij}\mathbf v^{\mathbf h}\|_{L^q}\\
      &+\|\mathsf B_{\mathrm{dbl}}^k(\widehat{\mathbf v}_m)
        -\mathsf B_{\mathrm{dbl}}^k(\widehat{\mathbf v})\|_{L^\infty}
      \|\partial_k\mathbf v^{\mathbf h}\|_{L^q}
      +\|\mathbf f_m-\mathbf f\|_{L^q}
    \Bigr).
\end{align*}
Thus
\begin{equation}\label{eq:fixed-trace-collar-continuity}
    \mathbf v_m^{\mathbf h}
    \longrightarrow
    \mathbf v^{\mathbf h}
    \quad\text{in }W_q^{2,1}(\mathcal C_{\mathrm{dbl},T};\mathbb R^{2L}).
\end{equation}
Taking the trace at $r=\delta$ yields
\begin{equation}\label{eq:intermediate-trace-continuity}
    \mathbf h_{\delta,m}^{\mathbf h}
    \longrightarrow
    \mathbf h_\delta^{\mathbf h}
    \quad\text{in }\mathcal Y_T(\Gamma_0;\mathbb R^{2L}).
\end{equation}

For the bulk equations set
\[
    f_m^{\sigma,\mathrm b}
    :=
    Q_b^\sigma
    (y,t,\widehat w_m^\sigma,
     \nabla\widehat w_m^\sigma),
    \qquad
    f^{\sigma,\mathrm b}
    :=
    Q_b^\sigma
    (y,t,\widehat w^\sigma,
     \nabla\widehat w^\sigma).
\]
Then
\[
    f_m^{\sigma,\mathrm b}
    \longrightarrow f^{\sigma,\mathrm b}
    \quad\text{in }L^q(\mathcal D_T^\sigma).
\]
Let
\[
    w_m^{\sigma,\mathbf h}
    :=w^{\sigma,\widehat{\mathcal U}_m}[\mathbf h],
    \qquad
    w^{\sigma,\mathbf h}
    :=w^{\sigma,\widehat{\mathcal U}}[\mathbf h].
\]
Their difference solves the equation with the fixed operator
$\mathscr L_b^\sigma$, zero initial data, forcing
$f_m^{\sigma,\mathrm b}-f^{\sigma,\mathrm b}$, and boundary datum equal to
the $\sigma$-component of
$\mathbf h_{\delta,m}^{\mathbf h}-\mathbf h_\delta^{\mathbf h}$.  Hence
\begin{align*}
    \|w_m^{\sigma,\mathbf h}-w^{\sigma,\mathbf h}\|_{W_q^{2,1}(\mathcal D_T^\sigma)}
    \le C\Bigl(&
       \|f_m^{\sigma,\mathrm b}-f^{\sigma,\mathrm b}\|_{L^q}
       +
       \|\mathbf h_{\delta,m}^{\mathbf h}-\mathbf h_\delta^{\mathbf h}
       \|_{\mathcal Y_T}
    \Bigr).
\end{align*}
Using \eqref{eq:intermediate-trace-continuity},
\begin{equation}\label{eq:fixed-trace-bulk-continuity}
    w_m^{\sigma,\mathbf h}
    \longrightarrow
    w^{\sigma,\mathbf h}
    \quad\text{in }W_q^{2,1}(\mathcal D_T^\sigma).
\end{equation}
Taking the two traces at $\Sigma_{2\delta}^\sigma$ gives
\begin{equation}\label{eq:Schwarz-parameter-continuity}
    \mathcal S_T^{\widehat{\mathcal U}_m}(\mathbf h)
    \longrightarrow
    \mathcal S_T^{\widehat{\mathcal U}}(\mathbf h)
    \quad\text{in }\mathcal Y_T(\Gamma_0;\mathbb R^{2L}).
\end{equation}

Let
\[
    \mathbf h_m
    :=\mathbf h^{\widehat{\mathcal U}_m},
    \qquad
    \mathbf h
    :=\mathbf h^{\widehat{\mathcal U}}
\]
be the fixed traces.  The Schwarz maps have a common contraction constant
$\kappa(T)<1$.  Therefore
\begin{align*}
    \|\mathbf h_m-\mathbf h\|_{\mathcal Y_T}
    &\le
    \|\mathcal S_T^{\widehat{\mathcal U}_m}(\mathbf h_m)
      -\mathcal S_T^{\widehat{\mathcal U}_m}(\mathbf h)\|_{\mathcal Y_T}\\
    &\quad+
    \|\mathcal S_T^{\widehat{\mathcal U}_m}(\mathbf h)
      -\mathcal S_T^{\widehat{\mathcal U}}(\mathbf h)\|_{\mathcal Y_T}\\
    &\le
    \kappa(T)\|\mathbf h_m-\mathbf h\|_{\mathcal Y_T}
    +
    \|\mathcal S_T^{\widehat{\mathcal U}_m}(\mathbf h)
      -\mathcal S_T^{\widehat{\mathcal U}}(\mathbf h)\|_{\mathcal Y_T}.
\end{align*}
Hence
\begin{equation}\label{eq:fixed-trace-parameter-continuity}
    \|\mathbf h_m-\mathbf h\|_{\mathcal Y_T}
    \le
    \frac{1}{1-\kappa(T)}
    \|\mathcal S_T^{\widehat{\mathcal U}_m}(\mathbf h)
      -\mathcal S_T^{\widehat{\mathcal U}}(\mathbf h)\|_{\mathcal Y_T}
    \longrightarrow0.
\end{equation}

We finally compare the outputs.  Write
\[
    \mathcal F_T(\widehat{\mathcal U}_m)
    =
    (\mathbf v_m,w_m^+,w_m^-),
    \qquad
    \mathcal F_T(\widehat{\mathcal U})
    =
    (\mathbf v,w^+,w^-).
\]
The collar difference satisfies
\[
\begin{cases}
    \begin{aligned}
    \mathscr L_m(\mathbf v_m-\mathbf v)
    ={}&\bigl(\mathsf A_{\mathrm{dbl}}^{ij}(\widehat{\mathbf v}_m)
          -\mathsf A_{\mathrm{dbl}}^{ij}(\widehat{\mathbf v})\bigr)
          \partial_{ij}\mathbf v\\
       &+\bigl(\mathsf B_{\mathrm{dbl}}^k(\widehat{\mathbf v}_m)
          -\mathsf B_{\mathrm{dbl}}^k(\widehat{\mathbf v})\bigr)
          \partial_k\mathbf v
        +(\mathbf f_m-\mathbf f),
    \end{aligned}\\[0.3em]
    (\mathbf v_m-\mathbf v)|_{r=2\delta}
    =\mathbf h_m-\mathbf h,\\
    (\mathbf v_m-\mathbf v)|_{r=-2\delta}
    =\mathcal R_\Lambda(\mathbf h_m)
     -\mathcal R_\Lambda(\mathbf h),\\
    (\mathbf v_m-\mathbf v)(\cdot,0)=0.
\end{cases}
\]
The trace-space Nemytskii estimate for $\mathcal R_\Lambda$, together with
\eqref{eq:frozen-coefficient-convergence} and
\eqref{eq:fixed-trace-parameter-continuity}, gives
\[
    \mathbf v_m\longrightarrow\mathbf v
    \quad\text{in }W_q^{2,1}(\mathcal C_{\mathrm{dbl},T};\mathbb R^{2L}).
\]
Their traces at $r=\delta$ therefore converge in $\mathcal Y_T$.  The
bulk differences satisfy linear Dirichlet problems with these convergent
boundary data and with right-hand sides converging in $L^q$, so
\[
    w_m^\sigma\longrightarrow w^\sigma
    \quad\text{in }W_q^{2,1}(\mathcal D_T^\sigma;\mathbb R^L).
\]
The continuous embedding into
$C^{1+\alpha,(1+\alpha)/2}$ now gives
\[
    \mathcal F_T(\widehat{\mathcal U}_m)
    \longrightarrow
    \mathcal F_T(\widehat{\mathcal U})
    \quad\text{in }\mathcal X_T^{\mathrm{match}}.
\]
\end{proof}

Since $\mathcal A_T$ is nonempty, closed, bounded, and convex in
$\mathcal X_T^{\mathrm{match}}$, Schauder's theorem gives
\begin{equation}\label{eq:Schauder-fixed-point}
    \mathcal U^*
    =
    \mathcal F_T(\mathcal U^*)
    \quad\text{for some }\mathcal U^*\in\mathcal A_T.
\end{equation}
At this fixed state the frozen equations become
\eqref{eq:nonlinear-doubled-problem} and \eqref{eq:bulk-Dirichlet-problem}.
The positive collar and
bulk components agree on the overlap.  It remains to recover the reflection
symmetry and the target constraints.

\subsection{Short-time uniqueness for the nonlinear doubled problem}

\begin{lemma}[Short-time nonlinear doubled uniqueness]
\label{lem:nonlinear-doubled-uniqueness}
Let $q>d+2$ and fix
$0<\alpha<1-(d+2)/q$.  For every $M_0>0$ there exists
$T_{\mathrm{uniq}}>0$, depending only on the fixed geometry, $M_0$, the
initial field, and the structural coefficient bounds, with the following
property.  For $0<T\le T_{\mathrm{uniq}}$ and a fixed admissible Schwarz
trace $\mathbf h$, the nonlinear doubled problem
\eqref{eq:nonlinear-doubled-problem} has at most one solution in
\begin{equation}\label{eq:nonlinear-uniqueness-class}
\begin{aligned}
    \mathcal B_T(M_0)
    :=
    \Bigl\{
      \mathbf v\in W_q^{2,1}(\mathcal C_{\mathrm{dbl},T};\mathbb R^{2L}):\;&
      \mathbf v(\overline{\mathcal C_{\mathrm{dbl}}}\times[0,T])
      \subset \operatorname{Tub}_{\rho_\Lambda}(\Lambda),\\
      &\|\mathbf v\|_{W_q^{2,1}}\le M_0
    \Bigr\}.
\end{aligned}
\end{equation}
\end{lemma}

\begin{proof}
Let $\mathbf v^1$ and $\mathbf v^2$ be two solutions with the same initial and
boundary data, and set
\[
    \mathbf z:=\mathbf v^1-\mathbf v^2.
\]
Then $\mathbf z(\cdot,0)=0$ and
$\mathbf z|_{r=\pm2\delta}=0$.  Subtraction gives
\begin{equation}\label{eq:nonlinear-uniqueness-difference}
\begin{aligned}
    \mathscr L_{\mathbf v^1}\mathbf z
    ={}&
    \bigl[
      \mathsf A_{\mathrm{dbl}}^{ij}(\mathbf v^1)
      -\mathsf A_{\mathrm{dbl}}^{ij}(\mathbf v^2)
    \bigr]\partial_{ij}\mathbf v^2\\
    &+
    \bigl[
      \mathsf B_{\mathrm{dbl}}^k(\mathbf v^1)
      -\mathsf B_{\mathrm{dbl}}^k(\mathbf v^2)
    \bigr]\partial_k\mathbf v^2\\
    &+
    \mathbf Q_{\mathrm{dbl}}(\mathbf v^1,\nabla\mathbf v^1)
    -\mathbf Q_{\mathrm{dbl}}(\mathbf v^2,\nabla\mathbf v^2).
\end{aligned}
\end{equation}
First obtain a uniform bound for the first derivatives of the two solutions.
The initial-trace estimate
\eqref{eq:initial-trace-holder-smallness}, the embedding
$W_q^{2-2/q}\hookrightarrow C^{1+\alpha}$, and the assumed
$W_q^{2,1}$ bounds imply
\begin{equation}\label{eq:nonlinear-uniqueness-gradient-bound}
    \|\nabla\mathbf v^a\|_{L^\infty}
    \le C\left(
       M_0+
       \|\mathbf v_0^{\mathrm{dbl}}\|_{W_q^{2-2/q}}
    \right)
    =:C_{M_0},
    \qquad a=1,2,
\end{equation}
for $T\le\min\{1,T_0\}$.  This step uses only the natural initial trace norm
and does not require $\mathbf v_0^{\mathrm{dbl}}\in W_q^2$.

Since the candidate solutions take values in
$\operatorname{Tub}_{\rho_\Lambda}(\Lambda)$, the derivatives of the coefficient maps
that appear below are uniformly bounded along them by
\Cref{lem:tubular-control}.  Hence
\[
    |\mathsf A_{\mathrm{dbl}}^{ij}(\mathbf v^1)
      -\mathsf A_{\mathrm{dbl}}^{ij}(\mathbf v^2)|
    +
    |\mathsf B_{\mathrm{dbl}}^k(\mathbf v^1)
      -\mathsf B_{\mathrm{dbl}}^k(\mathbf v^2)|
    \le C|\mathbf z|,
\]
and, using \eqref{eq:nonlinear-uniqueness-gradient-bound},
\[
    |\mathbf Q_{\mathrm{dbl}}(\mathbf v^1,\nabla\mathbf v^1)
      -\mathbf Q_{\mathrm{dbl}}(\mathbf v^2,\nabla\mathbf v^2)|
    \le
    C_{M_0}
    (|\mathbf z|+|\nabla\mathbf z|).
\]
Therefore the right-hand side $\mathbf f$ of
\eqref{eq:nonlinear-uniqueness-difference} satisfies
\begin{equation}\label{eq:nonlinear-uniqueness-forcing}
    \|\mathbf f\|_{L^q}
    \le
    C\|\mathbf z\|_{L^\infty}
       \|\nabla^2\mathbf v^2\|_{L^q}
    +C_{M_0}
       \bigl(
          \|\mathbf z\|_{L^q}
          +\|\nabla\mathbf z\|_{L^q}
       \bigr).
\end{equation}
Because $\mathbf z$ has zero initial trace,
\begin{align*}
    \|\mathbf z\|_{L^q}
    &\le T\|\mathbf z\|_{W_q^{2,1}},\\
    \|\nabla\mathbf z\|_{L^q}
    &\le CT^{1/2}\|\mathbf z\|_{W_q^{2,1}},\\
    \|\mathbf z\|_{L^\infty}
    &\le CT^{(1+\alpha)/2}\|\mathbf z\|_{W_q^{2,1}}.
\end{align*}
Also
$\|\nabla^2\mathbf v^2\|_{L^q}\le M_0$, and substitution into
\eqref{eq:nonlinear-uniqueness-forcing} yields
\begin{equation}\label{eq:nonlinear-uniqueness-small-forcing}
    \|\mathbf f\|_{L^q}
    \le C_{M_0}T^{1/2}
       \|\mathbf z\|_{W_q^{2,1}}.
\end{equation}
By the parabolic embedding and the bound $\|\mathbf v^1\|_{W_q^{2,1}}\le M_0$,
the coefficients
$\mathsf A_{\mathrm{dbl}}^{ij}(\mathbf v^1)$ are uniformly Lipschitz in the
spatial variables and uniformly $1/2$-H\"older in time.  The uniform bounds
on $D\mathcal R_\Lambda$ and its inverse over
$\overline{\operatorname{Tub}_{\rho_\Lambda}(\Lambda)}$, together with the
similarity representation of the reflected principal symbol, give uniform
parameter ellipticity and the Dirichlet Lopatinskii--Shapiro condition exactly as in
\Cref{prop:frozen-coefficients}.  The first-order coefficients
$\mathsf B_{\mathrm{dbl}}^k(\mathbf v^1)$ are uniformly bounded.
The left-hand operator satisfies the
hypotheses of \Cref{lem:uniform-Wq-dirichlet}, with constants depending only
on the quantities listed in the statement.  Since the initial and boundary
values of $\mathbf z$ vanish,
\[
    \|\mathbf z\|_{W_q^{2,1}}
    \le C_{\mathrm{lin}}\|\mathbf f\|_{L^q}
    \le C_{\mathrm{lin}}C_{M_0}T^{1/2}
       \|\mathbf z\|_{W_q^{2,1}}.
\]
Choose $T_{\mathrm{uniq}}$ so that the coefficient on the right is less than
one.  Then $\mathbf z\equiv0$.
\end{proof}

\subsection{Recovery of the interface conditions}

\begin{lemma}[Reflection symmetry and interface recovery]
\label{lem:recover-interface-conditions}
Let $\mathbf v$ be the doubled collar component of the Schauder fixed point.
After decreasing $T$,
\begin{equation}\label{eq:recovered-reflection-symmetry}
    \mathbf v(p,-r,t)
    =
    \mathcal R_\Lambda(\mathbf v(p,r,t)),
    \qquad 0<r<2\delta.
\end{equation}
Consequently,
\begin{equation}\label{eq:recovered-interface-conditions}
    \mathbf v(p,0,t)\in\Lambda,
    \qquad
    \partial_r\mathbf v(p,0,t)
    \perp_{\mathbb R^{2L}}T_{\mathbf v(p,0,t)}\Lambda.
\end{equation}
\end{lemma}

\begin{proof}
Define
\[
    \mathbf v^R(p,r,t)
    :=
    \mathcal R_\Lambda(\mathbf v(p,-r,t)).
\]
Since the fixed state belongs to $\mathcal A_T$, the uniform tubular control
\eqref{eq:frozen-Lambda-tube} gives
\[
    \mathbf v(\overline{\mathcal C_{\mathrm{dbl}}}\times[0,T])
    \subset\operatorname{Tub}_{\rho_\Lambda}(\Lambda).
\]
Normal reflection preserves this tubular neighborhood by
\eqref{eq:reflection-invariant-subtube}; hence $\mathbf v^R$ is well-defined
and takes values in the same fixed tubular neighborhood.  Since
$\mathcal R_\Lambda$ is
smooth on the larger neighborhood
$\operatorname{Tub}_{2\rho_\Lambda}(\Lambda)$, the Sobolev composition
theorem gives
\[
    \mathbf v^R\in
    W_q^{2,1}(\mathcal C_{\mathrm{dbl},T};\mathbb R^{2L}),
    \qquad
    \|\mathbf v^R\|_{W_q^{2,1}}
    \le C_{\Lambda,\rho_\Lambda}
       \bigl(1+\|\mathbf v\|_{W_q^{2,1}}^2\bigr).
\]
By the reflection-covariant definitions
\eqref{eq:target-reflected-principal}--\eqref{eq:doubled-target-nonlinearity}, the
reflected initial datum, and the reflected boundary values, $\mathbf v^R$
solves the same nonlinear doubled Dirichlet problem as $\mathbf v$.  By
\eqref{eq:uniform-output-maximal-regularity}, $\|\mathbf v\|_{W_q^{2,1}}\le M_*$,
where $M_*$ is independent of $T$.  Fix
$M_1:=M_*+C_{\Lambda,\rho_\Lambda}(1+M_*^2)$, so both fields belong to
$\mathcal B_T(M_1)$.  After replacing $T$ by
$\min\{T,T_{\mathrm{uniq}}(M_1)\}$,
\Cref{lem:nonlinear-doubled-uniqueness} gives $\mathbf v^R=\mathbf v$,
proving \eqref{eq:recovered-reflection-symmetry}.  (The restriction of
$\mathbf v$ to the shorter interval still solves
\eqref{eq:nonlinear-doubled-problem} there, with the trace $\mathbf h$
restricted to it, and the bounds on $\|\mathbf v\|_{W_q^{2,1}}$ and
$\|\mathbf v^R\|_{W_q^{2,1}}$ do not increase.)

At $r=0$, the symmetry gives
$\mathbf v=\mathcal R_\Lambda(\mathbf v)$.  The fixed-point set of the normal
reflection is exactly $\Lambda$.  Differentiating the symmetry at $r=0$ gives
\[
    D\mathcal R_\Lambda(\mathbf v)\partial_r\mathbf v
    =
    -\partial_r\mathbf v.
\]
By \eqref{eq:reflection-differential}, the tangential component of
$\partial_r\mathbf v$ vanishes.  This proves
\eqref{eq:recovered-interface-conditions}.
\end{proof}

\subsection{Preservation of the target manifolds}

\begin{proposition}[Preservation of a smooth target]
\label{prop:target-preservation}
Let $N\subset\mathbb R^L$ be a compact smooth embedded submanifold, let $D$ be
a compact smooth manifold of dimension $d$ with boundary, let $q>d+2$, and let
\[
    \mathscr L:=\partial_t-a^{ij}(x,t)\partial_{ij}-b^k(x,t)\partial_k,
\]
where $a^{ij}=a^{ji}$, the coefficients $a^{ij}$ and $b^k$ are bounded, and
\[
    a^{ij}(x,t)\xi_i\xi_j\ge\lambda_0|\xi|^2
\]
for some $\lambda_0>0$. Choose $\rho_N>0$ so that the nearest-point
projection
\[
    \Pi_N:\operatorname{Tub}_{\rho_N}(N)\longrightarrow N
\]
is smooth, and define
\[
    \widetilde{\mathrm{II}}:=-D^2\Pi_N
\]
on this tubular neighborhood. Suppose that
\[
    U\in W_q^{2,1}(D_T;\mathbb R^L)
\]
solves
\begin{equation}\label{eq:ambient-target-equation}
    \mathscr L U
    =a^{ij}\widetilde{\mathrm{II}}(U)(\partial_iU,\partial_jU),
\end{equation}
with
\[
    U(\overline D\times[0,T])
    \Subset\operatorname{Tub}_{\rho_N}(N),
\]
and that its initial and lateral values lie in $N$:
\[
    U(\cdot,0)\in N,
    \qquad
    U|_{\partial D\times(0,T)}\in N.
\]
Then
\[
    U(x,t)\in N
    \qquad\text{throughout }D\times[0,T].
\]
\end{proposition}

\begin{proof}
Write $\Pi:=\Pi_N$ and set
\[
    W:=U-\Pi(U),
    \qquad
    \rho:=\frac12|W|^2.
\]
Since $q>d+2$, the parabolic Sobolev embedding gives
$\nabla U\in L^\infty$.  The compact range assumption and the Sobolev chain
and product rules imply
\[
    \Pi(U),\ W,\ \rho
    \in W_q^{2,1}(D_T)\cap C(\overline D\times[0,T]),
\]
and all identities below hold almost everywhere.  The scalar-principal chain
rule gives
\begin{equation}\label{eq:projection-chain-rule}
    \mathscr L\Pi(U)
    =D\Pi(U)\mathscr L U
     -a^{ij}D^2\Pi(U)(\partial_iU,\partial_jU).
\end{equation}
Using \eqref{eq:ambient-target-equation} and
$\widetilde{\mathrm{II}}=-D^2\Pi$ therefore yields
\begin{equation}\label{eq:normal-displacement-preliminary}
    \mathscr L W
    =a^{ij}D\Pi(U)D^2\Pi(U)
      (\partial_iU,\partial_jU).
\end{equation}
For every point in the tubular neighborhood, $W=U-\Pi(U)$ is normal to $N$ at
$\Pi(U)$, whereas the range of $D\Pi(U)$ is contained in
$T_{\Pi(U)}N$.  Consequently,
\[
    W\cdot\mathscr L W=0.
\]
A second application of the scalar chain rule gives, almost everywhere,
\begin{equation}\label{eq:normal-distance-subsolution}
\begin{aligned}
    \mathscr L\rho
    &=W\cdot\mathscr L W-a^{ij}\partial_iW\cdot\partial_jW\\
    &=-a^{ij}\partial_iW\cdot\partial_jW\\
    &\le-\lambda_0|\nabla W|^2
    \le0.
\end{aligned}
\end{equation}
The initial and lateral boundary assumptions imply that $\rho=0$ on the
parabolic boundary.  Since $q>d+2$,
$\rho\in W_q^{2,1}(D_T)\subset
W_{d+1,\mathrm{loc}}^{2,1}(D_T)$ and
$\rho\in C(\overline D\times[0,T])$.  Together with
$\mathscr L\rho\le0$ almost everywhere, the parabolic
Aleksandrov maximum principle for strong Sobolev subsolutions
\cite[Chapter~VII, Theorem~7.1]{Lieberman1996} gives
\[
    \sup_{\overline D\times[0,T]}\rho
    \le\sup_{\partial_pD_T}\rho=0.
\]
Since $\rho\ge0$, one has $\rho\equiv0$, hence $U=\Pi(U)\in N$.
\end{proof}

For each $\sigma\in\{+,-\}$, gluing the collar and bulk fields across their
overlap defines a single fixed-domain field $U^\sigma$ on $\Omega_0^\sigma$.
In the collar it is the corresponding component of $\mathbf v$, and away
from the interface it is $w^\sigma$. By construction, it is the pullback of
the physical field by the ambient isotopy $\Psi_t$ from
\Cref{lem:fermi-compatible-parametrization}:
\[
    U^\sigma(y,t)=u^\sigma(\Psi_t(y),t).
\]
It satisfies on $\Omega_0^\sigma$ an equation of the form
\begin{equation}\label{eq:global-fixed-target-equation}
    \partial_tU^\sigma
    -a_b^{\sigma,ij}\partial_{ij}U^\sigma
    -b_b^{\sigma,k}\partial_kU^\sigma
    =a_b^{\sigma,ij}\widetilde{\mathrm{II}}^\sigma(U^\sigma)
      (\partial_iU^\sigma,\partial_jU^\sigma),
\end{equation}
where $(a_b^{\sigma,ij})$ is the symmetric inverse pullback metric and is
uniformly elliptic.  On the collar, $\Psi_t\circ\Phi_0^\sigma=\Phi_t^\sigma$;
the resulting equation agrees with the one-sided collar equation on the whole
collar $0<r<2\delta$ by \Cref{lem:coefficient-agreement-collar}.  By
\Cref{lem:recover-interface-conditions}, the boundary trace
on $\Gamma_0$ lies in $N^\sigma$, and the initial value lies in $N^\sigma$.
Since the fixed state belongs to $\mathcal A_T$,
\eqref{eq:frozen-target-tubes} shows that the whole range of $U^\sigma$ is
contained in $\operatorname{Tub}_{\rho_N/2}(N^\sigma)$, on which
\eqref{eq:target-nonlinearity-extension} agrees with
$-D^2\Pi_{N^\sigma}$.  Applying
\Cref{prop:target-preservation} to \eqref{eq:global-fixed-target-equation}
gives
\begin{equation}\label{eq:target-preserved}
    u^\sigma(x,t)\in N^\sigma,
    \qquad
    x\in\Omega_t^\sigma,
    \quad 0\le t\le T.
\end{equation}

\subsection{Local regularity on the physical side}

\begin{lemma}[Adapted boundary normal form]
\label{lem:adapted-boundary-normal-form}
Let
\[
    m:=\dim\mathbf N,
    \qquad
    k:=\dim\Lambda.
\]
Near each value attained by the interface trace, choose a relative Fermi chart
\[
    \Theta:V\subset\mathbb R^k\times\mathbb R^{m-k}
    \longrightarrow\mathbf N,
    \qquad
    y=(\vartheta,\zeta),
\]
for which $\Lambda=\{\zeta=0\}$.  Wherever
$\mathbf v=\Theta(y)$, the physical collar equation has the coordinate form
\begin{equation}\label{eq:adapted-target-coordinate-system}
    \partial_ty
    -\mathcal A^{ij}(p,r,t,y)\partial_{ij}y
    -\mathcal C^k(p,r,t,y)\partial_ky
    =\mathcal B(p,r,t,y,\nabla y),
\end{equation}
where $\mathcal A^{ij}$, $\mathcal C^k$, and $\mathcal B$ are smooth in their
displayed arguments, $\mathcal B$ is quadratic in $\nabla y$, and the principal
symbol is uniformly parameter elliptic on compact
subsets of $V$.  The Fermi structure gives
\begin{equation}\label{eq:adapted-normal-principal-matrices}
    \mathcal A^{ar}=0,
    \qquad
    \mathcal A^{rr}=I_m
\end{equation}
throughout the charted collar.  The interface conditions become
\begin{equation}\label{eq:adapted-mixed-boundary-condition}
    \zeta|_{r=0}=0,
    \qquad
    \partial_r\vartheta|_{r=0}=0.
\end{equation}
At $r=0$, the remaining principal matrices satisfy
\begin{equation}\label{eq:adapted-boundary-principal-matrices}
    \mathcal A^{ab}=a_{c,0}^{ab}(p,t)I_m,
\end{equation}
where $(a_{c,0}^{ab})$ is the inverse metric induced on $\Gamma_t$ and pulled back
to $\Gamma_0$.  Consequently, the mixed boundary operator in
\eqref{eq:adapted-mixed-boundary-condition} satisfies the uniform
Lopatinskii--Shapiro condition.
\end{lemma}

\begin{proof}
Set $J(y):=D\Theta(y)$, viewed as an isomorphism from $\mathbb R^m$ onto
$T_{\Theta(y)}\mathbf N$.  The pair coefficient $\mathsf A_c^{ij}$ preserves
$T\mathbf N$, because it acts by the scalar $a_c^{+,ij}$ on the first phase
block and by the scalar $a_c^{-,ij}$ on the second.  Substituting
$\mathbf v=\Theta(y)$ into \eqref{eq:physical-collar-problem} and applying the
inverse of $J(y)$ on its tangent image gives
\[
\begin{aligned}
    \mathcal A^{ij}(p,r,t,y)
    &:=J(y)^{-1}\circ
      \mathsf A_c^{ij}(p,r,t)|_{T_{\Theta(y)}\mathbf N}
      \circ J(y),\\
    \mathcal C^k(p,r,t,y)
    &:=J(y)^{-1}\circ
      \mathsf B_c^k(p,r,t)|_{T_{\Theta(y)}\mathbf N}
      \circ J(y).
\end{aligned}
\]
Here $\mathsf B_c^k$ also preserves $T\mathbf N$, by its phasewise scalar form.
The terms involving $D^2\Theta$ and $\mathbf Q_c$ form the smooth quadratic
function $\mathcal B$.  For a covector $\xi$, the matrix
$\mathcal A^{ij}\xi_i\xi_j$ is similar to the restriction to $T\mathbf N$ of
\[
    \operatorname{diag}
    \bigl(a_c^{+,ij}\xi_i\xi_j I_L,
          a_c^{-,ij}\xi_i\xi_j I_L\bigr).
\]
Uniform ellipticity of the two Fermi metrics and bounded condition numbers for
$J(y)$ on compact coordinate subsets give the parameter-elliptic
resolvent estimate for these matrices.  Notice that these matrices need not be scalar when
$r>0$.

Since $a_c^{\sigma,ar}=0$ and $a_c^{\sigma,rr}=1$ throughout the spatial Fermi
collar, similarity by $J(y)$ gives
\eqref{eq:adapted-normal-principal-matrices}.  At $r=0$, the two one-sided
Fermi metrics have the same tangential trace, so
$\mathsf A_c^{ab}=a_{c,0}^{ab}I_{2L}$; this proves
\eqref{eq:adapted-boundary-principal-matrices}.  The relative Fermi chart also
makes $D_\vartheta\Theta(\vartheta,0)$ orthogonal to
$D_\zeta\Theta(\vartheta,0)$, so the trace condition and
$\partial_r\mathbf v\in T^{\perp,\mathbf N}\Lambda$ are exactly
\eqref{eq:adapted-mixed-boundary-condition}.

Finally freeze the principal coefficients at a boundary point and take the
Laplace--Fourier transform in $(t,p)$.  For
$\lambda\in\overline{\mathfrak S}_{\pi-\theta}$ and
$(\lambda,\xi)\ne(0,0)$, every target component has a unique stable normal
mode $c\exp(-\mu r)$, where
\[
    \mu^2=\lambda+a_{c,0}^{ab}\xi_a\xi_b,
    \qquad
    \operatorname{Re}\mu>0.
\]
The Dirichlet condition annihilates the amplitudes in the $\zeta$ components,
whereas the Neumann condition gives $-\mu c=0$ in the $\vartheta$ components.
The only stable solution with homogeneous boundary data is therefore zero, and
uniformity follows from compactness of the boundary coefficient set and the
uniform ellipticity of $(a_{c,0}^{ab})$.
\end{proof}

\begin{lemma}[Compatibility in adapted target coordinates]
\label{lem:adapted-compatibility}
In the relative Fermi coordinates of
\Cref{lem:adapted-boundary-normal-form}, let
$y=(\vartheta,\zeta)$ and generate the formal time jets
$y_j=\partial_t^jy|_{t=0}$ recursively from
\eqref{eq:adapted-target-coordinate-system}.  Then the geometric
compatibility conditions in \Cref{subsec:compatibility-conditions} are equivalent
to the standard mixed-boundary compatibility conditions
\begin{equation}\label{eq:adapted-compatibility-hierarchy}
    (y_j)_\zeta|_{r=0}=0,
    \qquad
    \partial_r(y_j)_\vartheta|_{r=0}=0,
    \qquad j=0,1,2,\ldots,
\end{equation}
in every adapted chart.  The all-order compatibility assumption is thus
exactly the corner compatibility required for repeated mixed
Dirichlet--Neumann parabolic regularity.
\end{lemma}

\begin{proof}
At $r=0$, the relative Fermi chart identifies the trace constraint with
$\zeta=0$.  More invariantly, the trace defect
$\mathbf v-\Pi_\Lambda(\mathbf v)$ is, in these coordinates, a smooth function
of $(\vartheta,\zeta)$ whose derivative in the $\zeta$ variables at
$\zeta=0$ is an isomorphism onto the relative normal space
$T^{\perp,\mathbf N}\Lambda$.
By induction on the order of the jet, vanishing of all time jets of the
geometric trace defect is equivalent to
$\partial_t^j\zeta|_{r=0,t=0}=0$.

At $\zeta=0$, the two coordinate subspaces are orthogonal and
\[
    \partial_r\mathbf v
    =D_\vartheta\Theta(\vartheta,0)\partial_r\vartheta
     +D_\zeta\Theta(\vartheta,0)\partial_r\zeta.
\]
The tangential projection of the conormal vector is therefore an invertible
smooth matrix, depending on $\vartheta$, times $\partial_r\vartheta$.
Vanishing of all time jets of the geometric conormal defect is thus equivalent,
again inductively, to
$\partial_r\partial_t^j\vartheta|_{r=0,t=0}=0$.
Finally, substituting the equation recursively for the time derivatives shows
that these actual time jets are precisely the formal jets $y_j$.  This proves
\eqref{eq:adapted-compatibility-hierarchy}.
\end{proof}

\begin{proposition}[Instantaneous smoothing on the physical phases]
\label{prop:physical-smoothness}
Under the hypotheses of \Cref{thm:main-short-time-solution}, for every
$0<\tau<T'<T$ the glued physical phase
fields are smooth up to the physical interface on the closed time strip
$\tau\le t\le T'$.
\end{proposition}

\begin{proof}
Only the physical half-collar $r\ge0$ is used.  By
\Cref{lem:recover-interface-conditions,prop:target-preservation},
\[
    \mathbf v(p,0,t)\in\Lambda,
    \qquad
    \mathbf v(p,r,t)\in\mathbf N
    \quad(r\ge0),
\]
and
\[
    \partial_r\mathbf v(p,0,t)
    \in T_{\mathbf v(p,0,t)}^{\perp,\mathbf N}\Lambda.
\]
On the compact strip $\tau/2\le t\le T'$, the boundary trace lies in finitely
many target charts from
\Cref{lem:adapted-boundary-normal-form}.  After localization in $(p,t)$, the
coordinate field $y=(\vartheta,\zeta)$ satisfies
\begin{equation}\label{eq:adapted-system-expanded}
    \partial_ty
    -\mathcal A^{ab}(p,r,t,y)\partial_{ab}y
    -\partial_{rr}y
    -\mathcal C^k(p,r,t,y)\partial_ky
    =\mathcal B(p,r,t,y,\nabla y),
\end{equation}
with
\[
    \zeta|_{r=0}=0,
    \qquad
    \partial_r\vartheta|_{r=0}=0.
\]
For the corresponding reduction on a fixed domain, see Chen--Lin
\cite[Remark~2.11]{ChenLin1998}.
Here we used
$\mathcal A^{ar}=0$ and $\mathcal A^{rr}=I_m$ from
\eqref{eq:adapted-normal-principal-matrices}.  The principal symbol is normally
elliptic, and the mixed boundary operator satisfies the uniform
Lopatinskii--Shapiro condition.

The artificial boundary data at $r=2\delta$ have only the trace regularity
supplied by the fixed-point construction. We therefore localize the initial
Schauder bootstrap to a half-collar bounded away from $r=2\delta$.
Choose a spatial cutoff supported in $0\le r<7\delta/4$ and equal to one on
$0\le r\le3\delta/2$, and a time cutoff that vanishes for
$t\le\tau/2$ and equals one on $[\tau,T']$.  The established regularity
\[
    y\in W_q^{2,1}
      \cap C^{1+\alpha,(1+\alpha)/2}
\]
implies that the localized coefficients and right-hand side in
\eqref{eq:adapted-system-expanded} belong to
$C^{\alpha,\alpha/2}$. The local mixed-boundary Schauder estimate for normally
parabolic systems therefore upgrades the regularity to
\[
    y\in C^{2+\alpha,1+\alpha/2}
\]
on a smaller half-collar and time strip; see Solonnikov
\cite{Solonnikov1965}.  The time cutoff removes the initial corner, so this
step requires no higher compatibility condition.

We now iterate.  Suppose that the localized field has been controlled through
some parabolic H\"older order.  Differentiate
\eqref{eq:adapted-system-expanded} in time and in the tangential variables.
Every differentiated field satisfies a linear normally parabolic system with
the same principal part and the same homogeneous block boundary operators:
Dirichlet on the $\zeta$ components and Neumann on the $\vartheta$ components.
All commutator terms involve only already controlled derivatives.  The same
interior-in-time mixed-boundary Schauder estimate therefore raises the
tangential and time regularity by one parabolic order.

To recover higher normal derivatives, solve
\eqref{eq:adapted-system-expanded} for $\partial_{rr}y$:
\begin{equation}\label{eq:recover-second-normal-derivative}
    \partial_{rr}y
    =\partial_ty
     -\mathcal A^{ab}(p,r,t,y)\partial_{ab}y
     -\mathcal C^k(p,r,t,y)\partial_ky
     -\mathcal B(p,r,t,y,\nabla y).
\end{equation}
Differentiating this identity in $t$, in the tangential variables, and $k$ times
in $r$ expresses $\partial_r^{k+2}y$ through derivatives of $y$ with at most
$k+1$ normal derivatives.  The preceding iteration controls all tangential and
time derivatives of $y$ together with up to two normal derivatives, so
induction on $k$ shows that all derivatives of $y$, and hence of $\mathbf v$,
are continuous up to $r=0$; that is,
\[
    \mathbf v\in
    C^\infty\bigl(\Gamma_0\times[0,3\delta/2]\times[\tau,T']\bigr).
\]
Since the cutoff vanishes before $r=2\delta$, this argument never uses more
regularity of the artificial Dirichlet datum than is supplied by the
fixed-point construction.

The trace at $r=\delta$ is therefore smooth, and it is the boundary datum for
each bulk problem.  The bulk equations have smooth coefficients and scalar
principal part.  Standard Dirichlet parabolic bootstrapping away from the
initial time gives
\[
    w^\sigma\in C^\infty(\overline{\mathcal D^\sigma}\times[\tau,T']).
\]
On $\delta<r<2\delta$, the bulk and collar fields coincide, so the fixed
phase field may be defined using the smooth collar for $r<3\delta/2$ and the
smooth bulk field for $r>\delta$; the definitions agree on their open overlap.
This produces a single smooth map on
$\overline{\Omega_0^\sigma}\times[\tau,T']$.  Composing with
$\Psi_t^{-1}$ gives smoothness on the corresponding moving time strip.
\end{proof}

\begin{proposition}[Regularity at the initial corner]
\label{prop:initial-corner-smoothness}
If, in addition, $u_0^\pm$ and the prescribed geometry are smooth and the
compatibility conditions of every order from
\Cref{subsec:compatibility-conditions} hold, then the strong solution is
smooth on the closed moving phase cylinders, including at $t=0$.
\end{proposition}

\begin{proof}
Repeat the preceding localized mixed-boundary bootstrap without the time
cutoff.  In adapted target coordinates, the corner conditions required after
successive differentiations are exactly
\eqref{eq:adapted-compatibility-hierarchy} by
\Cref{lem:adapted-compatibility}.  Formula
\eqref{eq:recover-second-normal-derivative} again recovers the higher normal
derivatives.  This gives smoothness of the physical collar down to $t=0$.
Its smooth trace at $r=\delta$ gives boundary data for the bulk Dirichlet
problems. Corner compatibility holds because the collar and bulk
equations agree on the overlap. Standard Dirichlet corner regularity then
yields smoothness up to $t=0$, and gluing completes the proof.
\end{proof}

\subsection{Uniqueness and completion of the proof}
\label{subsec:physical-uniqueness}

\begin{proposition}[Uniqueness of strong prescribed-interface solutions]
\label{prop:physical-solution-uniqueness}
Let $q>d+2$, and let $(u_1^+,u_1^-)$ and $(u_2^+,u_2^-)$ be two strong
solutions of \eqref{eq:original-moving-interface-problem} on a common interval
$[0,T]$, in the sense of \Cref{def:strong-solution}, for the same prescribed
interface family and the same initial data.  Then
\[
    u_1^\sigma=u_2^\sigma
    \qquad\text{on }\overline{Q_T^\sigma},
    \quad \sigma=\pm.
\]
No smallness condition on $T$ is required.
\end{proposition}

\begin{proof}
For harmonic map heat flow with a target free-boundary condition on a fixed
spatial domain, Chen and Lin \cite[\S2.2]{ChenLin1998} proved uniqueness by
applying the scalar parabolic maximum principle to the squared difference of
two solutions.  In the present moving two-phase problem, the interface
condition controls only the sum of the two boundary fluxes, so a phasewise
maximum-principle argument is not directly available.  We instead use a
coupled moving-domain energy estimate.

Set
\[
    z^\sigma:=u_1^\sigma-u_2^\sigma,
    \qquad
    \mathbf z:=(z^+,z^-).
\]
Let
\[
    \mathbf p_a:=(u_a^+,u_a^-)|_{\Gamma_t}\in\Lambda,
    \qquad
    \mathbf c_a
    :=\bigl(-\partial_{\nu_t}u_a^+,
             \partial_{\nu_t}u_a^-\bigr)
    \in T_{\mathbf p_a}^{\perp,\mathrm{amb}}\Lambda,
    \quad a=1,2.
\]
Compactness and a local graph representation of $\Lambda$ give the uniform
chord--normal estimate
\begin{equation}\label{eq:physical-uniqueness-chord-normal}
    |(\mathbf p-\mathbf q)\cdot\boldsymbol\eta|
    \le
    C_\Lambda|\mathbf p-\mathbf q|^2|\boldsymbol\eta|,
    \qquad
    \mathbf p,\mathbf q\in\Lambda,
    \quad
    \boldsymbol\eta\in T_{\mathbf p}^{\perp,\mathrm{amb}}\Lambda.
\end{equation}
Indeed, near the diagonal this follows from the quadratic normal deviation of
a smooth submanifold from its tangent space; away from the diagonal it follows
after increasing $C_\Lambda$.

By the parabolic Sobolev embedding on the smooth moving cylinders,
$u_a^\sigma\in C^{1+\alpha,(1+\alpha)/2}(\overline{Q_T^\sigma})$ for some
$0<\alpha<1-(d+2)/q$.  All first spatial derivatives of the two
solutions and both conormal fields are therefore bounded on the compact moving
cylinders.  Let $M$ denote a common bound.
Using the fixed smooth ambient extensions of the second fundamental forms,
and the compactness of the targets, the difference of the two interior
nonlinearities satisfies
\begin{equation}\label{eq:physical-uniqueness-nonlinearity}
\begin{aligned}
    &\bigl|
      \widetilde{\mathrm{II}}^\sigma(u_1^\sigma)
       (\nabla u_1^\sigma,\nabla u_1^\sigma)
      -\widetilde{\mathrm{II}}^\sigma(u_2^\sigma)
       (\nabla u_2^\sigma,\nabla u_2^\sigma)
     \bigr|\\
    &\qquad\le C_M\bigl(|z^\sigma|+|\nabla z^\sigma|\bigr).
\end{aligned}
\end{equation}

Since $q>d+2>2$, the following energy identities are justified by pulling back
to the fixed phase domains and using the standard density argument.  Let
$\mathcal V_t$ be the normal velocity of $\Gamma_t$ in the direction
$\nu_t$, and define
\[
    E(t):=\frac12\sum_{\sigma=\pm}
      \int_{\Omega_t^\sigma}|z^\sigma|^2\,dx,
    \qquad
    D(t):=\sum_{\sigma=\pm}
      \int_{\Omega_t^\sigma}|\nabla z^\sigma|^2\,dx.
\]
The Reynolds transport formula, with outward normal velocities
$-\mathcal V_t$ on $\Omega_t^+$ and $\mathcal V_t$ on $\Omega_t^-$, gives
\begin{equation}\label{eq:physical-uniqueness-Reynolds}
\begin{aligned}
    E'(t)
    ={}&\sum_{\sigma=\pm}
       \int_{\Omega_t^\sigma}
          z^\sigma\cdot\partial_tz^\sigma\,dx\\
       &+\frac12\int_{\Gamma_t}\mathcal V_t
          \bigl(|z^-|^2-|z^+|^2\bigr)\,dS.
\end{aligned}
\end{equation}
Subtracting the two phase equations and integrating the Laplacian by parts
on both moving phases yields
\begin{equation}\label{eq:physical-uniqueness-energy-identity}
\begin{aligned}
    E'(t)+D(t)
    ={}&\int_{\Gamma_t}
       \mathbf z\cdot(\mathbf c_1-\mathbf c_2)\,dS\\
       &+\sum_{\sigma=\pm}
       \int_{\Omega_t^\sigma}z^\sigma\cdot\mathcal N^\sigma\,dx\\
       &+\frac12\int_{\Gamma_t}\mathcal V_t
          \bigl(|z^-|^2-|z^+|^2\bigr)\,dS,
\end{aligned}
\end{equation}
where $\mathcal N^\sigma$ denotes the nonlinear difference in
\eqref{eq:physical-uniqueness-nonlinearity}.

Applying \eqref{eq:physical-uniqueness-chord-normal} at $\mathbf p_1$ and
$\mathbf p_2$ gives
\[
    |\mathbf z\cdot(\mathbf c_1-\mathbf c_2)|
    \le
    C_\Lambda|\mathbf z|^2
       (|\mathbf c_1|+|\mathbf c_2|)
    \le C_M|\mathbf z|^2
    \qquad\text{on }\Gamma_t.
\]
The smooth family of phase domains has a uniform trace inequality: for every
$\varepsilon>0$,
\begin{equation}\label{eq:uniform-moving-trace-inequality}
    \sum_{\sigma=\pm}
    \|z^\sigma\|_{L^2(\Gamma_t)}^2
    \le
    \varepsilon D(t)+C_\varepsilon E(t),
\end{equation}
with constants independent of $t\in[0,T]$.  This follows, for example, by
pulling the standard trace inequality on $\Omega_0^\sigma$ forward through
the common smooth isotopy $\Psi_t$.

Using \eqref{eq:physical-uniqueness-nonlinearity}, Young's inequality, the
bound on $\mathcal V_t$, and the trace estimate
\eqref{eq:uniform-moving-trace-inequality} in
\eqref{eq:physical-uniqueness-energy-identity}, and choosing
$\varepsilon$ sufficiently small, we obtain
\[
    E'(t)+\frac12D(t)\le C_{M,\Gamma}E(t)
\]
for almost every $t\in(0,T)$.  Since the two solutions have the same initial data, $E(0)=0$, and Gronwall's
inequality gives $E(t)=0$ for all $t\in[0,T]$.  Hence $z^+=z^-=0$.
\end{proof}

\begin{proof}[Proof of \Cref{thm:main-short-time-solution}]
For each frozen state the Schwarz trace map is a contraction, so it has a
unique fixed trace, and the associated doubled collar and bulk solutions agree
on the whole overlap.  The state map is a compact continuous self-map of the
closed convex set $\mathcal A_T$, so Schauder's theorem gives a fixed state,
whose positive collar and bulk components solve the nonlinear ambient
equations.  By short-time uniqueness for the nonlinear doubled problem, this
solution is reflection-symmetric; evaluating the symmetry and its derivative
at $r=0$ gives the trace constraint and the conormal condition.  The
target-preservation proposition shows that the fixed point takes values in
$N^\pm$, so the ambient equations reduce to harmonic map heat flow.  Undoing
the Fermi and global coordinate changes gives a strong $W_q^{2,1}$ solution of
\eqref{eq:original-moving-interface-problem} on $[0,T]$ with the initial and
interface conditions in the sense of traces of \Cref{def:strong-solution}.  Continuity into
$W_q^{2-2/q}$ follows from the initial-trace theorem, smoothing for $t>0$ from
\Cref{prop:physical-smoothness}, and uniqueness in the $W_q^{2,1}$
strong-solution class from \Cref{prop:physical-solution-uniqueness}.
\end{proof}

\begin{proof}[Proof of \Cref{cor:smooth-short-time-solution}]
The smooth data satisfy the hypotheses of
\Cref{thm:main-short-time-solution}, which gives the unique strong solution.
Under the all-order compatibility conditions,
\Cref{prop:initial-corner-smoothness} carries the bootstrap through the initial
corner, so this solution is smooth on the closed moving phase cylinders.
\end{proof}

\section{Continuation at a finite terminal time}
\label{sec:continuation}

Continuation concerns the bulk problem only; it does not depend on how the
interface is generated.

\begin{proposition}[Continuation under a smoothly extendible interface]
\label{prop:prescribed-interface-continuation}
Fix $q>d+2$ and $T_*<\infty$.  Suppose that the prescribed smooth separating
interface family is defined on $[0,T_*+\varepsilon_0]$ for some
$\varepsilon_0>0$.  Let $(u^+,u^-)$ be a strong solution of
\eqref{eq:original-moving-interface-problem} on $[0,T_*)$ such that
\[
    u^\sigma\in W_q^{2,1}(Q_S^\sigma;\mathbb R^L)
    \qquad
    \text{for every }S<T_*,\quad \sigma=\pm.
\]
If
\begin{equation}\label{eq:prescribed-continuation-bound}
    \sup_{0\le t<T_*}
    \left(
       \|\nabla u^+(t)\|_{L^\infty(\Omega_t^+)}
       +\|\nabla u^-(t)\|_{L^\infty(\Omega_t^-)}
    \right)<\infty,
\end{equation}
then $(u^+,u^-)$ extends uniquely as a strong solution to
$[0,T_*+\varepsilon)$ for some $\varepsilon>0$.
\end{proposition}

To prove \Cref{prop:prescribed-interface-continuation}, fix the smooth
extension of the interface given in its statement and set
\[
    \Gamma_*:=\Gamma_{T_*},
    \qquad
    \Omega_*^\sigma:=\Omega_{T_*}^\sigma.
\]
Applying \Cref{lem:fermi-compatible-parametrization} after translating the
time variable gives, on a neighborhood of $T_*$, smooth one-sided Fermi maps
\[
    \Phi_{*,t}^\sigma:
    \Gamma_*\times[0,r_*)\longrightarrow\overline{\Omega_t^\sigma}
\]
and diffeomorphisms
\[
    \Psi_{*,t}:\mathbb T^d\longrightarrow\mathbb T^d,
    \qquad
    \Psi_{*,t}(\Omega_*^\sigma)=\Omega_t^\sigma,
    \qquad
    \Psi_{*,T_*}=\operatorname{Id}.
\]
All their derivatives and inverse derivatives are uniformly bounded on a
fixed terminal interval.  Write
\[
    U^\sigma(y,t):=u^\sigma(\Psi_{*,t}(y),t),
    \qquad y\in\Omega_*^\sigma,
\]
and denote the bound in \eqref{eq:prescribed-continuation-bound} by $M$.

\subsection{The terminal-strip estimate}

On the one-sided collar define
\[
    \mathbf v(p,r,t)
    :=
    \bigl(
       u^+(\Phi_{*,t}^+(p,r),t),
       u^-(\Phi_{*,t}^-(p,r),t)
    \bigr).
\]
Uniform control of the Fermi maps gives
\begin{equation}\label{eq:terminal-collar-gradient-bound}
    |\nabla_{p,r}\mathbf v(p,r,t)|\le C M.
\end{equation}
Since $\mathbf v(p,0,t)\in\Lambda$, after decreasing
$\delta_0\in(0,r_*/4)$ we have
\begin{equation}\label{eq:terminal-collar-tube}
    \mathbf v(p,r,t)
    \in\operatorname{Tub}_{\rho_\Lambda/2}(\Lambda),
    \qquad
    0\le r\le2\delta_0,
\end{equation}
uniformly for $t$ in the terminal interval.  Define the reflected field
\begin{equation}\label{eq:terminal-reflected-field}
    \mathbf V(p,r,t)
    :=
    \begin{cases}
       \mathbf v(p,r,t),&r\ge0,\\
       \mathcal R_\Lambda(\mathbf v(p,-r,t)),&r<0.
    \end{cases}
\end{equation}
The two conditions in \eqref{eq:collar-interface-conditions} and
\eqref{eq:reflection-differential} show that $\mathbf V$ and all its first
spatial derivatives have matching one-sided traces at $r=0$.  Since the two
pieces are $W_q^{2,1}$ up to $r=0$ on every interval ending before $T_*$, the
Sobolev gluing theorem gives
\[
    \mathbf V\in W_{q,\mathrm{loc}}^{2,1}
    \bigl(\Gamma_*\times(-2\delta_0,2\delta_0)\times(0,T_*)\bigr).
\]
No distribution supported on $r=0$ therefore occurs in its second spatial
derivatives, and $\mathbf V$ solves the doubled equation determined by
\eqref{eq:doubled-principal-coefficient},
\eqref{eq:doubled-first-order-coefficient}, and
\eqref{eq:doubled-target-nonlinearity} weakly across $r=0$.

\begin{proposition}[Uniform terminal-strip estimate]
\label{prop:terminal-strip-estimate}
There are $\tau>0$ and $C<\infty$ such that
\begin{equation}\label{eq:global-terminal-Wq-estimate}
    \sum_{\sigma=\pm}
    \|U^\sigma\|_{W_q^{2,1}
       (\Omega_*^\sigma\times(T_*-\tau,T_*);\mathbb R^L)}
    \le C.
\end{equation}
\end{proposition}

\begin{proof}
Set
\[
    D_0:=\Gamma_*\times(-2\delta_0,2\delta_0).
\]
This domain remains fixed throughout the proof.  Let $a_{c,*}^{ab}(p,t)$ be the
inverse metric induced on $\Gamma_t$ and pulled
back to $\Gamma_*$.  At $r=0$, the two one-sided tangential Fermi metrics both
equal $a_{c,*}^{ab}$, while $a_c^{\sigma,rr}=1$ and
$a_c^{\sigma,ar}=0$.

We first record the coefficient estimate used below.  On the physical
half-collar, smoothness of the Fermi metrics gives
\[
    |a_c^{\sigma,ab}(p,r,t)-a_{c,*}^{ab}(p,t)|\le C|r|.
\]
Moreover, \eqref{eq:terminal-collar-gradient-bound} and the fact that
$\mathcal R_\Lambda$ fixes $\Lambda$ imply
\[
    |\mathbf V(p,r,t)-\mathbf V(p,0,t)|\le C_M|r|.
\]
On the reflected half-collar the principal coefficient is obtained by the
similarity transformation in \eqref{eq:target-reflected-principal}.  At
$r=0$ its tangential part is the scalar matrix
$a_{c,*}^{ab}(p,t)I_{\mathbb R^{2L}}$, which commutes with
$D\mathcal R_\Lambda$.  The preceding two estimates and the bounded
derivatives of $\mathcal R_\Lambda$ on the fixed tube therefore give
\begin{equation}\label{eq:terminal-principal-smallness}
\begin{aligned}
    &\mathsf A_{\mathrm{dbl}}^{rr}=I_{\mathbb R^{2L}},
    \qquad
    \mathsf A_{\mathrm{dbl}}^{ar}=0,\\
    &\left\|
       \mathsf A_{\mathrm{dbl}}^{ab}(p,r,t,\mathbf V)
       -a_{c,*}^{ab}(p,t)I_{\mathbb R^{2L}}
     \right\|
     \le C_M|r|.
\end{aligned}
\end{equation}
The compactness of the targets,
\eqref{eq:terminal-collar-gradient-bound}, and the same bounds for the
reflection also give
\begin{equation}\label{eq:terminal-lower-order-bound}
\begin{aligned}
    \bigl\|\mathsf B_{\mathrm{dbl}}^k(p,r,t,\mathbf V)\bigr\|
    &\le C_M,\\
    \left|
       \mathbf Q_{\mathrm{dbl}}
       (p,r,t,\mathbf V,\nabla\mathbf V)
    \right|
    &\le C_M|\nabla\mathbf V|^2.
\end{aligned}
\end{equation}

Choose $\tau>0$ so that the preceding coordinates are defined on
$[T_*-2\tau,T_*]$, and choose a time cutoff $\eta$ which vanishes near
$T_*-2\tau$ and equals one on $[T_*-\tau,T_*]$.  On the fixed cylinder
$D_0\times(T_*-2\tau,s)$ consider the scalar principal operator
\[
    \mathscr L_*
    :=\partial_t-a_{c,*}^{ab}(p,t)\partial_{ab}-\partial_{rr}.
\]
By \Cref{lem:uniform-Wq-dirichlet}, after translating the initial time, its
homogeneous Dirichlet estimate has a constant $C_0$ independent of
$s\in(T_*-\tau/2,T_*)$.

Now choose $0<\delta<\delta_0/4$ so small that
\begin{equation}\label{eq:terminal-absorption-choice}
    C_0C_M\delta\le\frac12,
\end{equation}
and choose $\chi\in C_c^\infty((-2\delta,2\delta))$ with $\chi=1$ on
$[-\delta,\delta]$.  Define $Z:=\eta\chi\mathbf V$ on
$|r|<2\delta$ and extend it by zero to $D_0$.  Because $\chi$ is compactly
supported, this zero extension belongs to $W_q^{2,1}$ and has homogeneous
initial and boundary data.  Writing
\[
    E^{ab}
    :=\mathsf A_{\mathrm{dbl}}^{ab}
       -a_{c,*}^{ab}I_{\mathbb R^{2L}},
\]
the doubled equation \eqref{eq:nonlinear-doubled-problem}, written in the
terminal Fermi coordinates, and \eqref{eq:terminal-principal-smallness} yield
\begin{equation}\label{eq:terminal-cutoff-equation}
    \mathscr L_*Z=E^{ab}\partial_{ab}Z+F,
    \qquad
    \|E\|_{L^\infty(\operatorname{supp}Z)}\le C_M\delta.
\end{equation}
Here, since $\eta\chi$ is independent of $p$, the remainder is explicitly
\[
\begin{aligned}
    F
    ={}&\eta\chi\bigl(
         \mathsf B_{\mathrm{dbl}}^k(\mathbf V)\partial_k\mathbf V
         +\mathbf Q_{\mathrm{dbl}}(\mathbf V,\nabla\mathbf V)
       \bigr)\\
      &+\eta'\chi\mathbf V
       -\eta\bigl(\chi''\mathbf V
                   +2\chi'\partial_r\mathbf V\bigr).
\end{aligned}
\]
Thus \eqref{eq:terminal-collar-gradient-bound} and
\eqref{eq:terminal-lower-order-bound} give
\begin{equation}\label{eq:terminal-cutoff-remainder-bound}
    \|F\|_{L^q(D_0\times(T_*-2\tau,s))}
    \le C_{M,\delta,\tau}
\end{equation}
with a constant independent of $s$.

Applying the fixed-domain estimate to
\eqref{eq:terminal-cutoff-equation} and using
\eqref{eq:terminal-absorption-choice} gives
\[
\begin{aligned}
    \|Z\|_{W_q^{2,1}}
    &\le C_0C_M\delta\|Z\|_{W_q^{2,1}}
       +C_{M,\delta,\tau},\\
    \|Z\|_{W_q^{2,1}}
    &\le 2C_{M,\delta,\tau}.
\end{aligned}
\]
Since $Z=\mathbf V$ on
$\Gamma_*\times[-\delta,\delta]\times[T_*-\tau,s]$, this is a terminal-strip
bound for the physical collar, uniform in $s$.  The coordinate changes
between the terminal Fermi collar and the fixed domains $\Omega_*^\sigma$
have uniformly bounded derivatives, so the same bound holds for $U^\sigma$
on a fixed collar of $\Gamma_*$.

It remains to cover the regions separated from the interface.  Since $N^\sigma$
is compact and $|\nabla u^\sigma|\le M$, we have $|U^\sigma|\le C$ and
$|\nabla U^\sigma|\le CM$ on $\Omega_*^\sigma\times(T_*-2\tau,T_*)$.  Choose
finitely many coordinate patches $P_k$ and open sets $P_k'\Subset P_k$ with
$\operatorname{dist}(P_k,\Gamma_*)>\delta/2$ such that the $P_k'$ cover
$\{x\in\Omega_*^\sigma:\operatorname{dist}(x,\Gamma_*)\ge\delta\}$; together
with the collar $\{0\le r<\delta\}$ they cover $\Omega_*^\sigma$.  In each
$P_k$ the field $U^\sigma$ solves
\[
    \partial_tU^\sigma-a^{ij}\partial_{ij}U^\sigma=F^\sigma,
\]
where $a^{ij}$ are smooth, uniformly elliptic coefficients constructed from
$\Psi_{*,t}$ as in \Cref{subsec:bulk-problems}, and $F^\sigma$ collects the
first-order terms and the target term, so that $|F^\sigma|\le C(1+M^2)$.
Multiplying by a spatial cutoff supported in $P_k$ and equal to one on $P_k'$,
and by the time cutoff $\eta$, the local interior $W_q^{2,1}$ estimate bounds
$\|U^\sigma\|_{W_q^{2,1}(P_k'\times(T_*-\tau,s))}$ independently of $s$; the
commutators contain only $U^\sigma$ and $\nabla U^\sigma$, which are bounded.
No boundary data are needed at the inner edge, since these estimates are
interior.  Combining the finitely many estimates with the collar bound by a
fixed partition of unity and then letting $s\uparrow T_*$ proves
\eqref{eq:global-terminal-Wq-estimate}.
\end{proof}

\subsection{Terminal trace and restart}

\begin{lemma}[Terminal trace]
\label{lem:continuation-terminal-trace}
There exist terminal traces
\[
    U_*^\sigma
    \in W_q^{2-2/q}(\Omega_*^\sigma;N^\sigma)
\]
such that
\begin{equation}\label{eq:terminal-trace-convergence}
    U^\sigma(t)\longrightarrow U_*^\sigma
    \quad\text{in }W_q^{2-2/q}(\Omega_*^\sigma;\mathbb R^L)
    \quad\text{as }t\uparrow T_*.
\end{equation}
The corresponding physical maps $u_*^\sigma$ satisfy the value and conormal
conditions on $\Gamma_{T_*}$ in the trace sense.
\end{lemma}

\begin{proof}
The parabolic trace theorem applied to
\eqref{eq:global-terminal-Wq-estimate} gives
\[
    U^\sigma
    \in C\bigl([T_*-\tau,T_*];
       W_q^{2-2/q}(\Omega_*^\sigma;\mathbb R^L)\bigr),
\]
where the value at $T_*$ is the terminal trace.  Since $q>d+2$, this trace
space embeds into $C^{1+\alpha}$ for every
$0<\alpha<1-(d+2)/q$, so the maps and their first spatial derivatives
converge uniformly.  By compactness of $N^\pm$, closedness of $\Lambda$, and
smooth convergence of the interface normals, we can pass to the limit in the
target, value, and conormal constraints. This proves the lemma.
\end{proof}

\begin{proof}[Proof of \Cref{prop:prescribed-interface-continuation}]
Apply \Cref{thm:main-short-time-solution}, after translating time by $T_*$,
to the terminal data from \Cref{lem:continuation-terminal-trace} and the given
smooth continuation of the interface.  This produces a strong solution on
$[T_*,T_*+\varepsilon]$.  In the common fixed coordinates, the old and new
solutions have the same $W_q^{2-2/q}$ trace at $T_*$.  The standard temporal
gluing property of $W_q^{2,1}$ therefore makes their union a strong solution
on an interval containing $T_*$.  Uniqueness follows from
\Cref{prop:physical-solution-uniqueness}.

The local mixed-boundary regularity argument of
\Cref{prop:physical-smoothness}, applied with $T_*$ as an interior time,
shows that the glued solution is smooth across $T_*$ whenever the old
solution and interface are smooth.  No all-order compatibility condition at
$T_*$ is needed for the strong restart.
\end{proof}

\begin{proof}[Proof of \Cref{cor:mean-curvature-continuation}]
For a compact mean curvature flow in the flat torus, boundedness of
$\|A_{\Gamma_t}\|_{L^\infty}$ up to a finite time gives bounds for all
curvature derivatives on every terminal subinterval and a smooth extension of
the flow past that time; see
\cite{Ecker2004,HuiskenPolden1999}.  The limit hypersurface $\Gamma_{T_*}$ is
again embedded and separating.  Indeed, if two sheets of $\Gamma_{T_*}$
touched at a point, they would have a common tangent plane, since a
transversal intersection would persist for $t<T_*$.  For $t$ close to $T_*$
the two sheets would be graphs over this plane near the point, and the
difference of the graph functions would be positive for $t<T_*$, solve a linear
uniformly parabolic equation, and vanish at $t=T_*$, contradicting the strong
maximum principle.  Separation is preserved because the mod~$2$ homology class
of $\Gamma_t$ is constant along the flow.  Under
\eqref{eq:mean-curvature-continuation-bound}, the resulting extended interface
and the phase-gradient bound satisfy the hypotheses of
\Cref{prop:prescribed-interface-continuation}.  The phase fields therefore
extend strongly.  Since $T_*>0$, the preceding interior-time regularity makes
them smooth on every closed positive-time strip.  Finally, if
$T_{\max}<\infty$ and the $\limsup$ in the statement were finite, then
\eqref{eq:mean-curvature-continuation-bound} would hold with
$T_*=T_{\max}$, and the extension just obtained would contradict the
maximality of $T_{\max}$.  This proves the blow-up
alternative.
\end{proof}

\appendix
\crefalias{section}{appendix}
\section{Compatibility conditions at the initial time}
\label{subsec:compatibility-conditions}

Let
\[
    X_t:\Gamma_0\longrightarrow\Gamma_t
\]
be a fixed smooth parametrization of the prescribed moving interface.  For a
candidate solution define the pulled-back trace and conormal fields
\begin{equation}\label{eq:interface-trace-fields}
    \mathbf u^\Gamma(p,t)
    :=
    \bigl(
       u^+(X_t(p),t),
       u^-(X_t(p),t)
    \bigr),
\end{equation}
and
\begin{equation}\label{eq:interface-conormal-fields}
    \mathbf c^\Gamma(p,t)
    :=
    \bigl(
       -\partial_{\nu_t}u^+(X_t(p),t),
       \partial_{\nu_t}u^-(X_t(p),t)
    \bigr).
\end{equation}
All time derivatives below are taken at fixed $p\in\Gamma_0$.  At points of
$\Gamma_0$, time derivatives of $u^\sigma$ at $t=0$ are understood as limits
from $\overline{Q_T^\sigma}$, or formally through the phase equation.  Let
$\Pi_\Lambda$ be the nearest-point projection in a fixed tubular neighborhood
of $\Lambda$.  The trace constraint is equivalent to
\[
    \mathbf u^\Gamma-\Pi_\Lambda(\mathbf u^\Gamma)=0,
\]
and, given the trace constraint, the conormal condition is equivalent to
\[
    D\Pi_\Lambda(\mathbf u^\Gamma)\mathbf c^\Gamma=0.
\]
Accordingly, the compatibility conditions of order $m\ge0$ are
\begin{equation}\label{eq:Lambda-compatibility-order-m}
    \left.
    \frac{d^m}{dt^m}
    \left[
       \mathbf u^\Gamma-\Pi_\Lambda(\mathbf u^\Gamma)
    \right]
    \right|_{t=0}=0,
\end{equation}
and
\begin{equation}\label{eq:conormal-compatibility-order-m}
    \left.
    \frac{d^m}{dt^m}
    \left[
       D\Pi_\Lambda(\mathbf u^\Gamma)\mathbf c^\Gamma
    \right]
    \right|_{t=0}=0.
\end{equation}
For $m=0$, these are precisely the two conditions in
\eqref{eq:zeroth-compatibility}.

To make the higher-order conditions intrinsic to the initial data, define the
interface velocity field by
\[
    V_X(X_t(p),t):=\partial_tX_t(p).
\]
For every bulk field $G$,
\begin{equation}\label{eq:moving-trace-time-derivative}
    \frac{d}{dt}G(X_t(p),t)
    =
    \bigl(\partial_tG+V_X\cdot\nabla G\bigr)(X_t(p),t).
\end{equation}
Using the fixed ambient extension $\widetilde{\mathrm{II}}^\sigma$ specified in
\eqref{eq:target-nonlinearity-extension}, write the phase equation as
\[
    \partial_tu^\sigma
    =\Delta u^\sigma
     +\widetilde{\mathrm{II}}^\sigma(u^\sigma)
      (\nabla u^\sigma,\nabla u^\sigma).
\]
Then the first trace jet is
\begin{equation}\label{eq:first-interface-trace-jet}
\begin{aligned}
    \mathbf u_1^\Gamma
    :=
    \left.\frac{d}{dt}\mathbf u^\Gamma\right|_{t=0}
    =
    \Bigl(&
      \Delta u_0^+
      +\mathrm{II}^+(u_0^+)(\nabla u_0^+,\nabla u_0^+)
      +V_X(\cdot,0)\cdot\nabla u_0^+,\\
      &
      \Delta u_0^-
      +\mathrm{II}^-(u_0^-)(\nabla u_0^-,\nabla u_0^-)
      +V_X(\cdot,0)\cdot\nabla u_0^-
    \Bigr)\Big|_{\Gamma_0}.
\end{aligned}
\end{equation}
The first differentiated trace constraint therefore reads
\[
    \bigl(I-D\Pi_\Lambda(\mathbf u_0^\Gamma)\bigr)
    \mathbf u_1^\Gamma=0,
\]
which says that the first moving trace jet is tangent to $\Lambda$.

The derivative of the conormal field must include both the moving evaluation
point and the material derivative of the normal.  Set
\[
    \dot\nu_X(p,t)
    :=\frac{d}{dt}\bigl[\nu_t(X_t(p))\bigr].
\]
Then
\begin{align*}
    \frac{d}{dt}
    \bigl[\partial_{\nu_t}u^\sigma(X_t(p),t)\bigr]
    ={}&
    \partial_{\nu_t}\partial_tu^\sigma
    +\partial_{\dot\nu_X}u^\sigma
    +D^2u^\sigma[V_X,\nu_t]
    \quad\text{at }X_t(p).
\end{align*}
After substituting the phase equation, this expression depends only on the
initial map, its spatial derivatives, and the prescribed jets of $X_t$ and
$\nu_t$.  Repeated differentiation gives all higher jets recursively.
Both \eqref{eq:Lambda-compatibility-order-m} and
\eqref{eq:conormal-compatibility-order-m} are therefore conditions entirely on the
initial data and the prescribed moving geometry; they do not assume that a
solution has already been constructed.

These conditions do not depend on the tangential parametrization of the
interface.  Indeed, another parametrization has the form
$\widetilde X_t=X_t\circ\phi_t$ for a smooth family of diffeomorphisms of
$\Gamma_0$.  The two constraint fields are then composed with $\phi_t$.
If all time derivatives up to order $m$ of a constraint field vanish at
$t=0$, the chain rule shows inductively that the same is true after this
composition.

The $W_q^{2,1}$ existence construction requires only the conditions of
order $m=0$.  Indeed, the Dirichlet trace space $\mathcal Y_T$ has time
regularity $1-1/(2q)$, so the time derivative of the trace has no trace at
$t=0$, whereas the Neumann-type data space
$W_q^{1/2-1/(2q)}(0,T;L^q)\cap L^q(0,T;W_q^{1-1/q})$ has a time trace at
$t=0$ precisely because $q>3$.  In the construction, the zeroth-order
conormal condition enters through \eqref{eq:doubled-initial-Wq-trace}.
Compatibility conditions of every order are used only in the final bootstrap
to obtain smoothness up to the initial corner $\Gamma_0\times\{0\}$.

\section{Uniform short-time parabolic estimates}
\label{app:uniform-estimates}

\subsection{Function spaces on compact manifolds with boundary}
\label{subsec:function-spaces}

Let $D$ be a compact $C^2$ manifold of dimension $d$ with smooth boundary,
let $S$ be either a
boundary component of $\partial D$ or a compact $C^2$ hypersurface contained
in the interior, and let $E$ be a finite-dimensional Euclidean space.  For
$0<T\le T_0$, write
\[
    D_T:=D\times(0,T),
    \qquad
    S_T:=S\times(0,T).
\]
All function spaces below are $E$-valued componentwise. Since $E$ is
finite-dimensional, different Euclidean norms give equivalent function-space
norms. For noninteger $s>0$, $W_q^s$ denotes the Sobolev--Slobodeckii space.
More precisely, write $s=k+\theta$ with $k\in\mathbb N_0$ and
$0<\theta<1$, and set $W_q^0=L^q$.  On a coordinate domain
$U\subset\mathbb R^m$, where
$m=d$ or $m=d-1$, we use the norm
\begin{equation}\label{eq:Slobodeckii-spatial-norm}
\begin{aligned}
    \|v\|_{W_q^s(U;E)}
    :=\ &\|v\|_{W_q^k(U;E)}\\
    &+\left(
       \sum_{|\gamma|=k}
       \int_U\!\int_U
       \frac{|\partial^\gamma v(x)-\partial^\gamma v(y)|_E^q}
            {|x-y|^{m+\theta q}}\,dx\,dy
       \right)^{1/q}.
\end{aligned}
\end{equation}
For a Banach space $X$, the time-fractional space appearing below is defined,
for $0<\theta<1$, by
\begin{equation}\label{eq:Slobodeckii-time-norm}
\begin{aligned}
    \|h\|_{W_q^\theta(0,T;X)}
    :=\ &\|h\|_{L^q(0,T;X)}\\
    &+\left(
       \int_0^T\!\int_0^T
       \frac{\|h(t)-h(s)\|_X^q}{|t-s|^{1+\theta q}}\,dt\,ds
       \right)^{1/q}.
\end{aligned}
\end{equation}
Sobolev and Bochner-valued spaces on compact manifolds are defined through a
fixed finite smooth atlas and partition of unity; different choices give
equivalent norms.

The parabolic Sobolev space is
\begin{equation}\label{eq:parabolic-Sobolev-space}
    W_q^{2,1}(D_T;E)
    :=
    H_q^1(0,T;L^q(D;E))
    \cap
    L^q(0,T;W_q^2(D;E)),
\end{equation}

The Dirichlet trace space on $S_T$ is
\begin{equation}\label{eq:parabolic-trace-space}
\begin{aligned}
    \mathcal Y_T(S;E)
    :=
    &\ W_q^{1-\frac1{2q}}(0,T;L^q(S;E))\\
    &\cap L^q(0,T;W_q^{2-\frac1q}(S;E)).
\end{aligned}
\end{equation}
The initial trace space is
\[
    (L^q(D;E),W_q^2(D;E))_{1-1/q,q}
    =:W_q^{2-2/q}(D;E).
\]
In particular, the spatial trace theorem gives
\[
    u_0|_S\in W_q^{2-3/q}(S;E),
    \qquad
    \partial_\nu u_0|_S\in W_q^{1-3/q}(S;E)
    \quad\text{for }u_0\in W_q^{2-2/q}(D;E),
\]
where $\nu$ is any locally defined $C^1$ unit normal along $S$.
For the range $q>d+2$ considered below, the exponent $2-2/q$ is noninteger,
so the fractional-space convention above applies.
Moreover, $\dim S=d-1$ and $2-\frac1q-\frac{d+1}q=2-\frac{d+2}q>1+\beta$ for
$0<\beta<1-\frac{d+2}q$, so the parabolic Sobolev embedding gives, for every
fixed $T$,
\[
    \mathcal Y_T(S;E)\hookrightarrow C^{1+\beta,(1+\beta)/2}(S\times[0,T];E).
\]
Throughout the appendix, constants may depend on the fixed manifolds,
$T_0$, $q$, and the coefficient bounds, but never on
$T\in(0,T_0]$.

\subsection{Uniform Dirichlet maximal regularity}

\begin{lemma}[Uniform $W_q^{2,1}$ estimate for Dirichlet systems]
\label{lem:uniform-Wq-dirichlet}
Let $D$ be a fixed compact $C^2$ manifold with smooth boundary, let $E$ be
finite-dimensional, and let $q>d+2$.  In local coordinates consider
\[
    \mathscr L u
    :=\partial_tu
      -A^{ij}(x,t)\partial_{ij}u
      -A^i(x,t)\partial_iu
      -A^0(x,t)u.
\]
Assume that the leading coefficients are uniformly Lipschitz in $x$ and
uniformly $1/2$-H\"older in $t$, and that $A^i,A^0$ are uniformly bounded.
Write
\[
    \mathscr A(x,t,\xi):=A^{ij}(x,t)\xi_i\xi_j.
\]
Assume that for some $\theta\in(0,\pi/2)$ and $C_{\mathrm{pe}}>0$,
\begin{equation}\label{eq:uniform-parameter-ellipticity}
    \|\bigl(\lambda I+\mathscr A(x,t,\xi)\bigr)^{-1}\|
    \le\frac{C_{\mathrm{pe}}}{|\lambda|+|\xi|^2}
\end{equation}
for $\lambda\in\overline{\mathfrak S}_{\pi-\theta}$ and
$(\lambda,\xi)\ne(0,0)$, and assume the uniform Dirichlet
Lopatinskii--Shapiro condition.

For every $0<T\le T_0$, every
\[
    f\in L^q(D_T;E),
    \qquad
    u_0\in W_q^{2-2/q}(D;E),
    \qquad
    g\in\mathcal Y_T(\partial D;E),
\]
with $g(\cdot,0)=u_0|_{\partial D}$, the problem
\begin{equation}\label{eq:linear-Dirichlet-problem}
\begin{cases}
    \mathscr L u=f,&D_T,\\
    u=g,&\partial D\times(0,T),\\
    u(\cdot,0)=u_0,&D,
\end{cases}
\end{equation}
has a unique solution in $W_q^{2,1}(D_T;E)$, and
\begin{equation}\label{eq:uniform-Wq-estimate}
\begin{aligned}
    \|u\|_{W_q^{2,1}(D_T;E)}
    \le C\Bigl(
        &\|f\|_{L^q(D_T;E)}
        +\|u_0\|_{W_q^{2-2/q}(D;E)}\\
        &+\|g\|_{\mathcal Y_T(\partial D;E)}
    \Bigr),
\end{aligned}
\end{equation}
where $C$ depends only on the fixed geometry and the uniform structural
bounds, not on $T$.
\end{lemma}

\begin{proof}
After boundary flattening and localization, the assertion on a fixed time
interval is the inhomogeneous-boundary maximal-regularity theorem of
Denk--Hieber--Pr\"uss \cite[Theorem~2.1]{DenkHieberPruss2007}, specialized to
second order and equal space and time exponents.  In finite target dimension,
\eqref{eq:uniform-parameter-ellipticity} and the Lopatinskii--Shapiro condition are
the relevant uniform parameter-elliptic hypotheses.  A finite atlas and the
usual commutator localization give the estimate on the compact manifold.

To obtain a constant independent of $T$, extend the coefficients to the fixed
interval $(0,T_0)$. With respect to a fixed background connection on $D$,
extend the principal coefficients constantly beyond $t=T$ from their terminal
values, and extend the bounded lower-order remainder by zero. Temporal
$1/2$-H\"older continuity ensures that this extension preserves the
coefficient bounds, parameter ellipticity, and the Dirichlet
Lopatinskii--Shapiro condition uniformly in $T$.
Choose a fixed lifting $u^0$ of the initial value on
$D\times(0,T_0)$ with
\[
    u^0(\cdot,0)=u_0,
    \qquad
    \|u^0\|_{W_q^{2,1}(D\times(0,T_0))}
    \le C\|u_0\|_{W_q^{2-2/q}(D)}.
\]
Then
$g_1:=g-u^0|_{\partial D\times(0,T)}$ has zero initial trace.  The
zero-initial extension theorem
\cite[Lemma~2.5]{MeyriesSchnaubelt2012}, with $\mu=1$, gives an extension
$\widetilde g_1\in\mathcal Y_{T_0}$ satisfying
\[
    \|\widetilde g_1\|_{\mathcal Y_{T_0}}
    \le C\|g_1\|_{\mathcal Y_T},
\]
with $C$ independent of $T$.  Extend $f$ by zero and use
$u^0|_{\partial D}+\widetilde g_1$ as boundary datum on the fixed interval.
Solving there and restricting to $(0,T)$ gives
\eqref{eq:uniform-Wq-estimate}; uniqueness follows by the same fixed-cylinder
theorem.
\end{proof}

\subsection{Embeddings, traces, and short-time smallness}

\begin{lemma}[Uniform short-time embeddings and traces]
\label{lem:uniform-short-time-estimates}
Let $D$ be a fixed compact $C^2$ manifold with boundary, let $q>d+2$, and fix
\[
    0<\alpha<1-\frac{d+2}{q}.
\]
Let $S$ be a boundary component or a compact interior hypersurface.  There
exists $C$, independent of $T\in(0,T_0]$, such that:

\begin{enumerate}[label=\emph{(\roman*)},leftmargin=2.3em]
\item If $u\in W_q^{2,1}(D_T;E)$ and $u(\cdot,0)=0$, then
\begin{equation}\label{eq:zero-initial-holder-embedding}
    \|u\|_{C^{1+\alpha,(1+\alpha)/2}(\overline D\times[0,T])}
    \le C\|u\|_{W_q^{2,1}(D_T;E)}.
\end{equation}

\item If $u\in W_q^{2,1}(D_T;E)$ and $u(\cdot,0)=0$, then
\begin{align}
    \|u\|_{L^q(D_T)}
    &\le T\|u\|_{W_q^{2,1}(D_T)},
    \label{eq:zero-initial-Lq}\\
    \|\nabla u\|_{L^q(D_T)}
    &\le CT^{1/2}\|u\|_{W_q^{2,1}(D_T)},
    \label{eq:zero-initial-grad-Lq}\\
    \|u\|_{L^\infty(D_T)}
    &\le CT^{(1+\alpha)/2}\|u\|_{W_q^{2,1}(D_T)},
    \label{eq:zero-initial-Linfty}\\
    \|\nabla u\|_{L^\infty(D_T)}
    &\le CT^{\alpha/2}\|u\|_{W_q^{2,1}(D_T)}.
    \label{eq:zero-initial-grad-Linfty}
\end{align}

\item If $u\in W_q^{2,1}(D_T;E)$, then its boundary or interior hypersurface
trace satisfies
\begin{equation}\label{eq:uniform-hypersurface-trace}
    \|u|_{S_T}\|_{\mathcal Y_T(S;E)}
    \le C\|u\|_{W_q^{2,1}(D_T;E)}.
\end{equation}

\item If $h\in\mathcal Y_T(S;E)$ and $h(\cdot,0)=0$, then
\begin{align}
    \|h\|_{C^0(S\times[0,T])}
    &\le CT^{(1+\alpha)/2}\|h\|_{\mathcal Y_T(S;E)},
    \label{eq:trace-C0-smallness}\\
    \|\nabla_Sh\|_{C^0(S\times[0,T])}
    &\le CT^{\alpha/2}\|h\|_{\mathcal Y_T(S;E)}.
    \label{eq:trace-gradient-smallness}
\end{align}

\item If $S$ is a boundary component and $h$ has zero initial trace, then
there exists a lifting $H\in W_q^{2,1}(D_T;E)$ such that
\[
    H=h\text{ on }S_T,
    \qquad
    H=0\text{ on }(\partial D\setminus S)_T,
    \qquad
    H(\cdot,0)=0,
\]
and
\begin{equation}\label{eq:uniform-boundary-lifting}
    \|H\|_{W_q^{2,1}(D_T;E)}
    \le C\|h\|_{\mathcal Y_T(S;E)}.
\end{equation}

\item The time trace
\[
    \operatorname{tr}_{t=0}:\mathcal Y_{T_0}(S;E)
    \longrightarrow W_q^{2-3/q}(S;E)
\]
is onto and has a bounded linear right inverse $\mathfrak L_S$.  Thus
\begin{equation}\label{eq:mixed-trace-initial-lifting}
    (\mathfrak L_Sh_0)(\cdot,0)=h_0,
    \qquad
    \|\mathfrak L_Sh_0\|_{\mathcal Y_{T_0}(S;E)}
    \le C\|h_0\|_{W_q^{2-3/q}(S;E)}.
\end{equation}

\item Let
\[
    0<\alpha<\beta<1-\frac{d+2}{q}.
\]
If $u\in W_q^{2,1}(D_T;E)$ has initial trace
$u_0\in W_q^{2-2/q}(D;E)$, then there are $C>0$ and $\eta>0$,
independent of $T\in(0,T_0]$, such that
\begin{equation}\label{eq:initial-trace-holder-smallness}
\begin{aligned}
    \|u-u_0\|_{C^{1+\alpha,(1+\alpha)/2}
       (\overline D\times[0,T];E)}
    \le CT^\eta
    \bigl(
       \|u\|_{W_q^{2,1}(D_T;E)}
       +\|u_0\|_{W_q^{2-2/q}(D;E)}
    \bigr),
\end{aligned}
\end{equation}
where $u_0$ on the left is regarded as time independent.
\end{enumerate}
\end{lemma}

\begin{proof}
The zero-initial extension theorem of Meyries--Schnaubelt
\cite[Lemma~2.5]{MeyriesSchnaubelt2012}, with $\mu=1$, gives extensions from
$(0,T)$ to $(0,T_0)$ in both $W_q^{2,1}$ and $\mathcal Y_T$, with operator
norms independent of $T$; the same temporal operator acts on both factors
of each intersection.  The parabolic Sobolev--Morrey embedding on the fixed
cylinder then proves \emph{(i)}, since $1+\alpha<2-(d+2)/q$.

For \emph{(ii)}, time integration gives
$\|u\|_{L^q(D_T)}\le T\|\partial_tu\|_{L^q(D_T)}$.
Spatial interpolation consequently yields
\[
    \|\nabla u\|_{L^q(D_T)}
    \le C(T^{1/2}+T)\|u\|_{W_q^{2,1}(D_T)}
    \le CT^{1/2}\|u\|_{W_q^{2,1}(D_T)}.
\]
The remaining two bounds follow from the temporal H\"older seminorms in
\emph{(i)}, since $u(\cdot,0)=\nabla u(\cdot,0)=0$.

Part \emph{(iii)} follows from Weidemaier's uniform trace theorem
\cite[Theorem~1 and Remark~2]{Weidemaier1994}, with $m=0$ and equal space
and time exponents.  A fixed finite spatial localization gives the estimate
for both boundary and interior hypersurfaces, without any initial trace
condition.

For \emph{(v)}, extend $h$ by zero to the other boundary components and
apply \Cref{lem:uniform-Wq-dirichlet} componentwise to a fixed heat operator
with zero forcing and zero initial value.  To obtain \emph{(iv)}, apply
this lifting on the auxiliary cylinder $S\times(0,1)$, prescribing $h$ on
$S\times\{0\}$ and zero on $S\times\{1\}$, and restrict the last two
estimates in \emph{(ii)} to $S\times\{0\}$.  This also covers an interior
hypersurface $S$.

For \emph{(vi)}, apply the temporal trace and right-inverse results of
Denk--Saal--Seiler \cite[Lemma~4.4 and Theorem~4.5]{DenkSaalSeiler2008},
with integrability exponent $q$, zero exponential weight, and time--space
regularity pairs $(0,2-1/q)$ and $(1-1/(2q),0)$.
The resulting initial trace has order $(2-1/q)-2/q=2-3/q$.
A fixed partition of unity transfers the trace and its right inverse to $S$.

Finally, choose a bounded lifting $U^0$ of $u_0$ to
$W_q^{2,1}(D\times(0,T_0);E)$, as in
\cite[Lemma~4.3]{MeyriesSchnaubelt2012}, with $\mu=1$ and $m=1$.
Apply \emph{(i)} with exponent $\beta$ to $u-U^0$ and use the
fixed-cylinder embedding for $U^0$, together with
$W_q^{2-2/q}(D;E)\hookrightarrow C^{1+\beta}(\overline D;E)$, to obtain
\[
    \|u-u_0\|_{C^{1+\beta,(1+\beta)/2}}
    \le C\bigl(\|u\|_{W_q^{2,1}}+\|u_0\|_{W_q^{2-2/q}}\bigr).
\]
Set $w:=u-u_0$.  Since $w(\cdot,0)=\nabla w(\cdot,0)=0$, its temporal
H\"older seminorms give $C^0$ and gradient bounds with factors
$T^{(1+\beta)/2}$ and $T^{\beta/2}$, respectively.  H\"older interpolation
therefore gives
\[
    \|w\|_{C^{1+\alpha,(1+\alpha)/2}}
    \le CT^{(\beta-\alpha)/2}\|w\|_{C^{1+\beta,(1+\beta)/2}},
\]
which proves \emph{(vii)}.
\end{proof}

\subsection{Nemytskii estimate}

\begin{lemma}[Uniform smooth composition on a fixed-initial affine class]
\label{lem:nemytskii-trace}
Let $U\subset\mathbb R^m$ be open, $K\Subset U$, and
$G\in C^\infty(U;\mathbb R^\ell)$.  Fix
$h_*\in\mathcal Y_{T_0}(S;\mathbb R^m)$.  For every $M>0$
there is $C$, independent of $T\in(0,T_0]$, such that
\begin{equation}\label{eq:nemytskii-trace-estimate}
    \|G(h_1)-G(h_2)\|_{\mathcal Y_T(S;\mathbb R^\ell)}
    \le C
    \|h_1-h_2\|_{\mathcal Y_T(S;\mathbb R^m)}
\end{equation}
whenever
\[
    h_j(S\times[0,T])\subset K,
    \qquad
    h_j(\cdot,0)=h_*(\cdot,0),
    \qquad
    \|h_j-h_*|_{(0,T)}\|_{\mathcal Y_T}\le M.
\]
\end{lemma}

\begin{proof}
Let $\mathcal E_T^0$ be a uniformly bounded extension from the zero-initial
subspace of $\mathcal Y_T$ to the corresponding subspace of
$\mathcal Y_{T_0}$.  Set
\[
    H_j:=h_*+\mathcal E_T^0(h_j-h_*|_{(0,T)}).
\]
The fields $H_j$ lie in a fixed bounded subset of $\mathcal Y_{T_0}$, hence, by
the embedding recorded in \Cref{subsec:function-spaces}, in a fixed bounded
subset of $C^{1+\beta,(1+\beta)/2}$.  Choose
$\widetilde G\in C_c^\infty(\mathbb R^m;\mathbb R^\ell)$ that agrees with
$G$ on a neighborhood of $K$.  On the fixed cylinder $S\times(0,T_0)$ we claim
\[
    \|\widetilde G(H_1)-\widetilde G(H_2)\|_{\mathcal Y_{T_0}}
    \le C\|H_1-H_2\|_{\mathcal Y_{T_0}}.
\]
To see this, write $\widetilde G(H_1)-\widetilde G(H_2)=g\,(H_1-H_2)$ with
$g:=\int_0^1\widetilde G'(\tau H_1+(1-\tau)H_2)\,d\tau$.  By the bounds above,
$g$ and $\nabla g$ are bounded, and
$\|H_1-H_2\|_{L^\infty}+\|\nabla(H_1-H_2)\|_{L^\infty}\le
C\|H_1-H_2\|_{\mathcal Y_{T_0}}$.  For the temporal part of the
$\mathcal Y_{T_0}$-norm we use
\[
    (g(H_1-H_2))(t)-(g(H_1-H_2))(s)
    =g(t)\bigl[(H_1-H_2)(t)-(H_1-H_2)(s)\bigr]
     +\bigl[g(t)-g(s)\bigr](H_1-H_2)(s)
\]
together with
$|g(t)-g(s)|\le C\bigl(|H_1(t)-H_1(s)|+|H_2(t)-H_2(s)|\bigr)$.  For the
spatial part we use
$\nabla\bigl(g(H_1-H_2)\bigr)=g\,\nabla(H_1-H_2)+(\nabla g)(H_1-H_2)$ and the
product estimate
\[
    \|fh\|_{W_q^{\sigma}}
    \le C\bigl(\|f\|_{L^\infty}\|h\|_{W_q^{\sigma}}
      +\|h\|_{L^\infty}\|f\|_{W_q^{\sigma}}\bigr)
\]
in $W_q^{\sigma}(S)\cap L^\infty(S)$, $\sigma=1-1/q$, see e.g.\
\cite{RunstSickel1996}, together with
$\|\widetilde G^{(k)}(H)\|_{W_q^{\sigma}}\le C(1+\|H\|_{W_q^{\sigma}})$.
Restricting to $(0,T)$ and using the uniform extension bound proves
\eqref{eq:nemytskii-trace-estimate}.
\end{proof}

\bigskip
\noindent\textbf{Data availability.} Data sharing is not applicable to this article as no datasets were generated or analysed.

\medskip
\noindent\textbf{Conflict of interest.} The author declares that there is no conflict of interest.


\end{document}